\documentclass[10pt]{article}
\usepackage[
    colorlinks,%
    linkcolor=blue,citecolor=red,urlcolor=blue,
]{hyperref}
\usepackage{cite,amsfonts,amssymb,amsmath,exscale,bbm}
\usepackage{graphicx}

\usepackage{footnpag} 
\usepackage[all,poly]{xy}  

\newcommand\beq{\begin{equation}}
\newcommand\eeq{\end{equation}}
\newcommand\be{\begin{equation}}
\newcommand\ee{\end{equation}}
\newcommand\beqa{\begin{eqnarray}}
\newcommand\eeqa{\end{eqnarray}}
\newcommand\bean{\begin{eqnarray*}}
\newcommand\eean{\end{eqnarray*}}

\newcommand\U{{\mathrm U}}
\newcommand\C{{\mathbb C}}
\newcommand\N{{\mathbb N}}
\newcommand\R{{\mathbb R}}
\newcommand\Z{{\mathbb Z}}
\newcommand\Q{{\mathbb Q}}
\renewcommand{\H}{{\mathcal{H}}}
\newcommand\J{{\mathcal{J}}}
\newcommand\K{{\mathcal{K}}}
\newcommand\G{{\mathcal{G}}} 
\newcommand\RR{{\mathcal{R}}}
\newcommand\M{{\mathcal{M}}} 
\newcommand{\HH}{{\sf H}} 
\newcommand{\VV}{{\sf V}} 

\newcommand{\x}{\textit{\textbf{x}}}
\newcommand{\p}{\textit{\textbf{p}}}

\newcommand\unit{\mathbbm{1}}

\newcommand{\SO}{{\rm SO}}

\newcommand{\Hhat}{\widehat H}

\newcommand\Vect{{\rm{Vect}}}
\newcommand\Hilb{{\rm Hilb}}
\newcommand\tc{{\mathcal{C}}}  
\newcommand\Cat{{\bf{Cat}}} 
\newcommand\twoVe{{\bf{2Vect}}}
\newcommand\twoHi{{\bf{2Hilb}}}

\newcommand\Rep{{\bf{2Rep}}}
\newcommand\me{{\bf{Meas}}}
\newcommand\Mat{{\rm Mat}}
\newcommand\Ab{{\rm Ab}}

\newcommand\tensor{\otimes}
\newcommand\To{\Rightarrow}

\newcommand{\im}{\mathrm{im}}

\newcommand\direct{\int^{\oplus}}
\newcommand\extd{\mathrm {d}}

\newcommand{\gmean}[2]{\sqrt{{\extd #1}{\extd #2}}} 
\newcommand{\rnd}[2]{{\frac{\extd #1}{\extd #2}}}
\newcommand{\sqrnd}[2]{\sqrt{\frac{\extd #1}{\extd #2}}}
\newcommand{\alme}{\textit{a.e.}}

\newcounter{letter} \newcounter{numeral} \newcounter{Numeral}

\newenvironment{romanlist}{
\begin{list}{(\roman{numeral})}{\usecounter{numeral}}
}{\end{list}}

\newcommand\arr{\longrightarrow}
\renewcommand\d{\partial}

\newcommand\maps{\colon}

\newtheorem{theo}{Theorem}
\newtheorem{lemma}[theo]{Lemma}
\newtheorem{prop}[theo]{Proposition}
\newtheorem{defn}[theo]{Definition}
\newtheorem{cor}[theo]{Corollary}
\newenvironment{proof}{\noindent
\textbf{Proof: }}{\hfill\rule{.6em}{.8em} \medskip}
\newenvironment{proof.within.proof}
{\noindent{\it Proof:}}{
\hfill $\Box$ \medskip}

\begin{document}
%

\centerline{\Large \bf Infinite-Dimensional Representations of 2-Groups}

\bigskip

\centerline{\large John C.\ Baez$^1$, Aristide Baratin$^{2}$,
Laurent Freidel$^{3,4}$, Derek K.\ Wise$^{5}$ }

\bigskip \bigskip

{\small
\centerline{${}^1$ Department of Mathematics, University of California}
\centerline{Riverside, CA 92521, USA}
\vskip .3em
\centerline{${}^2$ Max Planck Institute for Gravitational
Physics, Albert Einstein Institute,} \centerline{Am M\"uhlenberg 1,
14467 Golm, Germany}
\vskip .3em
\centerline{${}^3$ Laboratoire de Physique, \'Ecole
Normale Sup{\'e}rieure de Lyon} \centerline{46 All\'ee d'Italie, 69364
Lyon Cedex 07, France}
\vskip .3em
\centerline{${}^4$ Perimeter Institute for Theoretical Physics}
\centerline{Waterloo ON, N2L 2Y5, Canada}
\vskip .3em
\centerline{${}^5$ Institute for Theoretical Physics III, University of Erlangen--N\"urnberg}
\centerline{Staudtstra{\ss}e 7 / B2, 91058 Erlangen, Germany}
}

\bigskip \bigskip


\begin{abstract}
A `2-group' is a category equipped with a multiplication satisfying
laws like those of a group.  Just as groups have representations on
vector spaces, 2-groups have representations on `2-vector spaces',
which are categories analogous to vector spaces.  Unfortunately, Lie
2-groups typically have few representations on the finite-dimensional
2-vector spaces introduced by Kapranov and Voevodsky.  For this
reason, Crane, Sheppeard and Yetter introduced certain
infinite-dimensional 2-vector spaces called `measurable categories'
(since they are closely related to measurable fields of Hilbert
spaces), and used these to study infinite-dimensional representations
of certain Lie 2-groups.  Here we continue this work.  We begin with a
detailed study of measurable categories.  Then we give a geometrical
description of the measurable representations, intertwiners and
2-intertwiners for any skeletal measurable 2-group.  We study tensor
products and direct sums for representations, and various concepts of
subrepresentation.  We describe direct sums of intertwiners, and
sub-intertwiners---features not seen in ordinary group representation
theory.  We study irreducible and indecomposable representations
and intertwiners.  We also study `irretractable'
representations---another feature not seen in ordinary group
representation theory.  Finally, we argue that measurable categories
equipped with some extra structure deserve to be considered `separable
2-Hilbert spaces', and compare this idea to a tentative definition of
2-Hilbert spaces as representation categories of commutative von
Neumann algebras.
\end{abstract}


\newpage

\tableofcontents

\newpage

%
\section{Introduction}
%

The goal of `categorification' is to develop a richer version of
existing mathematics by replacing sets with categories.  This lets
us exploit the following analogy:

\vskip 2em
\begin{center}
{\small
\begin{tabular}{c|c}                    \hline
set theory             &   category theory     
  \rule[1.4em]{0em}{0em} \rule[-.8em]{0em}{0em}   \\      \hline
elements    &  objects                    
\rule[1.4em]{0em}{0em} \rule[-.8em]{0em}{0em}   \\     
equations   &  isomorphisms                \\     
between elements   &  between objects   \rule[-.8em]{0em}{0em}     \\    
sets        &  categories                    \rule[-.8em]{0em}{0em} \\     
functions   &  functors                     \rule[-.8em]{0em}{0em} \\     
equations   &  natural isomorphisms    \\     
between functions   & between functors \rule[-.8em]{0em}{0em}      \\     \hline
\end{tabular}} \vskip 1em
\end{center}
\vskip 0.5em

\noindent Just as sets have elements, categories have objects.  Just
as there are functions between sets, there are functors between
categories.  The correct analogue of an equation between elements is not
an equation between objects, but an isomorphism.  More generally, the
analog of an equation between functions is a natural isomorphism
between functors.

The word `categorification' was first coined by Louis Crane
\cite{Crane} in the context of mathematical physics.  Applications to
this subject have always been among the most exciting
\cite{BaezLauda2}, since categorification holds the promise of
generalizing some of the special features of low-dimensional physics
to higher dimensions.  The reason is that categorification
\textit{boosts the dimension by one}.

To see this in the simplest possible way, note that we can draw
sets as 0-dimensional dots and functions between sets as 1-dimensional
arrows:
$$
  \xymatrix{
  S \bullet \ar@/^2ex/[rr]^{f}="g1" &&\bullet S'
}
$$
If we could draw all the sets in the world this way, and all the functions
between them, we would have a picture of the category of all sets.

But there are many categories beside the category of sets, and when we
study categories {\it en masse} we see an additional layer of structure.  We
can draw categories as dots, and functors between categories as
arrows.  But what about natural isomorphisms between functors, or more
general natural transformations between functors?  We can draw these
as 2-dimensional surfaces:
$$
  \xymatrix{
  C \bullet \ar@/^2ex/[rr]^{f}="g1"\ar@/_2ex/[rr]_{f'}="g2"&&\bullet C'
  \ar@{=>}^{h} "g1"+<0ex,-2.5ex>;"g2"+<0ex,2.5ex>
}
$$
So, the dimension of our picture has been boosted by one!  Instead of
merely a category of all categories, we say we have a `2-category'.  
If we could draw all the categories in the world this way, and all functors
between them, and all natural transformations between those, we
would have a picture of the 2-category of all categories.  

This story continues indefinitely to higher and higher dimensions:
categorification is a process than can be iterated.  But our goal here
lies in a different direction: we wish to take a specific branch of
mathematics, the theory of infinite-dimensional group representations, 
and categorify that just once.  This might seem like a
purely formal exercise, but we shall see otherwise.  In fact, the
resulting theory has fascinating relations both to well-known topics
within mathematics (fields of Hilbert spaces and Mackey's theory of
induced group representations) and to interesting ideas in physics (spin
foam models of quantum gravity, most notably the Crane--Sheppeard model).

\subsection{2-Groups}

To categorify group represenation theory, we must first
choose a way to categorify the basic notions
involved: the notions of `group' and `vector space'.  At present,
categorifying mathematical definitions is not a completely straightforward
exercise: it requires a bit of creativity and good taste.  
So, there is work to be done here. 

By now, however, there is a fairly uncontroversial way to categorify
the concept of `group'.  The resulting notion of `2-group' can be
defined in various equivalent ways \cite{BaezLauda}.  For example, we
can think of a 2-group as a category equipped with a multiplication
satisfying the usual axioms for a group.  Since categorification
involves replacing equations by natural isomorphisms, we should demand
that the group axioms hold {\it up to natural isomorphism}.  Then we
should demand that these isomorphisms obey some laws of their own,
called `coherence laws'.  This is where the creativity comes into
play.  Luckily, everyone agrees on the correct coherence laws for
2-groups.

However, to simplify our task in this paper, we shall only consider 
`strict' 2-groups, where the axioms for a group hold as {\it equations}---not
just up to natural isomorphisms.  This lets us ignore the issue of 
coherence laws.  Another advantage of strict 2-groups is that they are 
essentially the same as `crossed modules' \cite{ForresterBarker}, which
are structures already familiar in algebra.  So, henceforth we 
shall always use the term `2-group' to mean a 2-group of this kind.  

Suppose $\G$ is a 2-group of this kind.  Since $\G$ is a category, it
has objects and morphisms.  The objects form a group under
multiplication, so we can use them to describe symmetries.  The new
feature, where we go beyond traditional group theory, is the
morphisms.  For most of our more substantial results, we shall make a
drastic simplifying assumption: we shall assume $\G$ is not only
strict but also `skeletal'.  This means that there only exists a
morphism from one object of $\G$ to another if these objects are
actually equal.  In other words, all the morphisms between object of
$\G$ are actually automorphisms.  Since the objects of $\G$ describe
symmetries, their automorphisms describe {\em symmetries of
symmetries}.

The reader should not be fooled by the somewhat intimidating
language.  A skeletal 2-group is really a very simple thing.  Using
the theory of crossed modules, explained in Section \ref{crossmod},
we shall see that a skeletal 2-group $\G$ consists of:
\begin{itemize}
\item
a group $G$ (the group of objects of $\G$),
\item an abelian group $H$ (the group of automorphisms of any object),
\item a left action $\rhd$ of $G$ as automorphisms of $H$.
\end{itemize}

A nice example is the `Poincar\'e 2-group', first discovered by one of
the authors \cite{Baez1}.  But to understand this, and to prepare
ourselves for the discussion of physics applications later in this
introduction, let us first recall the ordinary Poincar\'e group.

In special relativity, we think of a point $\x = (t,x,y,z)$ in
$\R^4$ as describing the time and location of an event.  We equip
$\R^4$ with a bilinear form, the so-called `Minkowski metric':
\[         \x \cdot \x' = tt' - xx' - yy' - zz' \]
which serves as substitute for the usual dot product on $\R^3$.  
With this extra structure, $\R^4$ is called `Minkowski
spacetime'.  The group of all linear transformations
\[           T \maps \R^4 \to \R^4 \]
preserving the Minkowski metric is called $\mathrm{O}(3,1)$. The
connected component of the identity in this group is called
$\SO_0(3,1)$.  This smaller group is generated by rotations in space
together with transformations that mix time and space coordinates.
Elements of $\SO_0(3,1)$ are called `Lorentz transformations'.
In special relativity, we think of Lorentz transformations as
symmetries of spacetime.  However, we also want to count translations
of $\R^4$ as symmetries.  To include these, we need to take the
semidirect product
\[          \SO_0(3,1) \ltimes \R^4   ,\]
and this is called the {\bf Poincar\'e group}.  

The Poincar\'e 2-group is built from the same ingredients, Lorentz
transformation and translations but in a different way.  Now Lorentz
transformations are treated as symmetries---that is, objects---while
the translations are treated as symmetries of symmetries---that
is, morphisms.  More precisely, the {\bf Poincar\'e 2-group} is defined
to be the skeletal 2-group with:
\begin{itemize}
\item
$G =\SO_0(3,1)$: the group of Lorentz transformations,
\item $H = \R^4$: the group of translations of Minkowski space,
\item the obvious action of $\SO_0(3,1)$ on $\R^4$.
\end{itemize}
As we shall see, the representations of this particular 2-group may
have interesting applications to physics.   For other examples of
2-groups, see our invitation to `higher gauge theory' \cite{BaezHuerta}.
This is a generalization of gauge theory where 2-groups replace groups.

\subsection{2-Vector spaces}

Just as groups act on sets, 2-groups can act on categories.  If a
category is equipped with structure analogous to that of a vector
space, we may call it a `2-vector space', and call a 2-group action
preserving this structure a `representation'.  There is, however,
quite a bit of experimentation underway when it comes to axiomatizing
the notion of `2-vector space'.   In this paper we investigate
representations of 2-groups on infinite-dimensional 2-vector spaces,
following a line of work initiated by Crane, Sheppeard and Yetter
\cite{CraneSheppeard,CraneYetter,Yetter2}.  A quick review of the
history will explain why this is a good idea.

To begin with, finite-dimensional 2-vector spaces were introduced by
Kapranov and Voevodsky \cite{KV}.  Their idea was to replace the
`ground field' $\C$ by the category $\Vect$ of finite-dimensional
complex vector spaces, and exploit this analogy:

\vskip 1em
\begin{center}
{\small
\begin{tabular}{c|c}                               \hline
ordinary             &   higher   
\rule[1.4em]{0em}{0em} \\
linear algebra       &  linear algebra 
\rule[-.8em]{0em}{0em} \\     \hline
$\C$        &  $\Vect$          \rule[1.4em]{0em}{0em}             \\
$+$         &  $\oplus$         \rule[1.2em]{0em}{0em}             \\
$\times$    &  $\otimes$        \rule[1.2em]{0em}{0em}             \\
$0$         &  $\{0\}$          \rule[1.2em]{0em}{0em}             \\
$1$         &  $\C$             \rule[1.2em]{0em}{0em}             \\
\end{tabular}} \vskip 1em
\end{center}

\noindent
Just as every finite-dimensional vector space is isomorphic to $\C^N$
for some $N$, every finite-dimensional Kapranov--Voevodsky 2-vector
space is equivalent to $\Vect^N$ for some $N$.  We can take this as a
{\em definition} of these 2-vector spaces --- but just as with
ordinary vector spaces, there are also intrinsic characterizations
which make this result into a theorem \cite{Neuchl,Yetter1}.

Similarly, just as every linear map $T \maps \C^M \to \C^N$ is equal
to one given by a $N \times M$ matrix of complex numbers, every linear
map $T \maps \Vect^M \to \Vect^N$ is isomorphic to one given by an $N
\times M$ matrix of vector spaces.  Matrix addition and multiplication
work as usual, but with $\oplus$ and $\otimes$ replacing the usual
addition and multiplication of complex numbers.

The really new feature of higher linear algebra is that we also
have `2-maps' between linear maps.  If we have linear maps
$T, T' \maps \Vect^M \to \Vect^N$ given by
$N \times M$ matrices of vector spaces
$T_{n,m}$ and $T'_{n,m}$, then a 2-map $\alpha \maps T \To T'$
is a matrix of linear operators
$\alpha_{n,m} \maps T_{n,m} \to T'_{n,m}$.
If we draw linear maps as arrows:
$$
\xymatrix{\Vect^M \ar[r]^{T}& \Vect^N}
$$
then we should draw 2-maps as 2-dimensional surfaces, like this:
$$
  \xymatrix{
  \Vect^M \ar@/^2ex/[rr]^{T}="g1"\ar@/_2ex/[rr]_{T'}="g2"&&\Vect^N
  \ar@{=>}^{\alpha} "g1"+<0ex,-2.5ex>;"g2"+<0ex,2.5ex>
}
$$
So, compared to ordinary group representation theory, the key novelty
of 2-group representation theory is that besides intertwining operators
between representations, we also have `2-intertwiners', drawn as surfaces.
This boosts the dimension of our diagrams by one, giving 2-group
representation theory an intrinsically 2-dimensional character.

The study of representations of 2-groups on Kapranov--Voevodsky
2-vector spaces was initiated by Barrett and Mackaay
\cite{BarrettMackaay}, and continued by Elgueta \cite{Elgueta2}.  They
came to some upsetting conclusions.  To understand these, we need to
know a bit more about 2-vector spaces.

An object of $\Vect^N$ is an $N$-tuple of finite-dimensional
vector spaces $(V_1, \dots,
V_N)$, so every object is a direct sum of certain special objects
\[   e_i = (0, \dots,\underbrace{\C}_{\mbox{$i$th place}}, \dots, 0 ) . \]
These objects $e_i$ are analogous to the `standard basis' of $\C^N$.
However, unlike the case of $\C^N$, these objects $e_i$ are
essentially the {\it only} basis of $\Vect^N$.  More precisely, given
any other basis $e'_i$, we have $e'_i \cong e_{\sigma(i)}$ for some
permutation $\sigma$.

This fact has serious consequences for representation theory.  A
2-group $\G$ has a group $G$ of objects.  Given a representation of
$\G$ on $\Vect^N$, each $g \in G$ maps the standard basis $e_i$ to
some new basis $e'_i$, and thus determines a permutation $\sigma$.
So, we automatically get an action of $G$ on the finite set $\{1,
\dots, N\}$.

If $G$ is finite, it will typically have many actions on finite sets.
So, we can expect that finite 2-groups have enough interesting
representations on Kapranov--Voevodsky 2-vector spaces to yield
an interesting theory.  But there are many `Lie 2-groups', such 
as the Poincar\'e 2-group, where the group of objects is a Lie group with
few nontrivial actions on finite sets.  Such 2-groups have few
representations on Kapranov--Voevodsky 2-vector spaces.

This prompted the search for a `less discrete' version of
Kapranov--Voevodsky 2-vector spaces, where the finite index set $\{1,
\dots, N \}$ is replaced by something on which a Lie group can act in
an interesting way.  Crane, Sheppeard and Yetter \cite{CraneSheppeard,
CraneYetter, Yetter2} suggested replacing the index set by a
measurable space $X$ and replacing $N$-tuples of finite-dimensional
vector spaces by `measurable fields of Hilbert spaces' on $X$.

Measurable fields of Hilbert spaces have long been important for
studying group representations \cite{Mackey1968}, von Neumann algebras
\cite{Dixmier}, and their applications to quantum physics
\cite{Mackey1978,Varadarajan}.  Roughly, a measurable field of Hilbert
spaces on a measurable space $X$ can be thought of as assigning a
Hilbert space to each $x\in X$, in a way that varies measurably with
$x$.  There is also a well-known concept of `measurable field of
bounded operators' between measurable fields of Hilbert spaces over a
fixed space $X$.  These make measurable fields of Hilbert spaces over
$X$ into the objects of a category $H^X$.  This is the prototypical
example of what Crane, Sheppeard and Yetter call a `measurable
category'.

When $X$ is finite, $H^X$ is essentially just a Kapranov--Voevodsky
2-vector space.  If $X$ is finite and equipped with a measure, $H^X$
acquires a kind of inner product, so it becomes a finite-dimensional
`2-Hilbert space' \cite{Baez2}.  When $X$ is infinite, we should think
of the measurable category $H^X$ as some sort of {\it
infinite-dimensional} 2-vector space.  However, it lacks some features
we expect from an infinite-dimensional 2-Hilbert space: in particular,
there is no inner product of objects.  We discuss this issue further
in Section \ref{conclusion}.

Most importantly, since Lie groups have many actions on measurable
spaces, there is a rich supply of representations of Lie 2-groups on
measurable categories.  As we shall see, a representation of a 
2-group $\G$ on the category $H^X$ gives, in particular, an action of
the group $G$ of objects on the space $X$, just as representations on
$\Vect^N$ gave group actions on $N$-element sets. These actions lead
naturally to a geometric picture of the representation theory.

In fact, a measurable category $H^X$ already has a considerable
geometric flavor.  To appreciate this, it helps to follow Mackey
\cite{Mackey1978} and call a measurable field of Hilbert spaces on the
measurable space $X$ a `measurable Hilbert space bundle' over $X$.
Indeed, such a field $\H$ resembles a vector bundle in that it assigns
a Hilbert space $\H_x$ to each point $x \in X$.  The difference is
that, since $\H$ lives in the world of measure theory rather than
topology, we only require that each point $x$ lie in a {\it
measurable} subset of $X$ over which $\H$ can be trivialized, and we
only require the existence of {\it measurable} transition functions.
As a result, we can always write $X$ as a disjoint union of countably
many measurable subsets on which $\H_x$ has constant dimension.  In
practice, we demand that this dimension be finite or countably
infinite.  Similarly, measurable fields of bounded operators may be
viewed as measurable bundle maps.  So, the measurable category $H^X$
may be viewed as a measurable version of the category of Hilbert space
bundles over $X$.  In concrete examples, $X$ is often a manifold or
smooth algebraic variety, and measurable fields of Hilbert spaces
often arise from bundles or coherent sheaves of Hilbert spaces over
$X$.

\subsection{Representations}

The study of representations of skeletal 2-groups on measurable
categories was begun by Crane and Yetter \cite{CraneYetter}.  The
special case of the Poincar\'e 2-group was studied in detail by Crane
and Sheppeard \cite{CraneSheppeard}.  They noticed interesting
connections to the orbit method in geometric quantization, and also to
the theory of discrete subgroups of $\SO(3,1)$, known as
`Kleinian groups'.  These observations suggest that Lie 2-group
representations on measurable categories deserve a thorough and 
careful treatment.

This, then, is the goal of the present text.  We give {\em geometric} 
descriptions of:
\begin{itemize}
\item a representation $\rho$ of a skeletal 2-group $\G$ 
on a measurable category $H^X$, 
\item an intertwiner between such representations:
$
\xymatrix{\rho \ar[r]^{\phi}& \rho'}\,
$
\item a 2-intertwiner between such intertwiners:  \,
$
  \xymatrix{
  \rho\ar@/^2.5ex/[rr]^{\phi}="g1"\ar@/_2.5ex/[rr]_{\phi'}="g2"&&\rho'
  \ar@{=>}^{\alpha} "g1"+<0ex,-2.5ex>;"g2"+<0ex,2.5ex>
}.$
\end{itemize}
We use the term `intertwiner' as short for `intertwining operator'.
This is a commonly used term for a morphism between group representations;
here we use it to mean a morphism between 2-group representations.  
But in addition to intertwiners, we have something really new:
2-intertwiners between interwiners!  This extra layer of structure
arises from categorification.

We define all these concepts in Sections \ref{reps} and \ref{2vs}.  
Instead of previewing the definitions here, we prefer to sketch 
the geometric picture that emerges in Section \ref{MeasRep}.
So, we now assume $\G$ is a skeletal 2-group described by the data
$(G,H,\rhd)$, as above.  We also assume in what follows that all the
spaces and maps involved are measurable.  Under these assumptions we
can describe representations of $\G$, as well as intertwiners and
2-intertwiners, in terms of familiar geometric constructions---but
living in the category of measurable spaces, rather than smooth
manifolds.  Essentially---ignoring various technical issues which we
discuss later---we obtain the following dictionary relating
representation theory to geometry.

\vskip 1em
\begin{center}
{
\begin{tabular}{c|c}
\hline
\\
representation theory   &   geometry
            \\
    &        \\     \hline \\
a representation of $\G$ on $H^X$  &  
a right action of $G$ on $X$, and a map $X \to H^\ast$  
\\ & making  $X$ a `measurable $G$-equivariant bundle' over $H^\ast$ \\
\\
an intertwiner between   &   
a `Hilbert $G$-bundle' over the pullback of $G$-equivariant bundles
\\
representations on $H^X$ and $H^Y$   & 
and a `$G$-equivariant measurable family of measures' $\mu_y$ on $X$        
\\
\\
a 2-intertwiner    &  a map of Hilbert $G$-bundles \\ \\
\hline
\end{tabular}} \vskip 2em
\end{center}

This dictionary requires some explanation!  First, $H^\ast$ here is
not quite the Pontrjagin dual of $H$, but rather the group, under
pointwise multiplication, of measurable homomorphisms
\[  \chi \maps H \to \C^\times \]
where $\C^\times$ is the multiplicative group of nonzero complex
numbers.  However, this group $H^\ast$ contains the Pontrjagin dual of
$H$.  It turns out that a measurable homomorphism like $\chi$ above,
with our definition of measurable group, is automatically also
continuous.  Since $\C^\times \cong \U(1)\times \R$, we have
\[
      H^\ast = \Hhat \times \hom(H,\R)
\]
where $\Hhat$ is the Pontrjagin dual of $H$.  One can consistently
restrict to `unitary' representations of $\G$, where we replace
$H^\ast$ by $\Hhat$ in the above table.  In most of the paper, we
shall have no reason to make this restriction, but it is often useful
in examples, as we shall see below.

In any case, under some mild conditions on $H$, $H^*$ is again a
measurable space, and its group operations are measurable.
The left action $\rhd$ of $G$ on $H$ naturally
induces a right action of $G$ on $H^*$, say $(\chi,g) \mapsto \chi_g$,
given by
\[      \chi_g(h) = \chi(g\rhd h).\]
This promotes $H^*$ to a right $G$-space.

As indicated in the chart, a representation of $\G$ is simply a
$G$-equivariant map $X \to H^\ast$, where $X$ is a measurable
$G$-space.  Because of the measure-theoretic context, we are happy to
call this a `bundle' even with no implied local triviality in the
topological sense.  Indeed, most of the fibers may even be empty.
Because of the $G$-equivariance, however, fibers are isomorphic along
any given $G$-orbit in $H^\ast$.

This geometric pictures helps us understand irreducibility and related
notions for 2-group representations.  Recall that for ordinary groups,
a representation is `irreducible' if it has no subrepresentations
other than the 0-dimensional representation and itself.  It is
`indecomposable' if it has no direct summands other than the
0-dimensional representation and itself.  Since every direct summand
is a subrepresentation, every indecomposable representation is
irreducible.  The converse is generally false.  However, it is
true in some cases: for example, every {\it unitary} irreducible
representation is indecomposable.

The situation with 2-groups is more subtle.  The notions of
subrepresentation and direct summand generalize to 2-group
representations, but there is also an intermediate notion: a
`retract'.  In fact this notion already exists for group
representations.  A group representation $\rho'$ is a `retract' of
$\rho$ if $\rho'$ is a subrepresentation and there is also an
intertwiner projecting down from $\rho$ to this subrepresentation.
So, we may say a representation is `irretractable' if it has
no retracts other than the 0-dimensional representation and itself.
But for group representations, a retract turns out to be exactly the same
thing as a direct summand, so there is no need for these additional 
notions.

However, we can generalize the concept of `retract' to 2-group
representations---and now things become more interesting!  Now
we have:
\[   \textrm{direct summand} \Longrightarrow \textrm{retract} 
\Longrightarrow \textrm{subrepresentation} \]
and thus:
\[   \textrm{irreducible} \Longrightarrow \textrm{irretractable}
\Longrightarrow \textrm{indecomposable}  
\]
None of these implications are reversible, except {\it
perhaps} every irretractable representation is
irreducible.  At present this question is unsettled.

Indecomposable and irretractable representations play 
important roles in our work.  Each has a nice geometric picture.
Suppose we have a representation of our skeletal 2-group $\G$
corresponding to a $G$-equivariant map $X \to H^*$.  If the $G$-space
$X$ has more than a single orbit, then we can write it as a disjoint
union of $G$-spaces $X=X'\cup X''$ and split the map $X\to H^\ast$
into a pair of maps.  This amounts to writing our 2-group
representation as a direct sum of representations.  So, a
representation on $H^X$ is indecomposable if the $G$-action on
$X$ is transitive.  

By equivariance, this implies that the image of the corresponding map
$X\to H^\ast$ is a single orbit of $H^\ast$, and that the stabilizer
of a point in $X$ is a subgroup of the stabilizer of its image in
$H^\ast$.  In other words, the orbit in $H^\ast$ is a quotient of $X$.
It follows that indecomposable representations of $\G$ are classified
up to equivalence by pairs consisting of:
\begin{itemize}
\item an orbit in $H^\ast$, and
\item a subgroup of the stabilizer of a point in that orbit.
\end{itemize}

It turns out that a representation is irretractable if and only if it
is indecomposable and the map $X \to H^\ast$ is injective.  This of
course means that $X$ is isomorphic as a $G$-space to one of the
orbits of $H^\ast$.  Thus, irretractable representations are classified up
to equivalence by $G$-orbits in $H^\ast$.

In the case of the Poincar\'e 2-group, this has an interesting
interpretation.  The group $H=\R^4$ has $H^\ast\cong\C^4$.
So, a representation in general is given by a
$\SO_0(3,1)$-equivariant map $p\maps X \to \C^4$, where
$\SO_0(3,1)$ acts independently on the real and imaginary parts of a
vector in $\C^4$.  The representation is irretractable if the image of
$p$ is a single orbit.  Restricting to the Pontrjagin dual $\Hhat$
amounts to choosing the orbit of some {\em real} vector, an element of
$\R^4$.  Thus `unitary' irretractable representations are classified by
the $\SO_0(3,1)$ orbits in $\R^4$, which are familiar objects from
special relativity.

If we use $\p = (E,p_x,p_y,p_z)$ as our name for a point of
$\R^4$, then any orbit is a connected component of the solution set of
an equation of the form
\[
    \p \cdot \p = m^2
\]
where the dot denotes the Minkowski metric.  In other words:
\[        E^2 - p_x^2 - p_y^2 - p_y^2 = m^2  . \]
The variable names are the traditional ones in relativity: $E$ stands
for the energy of a particle, while $p_x,p_y,p_z$ are the three
components of its momentum, and the constant $m$ is its mass.  An
orbit corresponding to a particular mass $m$ describes the allowed
values of energy and momentum for a particle of this mass.  These
orbits can be drawn explicitly if we suppress one dimension:
\[
\xy
(0,0)*{\includegraphics{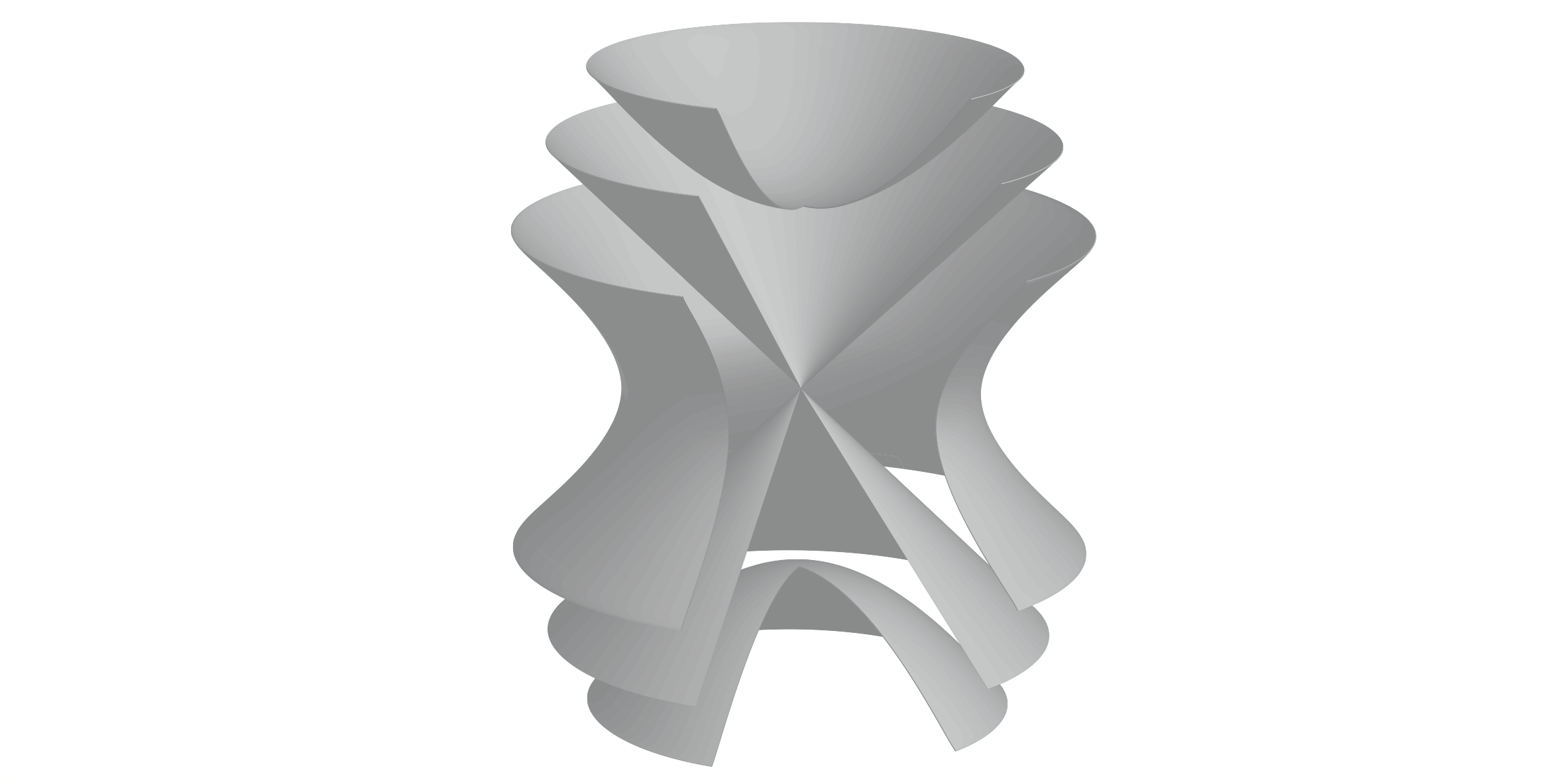}};
(-18,23)*{\scriptstyle m^2>0};
(-22,18)*{\scriptstyle m^2=0};
(-24,12)*{\scriptstyle m^2<0};
(30,3);(30,15)**\dir{-}?(1)*\dir{>};
(35,8)*{\scriptstyle E>0};
(30,-3);(30,-15)**\dir{-} ?(1)*\dir{>};
(35,-8)*{\scriptstyle E<0};
\endxy
\]

Though this picture is dimensionally reduced, it faithfully depicts
all of the orbits in the 4-dimensional case.  There are six types of
orbits, thus giving us six types of irretractable representations of the
Poincar\'e 2-group:
\begin{enumerate}
\item $E = 0$, $m = 0$: the trivial representation (orbit is a single point)
\item $E > 0$, $m = 0$: the `positive energy massless' representation 
\item $E < 0$, $m = 0$: the `negative energy massless' representation
\item $E > 0$, $m > 0$: `positive energy real mass' representations
(one for each $m>0$)
\item $E < 0$, $m > 0$: `negative energy real mass' representations
(one for each $m>0$)
\item $m^2< 0$: `imaginary mass' or `tachyon' representations (one for
each $-im>0$)
\end{enumerate}
On the other hand, there are many more {\em indecomposable}
representations, since these are classified by a choice of one of the
above orbits together with a subgroup of the corresponding point
stabilizer---$\SO(2)$, $\SO(3)$ or $\SO_0(2,1)$ depending on whether
$m^2=0$, $m^2>0$, or $m^2<0$. These indecomposable representations
were studied by Crane and Sheppeard \cite{CraneSheppeard}, though they
called them `irreducible'.

To any reader familiar with the classification of irreducible unitary
representations of the ordinary Poincar\'e {\em group}, the above
story should seem familiar, but also a bit strange.  It should seem
familiar because these group representations are {\em partially}
classified by $\SO(3,1)$ orbits in Minkowski spacetime.  The strange
part is that for these group representations, some extra data is also
needed.  For example, a particle with positive mass and energy is
characterized by both a mass $m > 0$ and a {\em spin}---an irreducible
representation of $\SO(3)$ (or in a more detailed treatment, the
double cover of this group).  By switching to the Poincar\'e 2-group,
we seem to have somehow lost the spin information.

This is not the case.  In fact, as we now explain, the `spin'
information from the ordinary Poincar\'e group representation theory
has simply been pushed up one categorical notch---we will find it in
the intertwiners!  In other words, the concept of spin shows up not in
the classification of representations of the Poincar\'e 2-group, but
in the classification of {\it morphisms} between representations.  The
reason, ultimately, is that Lorentz transformations and translations
of $\R^4$ show up at different levels in the Poincar\'e 2-group: the
Lorentz transformations as objects, and the translations as morphisms.

To see this in more detail, we need to understand the geometry of
intertwiners.  Suppose we have two representations, one on $H^X$ and
one on $H^Y$, given by equivariant bundles $\chi^{\phantom{I}}_1\maps
X\to H^\ast$ and $\chi^{\phantom{I}}_2\maps Y \to H^\ast$.  Looking
again at the chart, the key geometric object is a Hilbert bundle over
the pullback of $\chi_1$ and $\chi_2$.  This pullback may be seen as a
subspace $Z$ of $Y\times X$:
\[
{\xygraph{
  []!{0;<1.5cm,0cm>:<.75cm, 1.3cm>::}  
  []{Z}="Pullback" :@{->} [d]
  {X}="X"  [r] {Y}="Y"
      :@{->}^{\chi^{\phantom{I}}_2} [d] {H^\ast}="M"
        "X"        :@{->}_{\chi^{\phantom{I}}_1} "M"
        "Pullback" :@{->} "Y"
  }}
\qquad \qquad
Z=\{(y,x) \in Y\times X : \chi^{\phantom{I}}_2(y) = \chi^{\phantom{I}}_1(x)\}
\]
It is easy to see that $Z$ is a $G$-space under the diagonal action of
$G$ on $X\times Y$, and that the projections into $X$ and $Y$ are
$G$-equivariant.

If $H^X$ and $H^Y$ are both indecomposable representations, then $X$
and $Y$ each lie over a single orbit of $H^*$.  These orbits must be
the same in order for the pullback $Z$, and hence the space of
intertwiners, to be nontrivial.  If $H^X$ and $H^Y$ are both
irretractable, this implies that they are equivalent.  Thus, given an
irretractable representation represented by an orbit $X$ in $H^\ast$,
the self-intertwiners of this representation are classified by
equivariant Hilbert space bundles over $X$.

Equivariant Hilbert bundles are the subject of Mackey's induced
representation theory \cite{Mackey1952,Mackey1968,Mackey1978}.  In
general, a way to construct an equivariant bundle is to pick a point
in the base space $X$ and a Hilbert space that is a representation of
the stabilizer of that point, and then use the action of $G$ to
`translate' the Hilbert space along a $G$-orbit.  Conversely, given an
equivariant bundle, the fiber over a given point is a representation
of the stabilizer of that point.  Indeed, there is an equivalence of
categories:

\[
\left(\txt{$G$-equivariant vector bundles 
\\ over a homogeneous space $X$}\right)
    \simeq 
\left(\txt{representations of the
 \\ stabilizer of a point in $X$}\right)
\]

\noindent
Proving this is straightforward when we mean `vector bundles' in the
in the ordinary topological sense.  But in Mackey's work, he
generalized this correspondence to a measure-theoretic
context---precisely the context that arises in the theory of 2-group
representations we are considering here!  The upshot for us is that
self-intertwiners of an irretractable representation amount to
representations of the stabilizer subgroup.

To illustrate this idea, let us return to the example of the
Poincar\'e 2-group.  Suppose we have a unitary irretractable
representation of this 2-group.  As we have seen, this is given by one
of the orbits $X\subset \R^4$ of $\SO_0(3,1)$.  Now, consider any
self-intertwiner of this representation.  This is given by a
$\SO_0(3,1)$-invariant Hilbert space bundle over $X$.  By induced
representation theory, this amounts to the same thing as a
representation of the stabilizer of any point $x\in X$.  For a
`positive energy real mass' representation, for example, corresponding
to an ordinary massive particle in special relativity, this stabilizer
is $\SO(3)$, so self-intertwiners are essentially representations of
$\SO(3)$.

In ordinary group representation theory, there is no notion of
`reducibility' for intertwiners.  But here, because of the additional
level of categorical structure, 2-group intertwiners in many ways more
closely resemble group representations than group intertwiners.  There
is a natural concept of `direct sum' of intertwiners, and this gives a
notion of `indecomposable' intertwiner.  Similarly, the concept of
`sub-intertwiner' gives a notion of `irreducible' intertwiner.

Returning yet again to the Poincar\'e 2-group example, consider the
self-intertwiners of a positive energy real mass representation.
We have just seen that these correspond to representations of $\SO(3)$.  
When is such a self-intertwiner irreducible?  Unsurprisingly, the
answer is: precisely when the corresponding representation of $\SO(3)$
is irreducible.

Because of the added layer of structure, we can also ask how a pair of
intertwiners with the same source and target representations might be
related by 2-intertwiner.  As we shall see, intertwiners satisfy an
analogue of Schur's lemma: a 2-intertwiner between {\em irreducible}
intertwiners is either null or an isomorphism, and in the latter case
is essentially unique.  So, there is no interesting information in the
self-2-intertwiners of an irreducible intertwiner.

We conclude with a small warning: in the foregoing description of the
representation theory, we have for simplicity's sake glossed over
certain subtle measure theoretic issues.  Most of these issues make
little difference in the case of the Poincar\'e 2-group, but may be
important for general representations of an arbitrary measurable
2-group.  For details, read the rest of the book!

\subsection{Applications}

Next we describe some potential applications to physics.  Crane and
Sheppeard \cite{CraneSheppeard} originally examined representations of
the Poincar\'e 2-group as part of a plan to construct a physical
theory of a specific sort.   We believe a very similar model 
is implicit in the work of two of the current authors on Feynman
diagrams in quantum gravity \cite{BaratinFreidel1}.  Since proving
this was one of our main motivations for studying the representations
of Lie 2-groups, we would like to recall the ideas here.

A major problem in physics today is trying to extend quantum field
theory, originally formulated for theories that neglect gravity, to
theories that include gravity.  Quantum field theories that neglect
gravity, such as the Standard Model of particle physics, treat
spacetime as flat.  More precisely, they treat it as $\R^4$ with its
Minkowski metric.  The ordinary Poincar\'e group acts as symmetries
here.  

In quantum field theories, physical quantities are often computed with
the help of `Feynman diagrams'.  The details can be found in any good
book on quantum field theory---or, for that matter, Borcherds' review
article for mathematicians \cite{Borcherds}.  However, from a very
abstract perspective, a Feynman diagram can be seen as a graph with:
\begin{itemize}
\item
edges labelled by irreducible representations of some group $G$, and
\item
vertices labelled by intertwiners,
\end{itemize}
where the intertwiner at any vertex goes from the trivial
representation to the tensor product of all the representations labelling
edges incident to that vertex.  In the simplest theories, the group
$G$ is just the Poincar\'e group.  In more complicated theories,
such as the Standard Model, we use a larger group.

There is a way to evaluate Feynman diagrams and get complex numbers,
called `Feynman amplitudes'.  Physically, we think of the group
representations labelling Feynman diagram edges as {\it particles}.
Indeed, we have already said a bit about how an irreducible representation of
the Poincar\'e group can describe a particle with a given mass and
spin.  We think of the intertwiners as {\it interactions}: ways for
the particles to collide and turn into other particles.  So, a
Feynman diagram describes a process involving particles.  When we take
the absolute value of its amplitude and square it, we obtain the
probability for this process to occur.

Feynman diagrams are essentially one-dimensional structures, since
they have vertices and edges.  On the other hand, there is an approach
to quantum gravity that uses closely analogous {\em two-dimensional}
structures called `spin foams' \cite{BaezSF, BarrettCrane,
FreidelKrasnov, Rovelli}.  The 2-dimensional analogue of a graph is
called an `2-complex': it is a structure with vertices, edges {\it and
faces}.  In a spin foam, we label the vertices, edges and faces of a
2-complex with data of some sort.  Like Feynman diagrams, spin foams
should be thought of as describing physical processes---but now of a
higher-dimensional sort.  A spin foam model is a recipe for computing
complex numbers from spin foams: their `amplitudes'.  As before, when
we take the absolute value of these amplitude and square them, we
obtain probabilities.

The first spin foam model, only later recognized as such, goes back to
a famous 1968 paper by Ponzano and Regge \cite{PonzanoRegge}.  This
described {\em Riemannian} quantum gravity in {\em 3-dimensional}
spacetime---two drastic simplifications that are worth explaining.

First of all, gravity is much easier to deal with in 3d spacetime,
since in this case, in the absence of matter, all solutions of
Einstein's equations for general relativity look alike locally.
More precisely, any spacetime obeying these equations can be locally
identified, after a suitable coordinate transformation, with $\R^3$
equipped with its Minkowski metric
\[         \x \cdot \x' = tt' - xx' - yy' . \]
This is very different from the physically realistic 4d case, where
gravitational waves can propagate through the vacuum, giving a
plethora of locally distinct solutions.  Physicists say that 3d
gravity lacks `local degrees of freedom'.  This makes it much easier
to study---but it retains some of the conceptual and technical
challenges of the 4d problem.

Second of all, in `Riemannian quantum gravity', we investigate a
simplified world where time is just the same as space.  In 4d
spacetime, this involves replacing Minkowski spacetime with 4d
Euclidean space---that is, $\R^4$ with the inner product
\[         \x \cdot \x' = tt' + xx' + yy' + zz' . \]
While physically quite unrealistic, this switch simplifies some of the
math.  The reason, ultimately, is that the group of Lorentz
transformations, $\SO_0(3,1)$, is noncompact, while the rotation group
$\SO(4)$ is compact.  A compact Lie group has a countable set of
irreducible unitary representations instead of a continuum, and this
makes some calculations easier.  For example, certain integrals become
sums.  

Ponzano and Regge found that after making both these simplifications,
they could write down an elegant theory of quantum gravity, now called
the Ponzano--Regge model.  Their theory is deeply related to
representations of the 3-dimensional rotation group, $\SO(3)$.  In
modern terms, the idea is to start with a 3-manifold equipped with a
triangulation $\Delta$.  Then we form the Poincar\'e dual of $\Delta$
and look at its 2-skeleton $K$.  In simple terms, $K$ is the 2-complex
with:
\begin{itemize}
\item one vertex for each tetrahedron in $\Delta$,
\item one edge for each triangle in $\Delta$, 
\item one face for each edge of $\Delta$.
\end{itemize}
We call such a thing a `2-complex'.  Note that a 2-complex is
precisely the sort of structure that, when suitably labelled, gives a
spin foam!  To obtain a spin foam, we:
\begin{itemize}
\item label each face of $K$ with an irreducible representation of
$\SO(3)$, and
\item label each edge of $K$ with an intertwiner.
\end{itemize}
There is a way to compute an amplitude for such a spin foam,
and we can use these amplitudes to answer physically interesting 
questions about 3d Riemannian quantum gravity.  

The Ponzano--Regge model served as an inpiration for many further
developments.  In 1997, Barrett and Crane proposed a similar model for
{\it 4-dimensional} Riemannian quantum gravity \cite{BarrettCrane}.
More or less simultaneously, the general concept of `spin foam model'
was formulated \cite{BaezSF}.  Shortly thereafter, spin foam models of
4d Lorentzian quantum gravity were proposed, closely modelled after
the Barrett-Crane model \cite{DFKR,PR}.  Later, `improved' models were
developed by Freidel and Krasnov \cite{FreidelKrasnov} and Engle,
Pereira, Rovelli and Livine \cite{EPRL}.  These newer models are
beginning to show signs of correctly predicting some phenomena we
expect from a realistic theory of quantum gravity.  However, this is
work in progress, whose ultimate success is far from certain.

One fundamental challenge is to incorporate {\em matter} in a spin
foam model of quantum gravity.  Indeed, any theory that fails to do
this is at best a warmup for a truly realistic theory.  Recently, a
lot of progress has been made on incorporating matter in the
Ponzano--Regge model.  Here is where spin foams meet Feynman diagrams!

The idea is to compute Feynman amplitudes using a slight
generalization of the Ponzano--Regge model which lets us include
matter \cite{Barrett}.  This model takes the gravitational
interactions of particles into account.  As a consistency check, we
want the `no-gravity limit' of this model to reduce to the standard
recipe for computing Feynman amplitudes in quantum field theory---or
more precisely its analogue with Euclidean $\R^3$ replacing 4d
Minkowski spacetime.  And indeed, this was shown to be true \cite{PR1,
PRIIIa, PRIIIb}.

This raised the hope that the same sort of strategy can work in
4-dimensional quantum gravity.  It was natural to start with the
`no-gravity limit', and ask if the usual Feynman amplitudes for
quantum field theory in flat 4d spacetime can be computed using a spin
foam model.  If we could do this, the result would not be a theory of
quantum gravity, but it would provide a radical new formulation of
quantum field theory, in which Minkowski spacetime is replaced by an
inherently quantum-mechanical spacetime built from spin foams.  If a
formulation exists, it may help us develop models describing quantum
gravity and matter in 4 dimensions.

Recent work by \cite{BaratinFreidel1} gives precisely such a
formulation, at least in the 4-dimensional {\it Riemannian} case.  In
other words, this work gives a spin foam model for computing Feynman
amplitudes for quantum field theories, not on Minkowski spacetime, but
rather on 4-dimensional Euclidean space.  Feynman diagrams for such
theories are built using representations, not of the Poincar\'e group,
but of the {\bf Euclidean group}:
\[          \SO(4) \ltimes \R^4 .\]

More recently still, it was seen that this new model is a close
relative of the Crane--Sheppeard model
\cite{BaratinFreidel2,BaratinWise}!  The only difference is that where
the Crane--Sheppeard model uses the Poincar\'e 2-group, the new model
uses the {\bf Euclidean 2-group}, a skeletal 2-group for which:
\begin{itemize}
\item
$G =\SO(4)$: the group of rotations of 4d Euclidean space,
\item $H = \R^4$: the group of translations 4d Euclidean space,
\item the obvious action of $\SO(4)$ on $\R^4$.
\end{itemize}
The representation theory of the Euclidean 2-group is very much like
that of the Poincar\'e 2-group, but with concentric spheres replacing
the hyperboloids
\[        E^2 - p_x^2 - p_y^2 - p_y^2 = m^2  . \]

So, we can now guess the meaning of the Crane--Sheppeard model: it
should give a new way to compute Feynman integrals for ordinary
quantum field theories on 4d Minkowski spacetime.  To conclude, 
let us just say a word about how this model actually works.

It helps to go back to the Ponzano--Regge model.  We can describe this
directly in terms of a 3-manifold with triangulation $\Delta$, instead of the
Poincar\'e dual picture.  In these terms, each spin foam corresponds
to a way to:
\begin{itemize}
\item label each edge of $\Delta$ with an irreducible representation of
$\SO(3)$, and
\item label each triangle of $\Delta$ with an intertwiner.
\end{itemize}
The Ponzano--Regge model gives a way to compute an amplitude for any
such labelling.  

The Crane--Sheppeard model does a similar thing one dimension up.  
Suppose we take a 4-manifold with a triangulation $\Delta$.
Then we may:
\begin{itemize}
\item label each edge of $\Delta$ with an irretractable representation of
the Poincar\'e 2-group,
\item label each triangle of $\Delta$ with an irreducible intertwiner, and
\item label each tetrahedron of $\Delta$ with a 2-intertwiner.
\end{itemize}
The Crane--Sheppeard model gives a way to compute an amplitude for any
such labelling.

\subsection{Plan of the paper}

Above we describe a 2-group as a category equipped with
a multiplication and inverses.  While this is correct, another
equivalent approach turns out to be more useful for our purposes here.
Just as a group can be thought of as a category that has one object
and for which all morphisms are invertible, a 2-group can be thought
of as a 2-category that has one object and for which all morphisms and
2-morphisms are invertible.  In Section \ref{reps} we recall the
definition of a 2-category and explain how to think of a 2-group as a
2-category of this sort.  We also describe how to construct 2-groups
from crossed modules, and vice versa.  We conclude by defining the
2-category $\Rep(\G)$ of representations of a fixed 2-group $\G$ in a
fixed 2-category $\tc$.

In Section \ref{2vs} we explain measurable categories.  We first
recall Kapranov and Voevodsky's 2-vector spaces, and then introduce
the necessary analysis to present Yetter's results on measurable
categories.  To do this, we need to construct the 2-category $\me$ of
measurable categories.  The problem is that we do not yet know an
intrinsic characterization of measurable categories.  At present, a
measurable category is simply defined as one that is
`$C^*$\!-equivalent' to a category of measurable fields of Hilbert
spaces.  So, it is a substantial task to construct the 2-category
$\me$.  As a warmup, we carry out a similar construction of the
2-category of Kapranov--Voevodsky 2-vector spaces (for which an
intrinsic characterization is known, making a simpler approach
possible).

Working in this picture, we study the representations of 2-groups on
measurable categories in Section \ref{MeasRep}.  We present a detailed
study of equivalence, direct sums, tensor products, reducibility,
decomposability, and retractability for representations and
1-intertwiners.  While our work is hugely indebted to that of Crane,
Sheppeard, and Yetter, we confront many issues they did not discuss.
Some of these arise from the fact that they implicitly consider
representations of discrete 2-groups, while we treat {\it measurable}
representations of {\it measurable} 2-groups---for example, Lie
2-groups.  The representations of a Lie group viewed as a discrete
group are vastly more pathological than its measurable
representations.  Indeed, this is already true for $\mathbb{R}$, which
has enormous numbers of nonmeasurable 1-dimensional representations if
we assume the axiom of choice, but none if we assume the axiom of
determinacy.  The same phenomenon occurs for Lie 2-groups.  So, it is
important to treat them as measurable 2-groups, and focus on their
measurable representations.

In Section \ref{conclusion}, we conclude by sketching some directions
for future research.  We argue that a measurable category $H^X$ becomes
a `separable 2-Hilbert space' when the measurable space $X$ is equipped
with a $\sigma$-finite measure.  We also sketch how this approach to
separable 2-Hilbert spaces should fit into a more general approach
to 2-Hilbert spaces based on von Neumann algebras.

Finally, Appendix \ref{tools} contains some results from analysis that
we need.  {\bf Nota Bene:} in this paper, we always use `measurable
space' to mean `standard Borel space': that is, a set $X$ with a
$\sigma$-algebra of subsets generated by the open subsets for some
complete separable metric on $X$.  Similarly, we use `measurable
group' to mean `lcsc group': that is, a topological group for which
the topology is locally compact Hausdorff and second countable.  We
also assume all our measures are $\sigma$-finite and positive.  These
background assumptions give a fairly convenient framework for the
analysis in this paper.

%
\section{Representations of 2-groups}
\label{reps}
%

%
\subsection{From groups to 2-groups}
%

\subsubsection{2-groups as 2-categories}
\label{2group2cat}

We have said that a 2-group is a category equipped with product
and inverse operations satisfying the usual group axioms.  However,
a more powerful approach is to think of a 2-group as a special
sort of 2-category.

To understand this, first note that
a group $G$ can be thought of as a category with a single object
$\star$, morphisms labeled by elements of $G$, and composition defined
by multiplication in $G$:
\[
\xymatrix{\star \ar[r]^{g_1} & \star \ar[r]^{g_2} &\star }  \quad =
\quad \xymatrix{\star \ar[r]^{g_2g_1} & \star }
\]
In fact, one can define a group to be a category with a single object
and all morphisms invertible. The object $\star$ can be thought of as
an object whose symmetry group is $G$.

In a 2-group, we add an additional layer of structure to this picture,
to capture the idea of {\em symmetries between symmetries}.  So, in
addition to having a single object $\star$ and its automorphisms, we
have isomorphisms {\em between} automorphisms of $\star$:
$$
  \xymatrix{
  \star\ar@/^2ex/[rr]^{g}="g1"\ar@/_2ex/[rr]_{g'}="g2"&&\star
  \ar@{=>}^{h} "g1"+<0ex,-2.5ex>;"g2"+<0ex,2.5ex>
}
$$
These `morphisms between morphisms' are called {\it 2-morphisms}.

To make this precise, we should recall that a 2-category consists of:

\begin{itemize}
\item objects: \,$X,Y,Z, \ldots$
\item morphisms:
$
\xymatrix{X \ar[r]^{f}&Y}
$
\item  2-morphisms:  \,
$
  \xymatrix{
  X\ar@/^2.5ex/[rr]^{f}="g1"\ar@/_2.5ex/[rr]_{f'}="g2"&&Y
  \ar@{=>}^{\alpha} "g1"+<0ex,-2.5ex>;"g2"+<0ex,2.5ex>
}
$
\end{itemize}
Morphisms can be composed as in a category, and 2-morphisms can be
composed in two distinct ways: vertically:
$$
\xymatrix{
   X\ar@/^4ex/[rr]^{f}="g1"\ar[rr]^(0.35){f'}\ar@{}[rr]|{}="g2"
  \ar@/_4ex/[rr]_{f''}="g3"&&Y
  \ar@{=>}^{\alpha} "g1"+<0ex,-2ex>;"g2"+<0ex,1ex>
  \ar@{=>}^{\alpha'} "g2"+<0ex,-1ex>;"g3"+<0ex,2ex>
}
\quad =\quad
\xymatrix{
 X\ar@/^3ex/[rr]^{f}="g1"
  \ar@/_3ex/[rr]_{f''}="g3"&&Y
  \ar@{=>}^{\alpha'\cdot\alpha} "g1"+<0ex,-2ex>;"g3"+<0ex,2ex>
}
$$
and horizontally:
$$
  \xymatrix{
  X\ar@/^2.5ex/[rr]^{f_1}="g1"\ar@/_2.5ex/[rr]_{f'_1}="g2"&&Y
  \ar@{=>}^{\alpha_1} "g1"+<0ex,-2.5ex>;"g2"+<0ex,2.5ex>
  \ar@/^2.5ex/[rr]^{f_2}="g1"\ar@/_2.5ex/[rr]_{f'_2}="g2"&&Z
  \ar@{=>}^{\alpha_2} "g1"+<0ex,-2.5ex>;"g2"+<0ex,2.5ex>
}
\quad =\quad
\xymatrix{
 X\ar@/^3ex/[rr]^{f_2 f_1}="g1"
  \ar@/_3ex/[rr]_{f'_2 f_1'}="g3"&&Y
  \ar@{=>}^{\alpha_2\circ\alpha_1} "g1"+<-2ex,-2ex>;"g3"+<-2ex,2ex>
}
$$
A few simple axioms must hold for this to be a 2-category:
\begin{itemize}
\item Composition of morphisms must be associative, and every
object $X$ must have a morphism
$$
\xymatrix{X \ar[r]^{1_x}&X}
$$
serving as an identity for composition, just as in an ordinary
category.
\item Vertical composition must be associative,
and every morphism $\xymatrix{X \ar[r]^{f}&Y}$ must have a 2-morphism
$$
  \xymatrix{
  X\ar@/^2.5ex/[rr]^{f}="g1"\ar@/_2.5ex/[rr]_{f}="g2"&&Y
  \ar@{=>}^{1_f} "g1"+<0ex,-2.5ex>;"g2"+<0ex,2.5ex>
}
$$
serving as an identity for vertical composition.
\item
Horizontal composition must be associative, and the
2-morphism
$$
  \xymatrix{
  X\ar@/^2.5ex/[rr]^{1_X}="g1"\ar@/_2.5ex/[rr]_{1_X}="g2"&&X
  \ar@{=>}^{1_{1_X}} "g1"+<0ex,-2.5ex>;"g2"+<0ex,2.5ex>
}
$$
must serve as an identity for horizontal composition.
\item
Vertical composition and
horizontal composition of 2-morphisms
must satisfy the following \textbf{exchange law}:
\beq
(\alpha'_2 \cdot \alpha_2) \circ (\alpha'_1 \cdot \alpha_1) = (\alpha'_2 \circ \alpha'_1) \cdot (\alpha_2 \circ
\alpha_1)
\eeq
so that diagrams of the form
$$
\xymatrix{
  X\ar@/^4ex/[rr]^{f_1}="g1"\ar[rr]^(0.35){f'_1}\ar@{}[rr]|{}="g2"
  \ar@/_4ex/[rr]_{f''_1}="g3"&&Y
  \ar@{=>}^{\alpha_1} "g1"+<0ex,-2ex>;"g2"+<0ex,1ex>
  \ar@{=>}^{\alpha'_1} "g2"+<0ex,-1ex>;"g3"+<0ex,2ex>
  \ar@/^4ex/[rr]^{f_2}="g1"\ar[rr]^(0.35){f'_2}\ar@{}[rr]|{}="g2"
  \ar@/_4ex/[rr]_{f''_2}="g3"&& Z
  \ar@{=>}^{\alpha_2} "g1"+<0ex,-2ex>;"g2"+<0ex,1ex>
  \ar@{=>}^{\alpha'_2} "g2"+<0ex,-1ex>;"g3"+<0ex,2ex>
}
$$
define unambiguous 2-morphisms.
\end{itemize}
For more details, see the references \cite{KS,MacLane}.

We can now define a 2-group:

\begin{defn}
A {\bf 2-group} is a 2-category with a unique object
such that all morphisms and 2-morphisms are invertible.
\end{defn}
In fact it is enough for all 2-morphisms to have `vertical' inverses;
given that morphisms are invertible it then follows that 2-morphisms
have horizontal inverses.  Experts will realize that we are defining a
`strict' 2-group \cite{BaezLauda}; we will never use any other sort.

The 2-categorical approach to 2-groups is a powerful conceptual tool.
However, for explicit calculations it is often useful to treat
2-groups as `crossed modules'.

\subsubsection{Crossed modules} \label{crossedmod}

\label{crossmod}

Given a 2-group $\G$, we can extract from it four pieces of
information which form something called a `crossed module'.
Conversely, any crossed module gives a 2-group.  In fact,
2-groups and crossed modules are just different ways of
describing the same concept.  While less elegant than 2-groups,
crossed modules are good for computation, and also good for
constructing examples.

Let $\G$ be a 2-group.  From this we can extract:
\begin{itemize}
\item the group $G$ consisting of all morphisms of $\G$: \,
$
\xymatrix{\star \ar[r]^{g} & \star }
$
\item the group $H$ consisting of all 2-morphisms whose source
is the identity morphism:
$$
  \xymatrix{
  \star\ar@/^2ex/[rr]^{1}="g1"\ar@/_2ex/[rr]_{g}="g2"&&\star
  \ar@{=>}^{h} "g1"+<0ex,-2.5ex>;"g2"+<0ex,2.5ex>
}
$$
\item the homomorphism $\d \maps H \to G$ assigning to
each 2-morphism $h \in H$ its target:
$$
  \xymatrix{
  \star\ar@/^2ex/[rr]^{1}="g1"\ar@/_2ex/[rr]_{\d(h) := g}="g2"&&\star
  \ar@{=>}^{h} "g1"+<0ex,-2.5ex>;"g2"+<0ex,2.5ex>
}
$$
\item the action $\rhd$ of $G$ as automorphisms of $H$ given by
`horizontal conjugation':
$$
  \xymatrix{
  \star\ar@/^2ex/[rr]^{1}="g1"\ar@/_2ex/[rr]_{g\d(h)g^{-1}}="g2"&&\star
  \ar@{=>}^{g \rhd h} "g1"+<0ex,-2.5ex>;"g2"+<0ex,2.5ex>
}
:=
  \xymatrix{
  \star \ar@/^2.5ex/[rr]^{g^{-1}}="g1"\ar@/_2.5ex/[rr]_{g^{-1}}="g2"&& \star
  \ar@{=>}^{1_{g^{-1}}} "g1"+<0ex,-2.5ex>;"g2"+<0ex,2.5ex>
  \ar@/^2.5ex/[rr]^{1}="g1"\ar@/_2.5ex/[rr]_{\d h}="g2"&& \star
  \ar@{=>}^{h} "g1"+<0ex,-2.5ex>;"g2"+<0ex,2.5ex>
  \ar@/^2.5ex/[rr]^{g}="g1"\ar@/_2.5ex/[rr]_{g}="g2"&& \star
  \ar@{=>}^{1_g} "g1"+<0ex,-2.5ex>;"g2"+<0ex,2.5ex>
}
$$

\end{itemize}
It is easy to check that
the homomorphism $\d\maps H \to G$ is compatible with $\rhd$ in
the following two ways:
\beqa \label{comp1}
\d(g \rhd h) &=& g \d(h) g^{-1}\\ \label{comp2}
\d(h) \rhd h'&=& h  h'  h^{-1}.
\eeqa
Such a system $(G,H,\rhd,\d)$ satisfying equations (\ref{comp1}) and
(\ref{comp2}) is called a \textbf{crossed module}.

We can recover the 2-group $\G$ from its crossed module
$(G,H,\rhd,\d)$, using a process we now describe.  In fact, every
crossed module gives a 2-group via this process \cite{ForresterBarker}.

Given a crossed module $(G,H,\rhd,\d)$, we construct a 2-group $\G$
with:
\begin{itemize}
\item one object: \,$\star$
\item elements of $G$ as morphisms: \,
$\xymatrix{\star \ar[r]^{g}&\star}$
\item pairs $u = (g, h) \in G\times H$ as 2-morphisms, where $(g,h)$
is a 2-morphism from $g$ to $\d(h)g$.  We draw such a pair as:
$$ u=
  \xymatrix{
  \star\ar@/^2.5ex/[rr]^{g}="g1"\ar@/_2.5ex/[rr]_{g'}="g2"&&\star
  \ar@{=>}^{h} "g1"+<0ex,-2ex>;"g2"+<0ex,2ex>
}
$$
where $g'=\d(h)g$.

\end{itemize}
Composition of morphisms and vertical composition of 2-morphisms
are defined using multiplication in $G$ and $H$, respectively:
\[
\begin{aligned}
\xymatrix{\star \ar[r]^{g_1} & \star \ar[r]^{g_2} & \star}  \quad =
\quad \xymatrix{\star \ar[r]^{g_2 g_1} & \star}
\end{aligned}
\]
and
$$
\begin{aligned}
\xymatrix{
\star\ar@/^4ex/[rr]^{g}="g1"\ar[rr]^(0.35){g'}\ar@{}[rr]|{}="g2"
  \ar@/_4ex/[rr]_{g''}="g3"&&\star
  \ar@{=>}^{h} "g1"+<0ex,-2ex>;"g2"+<0ex,1ex>
  \ar@{=>}^{h'} "g2"+<0ex,-1ex>;"g3"+<0ex,2ex>
}
\quad =\quad
\xymatrix{
 \star\ar@/^3ex/[rr]^{g}="g1"
  \ar@/_3ex/[rr]_{g''}="g3"&&\star
  \ar@{=>}^{h' h} "g1"+<0ex,-2ex>;"g3"+<0ex,2ex>
}
\end{aligned}
$$
with $g' = \d(h)g$ and $g'' = \d(h') \d(h) g = \d(h'h) g$.  In other words,
suppose we have 2-morphisms $u=(g,h)$ and $u'=(g',h')$.  If $g'=\d(h)g$,
they are vertically composable, and their vertical composite is given by:
\beq
            u'\cdot u = (g',h')\cdot(g,h) = (g,h'h)
\eeq
They are always horizontally composable, and we define their horizontal
composite by:
$$
\begin{aligned}
  \xymatrix{
  \star\ar@/^2ex/[rr]^{g_1}="g1"\ar@/_2ex/[rr]_{g'_1}="g2"&&\star
  \ar@{=>}^{h_1} "g1"+<0ex,-2.5ex>;"g2"+<0ex,2.5ex>
  \ar@/^2ex/[rr]^{g_2}="g1"\ar@/_2ex/[rr]_{g'_2}="g2"&&\star
  \ar@{=>}^{h_2} "g1"+<0ex,-2.5ex>;"g2"+<0ex,2.5ex>
}
\quad =\quad
\xymatrix@C=1.3cm{
 \star\ar@/^3ex/[rr]^{g_2 g_1}="g1"
  \ar@/_3ex/[rr]_{g'_2 g_1'}="g3"&&\star
  \ar@{=>}^{h_2 (g_2 \rhd h_1)} "g1"+<-4ex,-3ex>;"g3"+<-4ex,3ex>
}
\end{aligned}
$$
So, horizontal composition makes the set of 2-morphisms into a
group, namely the semidirect product $G \ltimes H$ with multiplication:
\beq
(g_2,h_2)\circ (g_1, h_1)  \equiv (g_2g_1, h_2 (g_2 \rhd h_1))
\eeq
One can check that the exchange law
\beq
(u'_2 \cdot u_2) \circ (u'_1 \cdot u_1) =
(u'_2 \circ u'_1) \cdot (u_2 \circ u_1)
\eeq
holds for 2-morphisms $u_i=(g_i,h_i)$ and
$u'_i=(g'_i,h'_i)$, so that the diagram
$$
\begin{aligned}
\xymatrix{
  \star\ar@/^4ex/[rr]^{g_1}="g1"\ar[rr]^(0.35){g'_1}\ar@{}[rr]|{}="g2"
  \ar@/_4ex/[rr]_{g''_1}="g3"&&\star
  \ar@{=>}^{h_1} "g1"+<0ex,-2ex>;"g2"+<0ex,1ex>
  \ar@{=>}^{h'_1} "g2"+<0ex,-1ex>;"g3"+<0ex,2ex>
  \ar@/^4ex/[rr]^{g_2}="g1"\ar[rr]^(0.35){g'_2}\ar@{}[rr]|{}="g2"
  \ar@/_4ex/[rr]_{g''_2}="g3"&&\star
  \ar@{=>}^{h_2} "g1"+<0ex,-2ex>;"g2"+<0ex,1ex>
  \ar@{=>}^{h'_2} "g2"+<0ex,-1ex>;"g3"+<0ex,2ex>
}
\end{aligned}
$$
gives a well-defined 2-morphism.

\medskip
To see an easy example of a 2-group, start with a group $G$ acting as
automorphisms of a group $H$.  If we take $\rhd$ to be this action and
let $\d\colon H \to G$ be the trivial homomorphism, we can easily
check that the crossed module axioms (\ref{comp1}) and (\ref{comp2})
hold {\it if $H$ is abelian}.  So, if $H$ is abelian, we obtain a
2-group with $G$ as its group of objects and $G \ltimes H$ as its
group of morphisms, where the semidirect product is defined using the
action $\rhd$.

Since $\d$ is trivial in this example, any 2-morphism $u = (g,h)$
goes from $g$ to itself:
$$
  \xymatrix{
  \star\ar@/^2.5ex/[rr]^{g}="g1"\ar@/_2.5ex/[rr]_{g}="g2"&&\star
  \ar@{=>}^{h} "g1"+<0ex,-2ex>;"g2"+<0ex,2ex>
}
$$
So, this type of 2-group has only 2-{\em auto}morphisms,
and each morphism has precisely one 2-automorphism for
each element of $H$.

A 2-group with trivial $\d$ is called \textbf{skeletal}, and one can
easily see that every skeletal 2-group is of the form just described.
An important point is that for a skeletal 2-group, the group $H$ is
necessarily abelian.  While we derived this using (\ref{comp2}) above,
the real reason is the Eckmann--Hilton argument \cite{EH}.

An important example of a skeletal 2-group is the `Poincar\'e 2-group'
coming from the semidirect product $SO(3,1)\ltimes \R^4$ in precisely
the way just described \cite{Baez1}.  

%
\subsection{From group representations to 2-group representations}
%

\subsubsection{Representing groups}

In the ordinary theory of groups, a group $G$
may be represented on a vector space.  In the language of categories,
such a representation is nothing but a {\it functor} $\rho \maps G \to \Vect$,
where $G$ is seen as category with one object $\ast$, and $\Vect$ is
the category of vector spaces and linear operators.  To see this, note
that such a functor must send the object $\ast \in G$ to some vector space
$\rho(\ast) = V \in \Vect$.  It must also send each morphism $\star \stackrel{g}{\to}
\star$ in $G$---or in other words, each \textit{element} of our group---to a linear map
$$
\xymatrix{V \ar[r]^{\rho(g)} & V}
$$
Saying that $\rho$ is a functor then means that it preserves identities and composition:
$$\rho(1) = \unit_V $$
$$\rho(gh) = \rho(g)\rho(h)$$
for all group elements $g,h$.

In this language, an intertwining operator between group
representations---or `intertwiner', for short---is nothing
but a  {\it natural transformation}.  To see this,
suppose that $\rho_1,\rho_2 \maps G \to \Vect$ are functors and
$ \phi \maps \rho_1 \To \rho_2 $
is a natural transformation.   Such a transformation must
give for each object $\star \in G$ a linear operator from $\rho_1(\ast) = V_1$
to $\rho_2(\ast) = V_2$.  But $G$ is a category with one object, so
we have a single operator $\phi \maps V_1 \to V_2$.
Saying that the transformation is `natural' then means that this square
commutes:
\beq \label{intertwine}
\xymatrix{
  V_1\ar[rr]^{\rho_1(g)}\ar[dd]_{\phi}&& V_1\ar[dd]^{\phi}\\
  \\
  V_2\ar[rr]_{\rho_2(g)}&&V_2
}
\eeq
for each group element $g$.  This says simply that
\beq
\label{commute}
\rho_2(g)\phi = \phi\rho_1(g)
\eeq
for all $g \in G$.  So, $\phi$ is an intertwiner in the usual sense.

Why bother with the categorical viewpoint on
on representation theory?  One reason is that it lets us
generalize the concepts of group representation and intertwiner:

\begin{defn}
If $G$ is a group and $C$ is any category, a {\bf representation} of $G$ in $C$ is a
functor $\rho$ from $G$ to $C$, where $G$ is seen as a category with one object.
Given representations $\rho_1$ and $\rho_2$ of $G$ in $C$, an {\bf intertwiner}
$\phi \maps \rho \to \rho'$ is a natural transformation from $\rho$ to $\rho'$.
\end{defn}
In ordinary representation theory we take $C = \Vect$; but we
can also, for example, work with the category of sets $C= \mathrm{Set}$,
so that a representation of $G$ in $C$ picks out a set together with
an action of $G$ on this set.

Quite generally, there is a category ${\bf Rep}(G)$ whose objects are
representations of $G$ in $C$, and whose morphisms are the
intertwiners.  Composition of intertwiners is defined by composing
natural transformations.  We define two representations $\rho_1,
\rho_2 \maps G \to C$ to be {\bf equivalent} if there exists an
intertwiner between them which has an inverse.  In other words,
$\rho_1$ and $\rho_2$ are equivalent if there is a natural {\it
isomorphism} between them.

In the next section we shall see that the representation theory of 2-groups
amounts to taking all these ideas and `boosting the dimension by one', using
2-categories everywhere instead of categories.

\subsubsection{Representing 2-groups} \label{2rep}

Just as groups are typically represented in the category of vector
spaces, 2-groups may be represented in some 2-category of
`2-vector spaces'. However, just as for group
representations, the definition of a 2-group representation does not
depend on the particular target 2-category we wish to represent our
2-groups in. We therefore present the definition in its abstract form
here, before describing precisely what sort of 2-vector spaces we will
use, in Section \ref{2vs}.

We have seen that a representation of a group $G$ in a category $C$ is
a functor $\rho \maps G \to C$ between categories.  Similarly, a
representation of a 2-group will be a `2-functor' between
2-categories.  As with group representations, we have intertwiners
between 2-group representations, which in the language of 2-categories are
`pseudonatural transformations'.  But the extra layer of categorical
structure implies that in 2-group representation theory we also have
`2-intertwiners' going between intertwiners.  These are
defined to be `modifications' between pseudonatural transformations.

The reader can learn the general notions of `2-functor', `pseudonatural
transformation' and `modification' from the review article by Kelly and
Street \cite{KS}.  However, to make this paper self-contained, we
describe these concepts below in the special cases that we
actually need.

%
%
\begin{defn} If $\G$ is a 2-group and $\tc$ is any 2-category, then a
{\bf representation} of $\G$ in $\tc$ is a 2-functor $\rho$ from $\G$ to $\tc$.
\end{defn}
%
%

Let us describe what such a 2-functor amounts to.  Suppose a 2-group
$\G$ is given by the crossed module $(G,H,\d,\rhd)$, so that $G$ is
the group of morphisms of $\G$, and $G\ltimes H$ is the group of
2-morphisms, as described in section \ref{crossedmod}.  Then a
representation $\rho\maps \G \to \tc$ is specified by:

\begin{itemize}
\item an object $V$ of $\tc$, associated to the single object of the
      2-group: $\rho(\star) =V$
\item for each morphism $g\in G$, a morphism in $\tc$ from $V$ to itself:
$$
\xymatrix{V \ar[r]^{\rho(g)} & V}
$$
\item for each 2-morphism $u=(g,h)$, a 2-morphism in $\tc$
$$
  \xymatrix{
  V \ar@/^2ex/[rr]^{\rho(g)}="g1"\ar@/_2ex/[rr]_{\rho(\d hg)}="g2"&& V
  \ar@{=>}^{\rho(u)} "g1"+<0ex,-2.5ex>;"g2"+<0ex,2.5ex>
}
$$
\end{itemize}
That $\rho$ is a 2-functor means these correspondences preserve
identities and all three composition operations: composition of
morphisms, and horizontal and vertical composition of 2-morphisms.  In
the case of a 2-group, preserving identities follows from preserving
composition.  So, we only need require:
\begin{itemize}
\item
for all morphisms $g,g'$:
\beq \label{1comp}
\rho(g'g) = \rho(g')\, \rho(g)
\eeq
\item
for all vertically composable 2-morphisms $u$ and $u'$:
\beq
\label{2comp_vert}
\rho(u'\cdot u) = \rho(u') \cdot \rho(u)
\eeq
\item
for all 2-morphisms $u,u'$:
\be
\label{2comp_hor}
\rho(u'\circ u) = \rho(u') \circ \rho(u)
\eeq
\end{itemize}
Here the compositions laws in $\mathcal{G}$ and $\tc$ have been
denoted the same way, to avoid an overabundance of notations.

%
%
\begin{defn} Given a 2-group $\G$, any 2-category $\tc$, and
representations $\rho_1,\rho_2$ of $\G$ in $\tc$, an {\bf intertwiner}
$\phi \maps \rho_1 \to \rho_2$ is a pseudonatural transformation from
$\rho_1$ to $\rho_2$.
\end{defn}
%
%
This is analogous to the usual representation theory of groups, where
an intertwiner is a natural transformation between functors.  As
before, an intertwiner involves a morphism $\phi \maps V_1 \to V_2$ in
$\tc$.  However, as usual when passing from categories to
2-categories, this morphism is only required to satisfy the
commutation relations (\ref{commute}) {\it up to 2-isomorphism}.
In other words, whereas before the diagram
(\ref{intertwine}) commuted, so that the morphisms $\rho_2(g) \phi$
and $\phi \rho_1(g)$ were {\em equal}, here we only require that
there is a specified invertible 2-morphism $\phi(g)$ from one to the other.
(An invertible 2-morphism is called a `2-isomorphism'.)
The commutative square (\ref{intertwine}) for intertwiners
is thus generalized to:
\beq \label{1int}
\xymatrix{
  V_1\ar[rr]^{\rho_1(g)}\ar[dd]_{\phi}&&V_1 \ar[dd]^{\phi}\\
  \\
  V_2 \ar[rr]_{\rho_2(g)}&&V_2
  \ar@{<=}_{\phi(g)} "1,3"+<-3ex,-3ex>; "3,1"+<6ex,6ex>
}
\eeq
We say the commutativity of the diagram (\ref{intertwine}) has
been `weakened'.

In short, a intertwiner from $\rho_1$ to $\rho_2$
is really a pair consisting of a morphism $\phi \maps V_1 \to V_2$
together with a family of 2-isomorphisms
\beq
\phi(g) \maps \rho_2(g)\,\phi \stackrel\sim\longrightarrow
\phi\,\rho_1(g)
\eeq
one for each $g \in G$.
These data must satisfy some additional conditions in order
to be `pseudonatural':
\begin{itemize}
\item
$\phi$ should be compatible with the identity $1 \in G$:
\beq \label{norm1int}
\phi(1) = \unit_\phi
\eeq
where $\unit_\phi \maps \phi \to \phi$ is the identity 2-morphism.  Diagrammatically:
\[
\xymatrix{
  V_1\ar[rr]^{\unit_{V_1}}\ar[dd]_{\phi}&&V_1 \ar[dd]^{\phi}\\
  \\
  V_2 \ar[rr]_{\unit_{V_2}}&&V_2
  \ar@{<=}_{\phi(1)} "1,3"+<-4ex,-4ex>; "3,1"+<4ex,4ex>
}
\quad
\xymatrix{\\=\\}
\quad
\xymatrix{
  V_1\ar@/^7ex/[rrdd]^{\phi}\ar@/_7ex/[ddrr]_{\phi}&& \\
  \\
  & &V_2
  \ar@{<=}_{\unit_\phi} "1,3"+<-4ex,-4ex>; "3,1"+<4ex,4ex>
}
\]
\item $\phi$ should be compatible with composition of morphisms in $G$.
Intuitively, this means we should be able to glue $\phi(g)$ and
$\phi(g')$ together in the most obvious way, and obtain $\phi(g'g)$:
\beq
\label{1interdiag}
\xymatrix{
  V_1\ar[rr]^{\rho_1(g)}\ar[dd]_{\phi}&&V_1 \ar[rr]^{\rho_1(g')} \ar[dd]^{\phi}&&V_1 \ar[dd]^{\phi}\\
  \\
  V_2 \ar[rr]_{\rho_2(g)}&&V_2 \ar[rr]_{\rho_2(g')} && V_2
  \ar@{<=}_{\phi(g)} "1,3"+<-4ex,-4ex>; "3,1"+<4ex,4ex>
  \ar@{<=}_{\phi(g')} "1,5"+<-4ex,-4ex>; "3,3"+<4ex,4ex>
}
\quad
\xymatrix{\\=\\}
\quad
\xymatrix{
  V_1\ar[rr]^{\rho_1(g'g)}\ar[dd]_{\phi}&&V_1 \ar[dd]^{\phi}\\
  \\
  V_2 \ar[rr]_{\rho_2(g'g)}&&V_2
  \ar@{<=}_{\phi(g'g)} "1,3"+<-4ex,-4ex>; "3,1"+<4ex,4ex>
}
\eeq
To make sense of this equation we need the concept of `whiskering',
which we now explain. Suppose in any 2-category we have morphisms
$f_1,f_2 \maps x \to y$, a 2-morphism $\phi\maps f_1 \To f_2$, and a
morphism $g\maps y \to z$.  Then we can {\bf whisker} $\phi$ by $g$ by
taking the horizontal composite $\unit_{g} \circ \phi$, defining:
\[
\xymatrix{x \ar@/^3ex/[r]^{f_1}="f1" \ar@/_3ex/[r]_{f_2}="f2" & y \ar[r]^{g} & z
 \ar@{=>}^{\phi} "f1"+<0ex,-2.5ex>;"f2"+<0ex,2.5ex>
}
\qquad := \qquad
\xymatrix{x \ar@/^3ex/[r]^{f_1}="f1" \ar@/_3ex/[r]_{f_2}="f2" & y \ar@/^3ex/[r]^{g}="g1"
\ar@/_3ex/[r]_{g}="g2" & z
 \ar@{=>}^{\phi} "f1"+<0ex,-2.5ex>;"f2"+<0ex,2.5ex>
 \ar@{=>}^{\unit_{g}} "g1"+<0ex,-2.5ex>;"g2"+<0ex,2.5ex> }
\]
We can also whisker on the other side:
\[
\xymatrix{
x \ar[r]^{f} & y
\ar@/^3ex/[r]^{g_1}="f1" \ar@/_3ex/[r]_{g_2}="f2"
& z
 \ar@{=>}^{\phi} "f1"+<0ex,-2.5ex>;"f2"+<0ex,2.5ex>
}
\qquad := \qquad
\xymatrix{
x \ar@/^3ex/[r]^{f}="f1" \ar@/_3ex/[r]_{f}="f2" & y \ar@/^3ex/[r]^{g_1}="g1"
\ar@/_3ex/[r]_{g_2}="g2" & z
 \ar@{=>}^{\unit_{f}} "f1"+<0ex,-2.5ex>;"f2"+<0ex,2.5ex>
 \ar@{=>}^{\phi} "g1"+<0ex,-2.5ex>;"g2"+<0ex,2.5ex> }
\]

To define the 2-morphism given by the diagram on the left-hand side of
(\ref{1interdiag}), we whisker $\phi(g)$ on one side by $\rho_2(g')$,
whisker $\phi(g')$ on the other side by $\rho_1(g)$, and then vertically
compose the resulting 2-morphisms.
So, the equation in (\ref{1interdiag}) is a diagrammatic way of writing:
\beq \label{compatib}
\left[\phi(g') \circ \unit_{\rho_1(g)} \right]
\cdot
\left[\unit_{\rho_2(g')} \circ \phi(g) \right]
=
\phi(g'g)
\eeq
\item
Finally, the intertwiner $\phi$ should satisfy a higher-dimensional analogue of
diagram (\ref{intertwine}), so that it `intertwines' the 2-morphisms $\rho_1(u)$
and $\rho_2(u)$ where $u=(g,h)$ is a 2-morphism in the 2-group.  So, we demand
that the following ``pillow'' diagram commute for all $g \in G$ and $h\in H$:
\beq \label{pillowdiagram}
\xymatrix{
  V_1\ar[ddd]_{\phi}
    \ar@{}[ddd]^(0.85){}="fx"\ar@/^3ex/[rrr]^{\rho_1(g')}="a"\ar@/_3ex/[rrr]_{\rho_1(g)}="b"
  &&&V_1\ar[ddd]^{\phi}\ar@{}[ddd]_(0.15){}="fy"\\
  \\
  \\
  V_2\ar@{-->}@/^3ex/[rrr]^{\rho_2(g')}="c"\ar@/_3ex/[rrr]_{\rho_2(g)}="d"&&&V_2\\
  \ar@{<=}^{\rho_1(u)} "a"+<0pt,-2.5ex>;"b"+<0pt,2.5ex>
  \ar@{<:}^{\rho_2(u)} "c"+<0pt,-2.5ex>;"d"+<0pt,2.5ex>
  \ar@{} "fy";"c"|(0.3){}="f1"
  \ar@{} "fy";"c"|(0.7){}="f2"
  \ar@{} "b";"fx"|(0.3){}="b1"
  \ar@{} "b";"fx"|(0.7){}="b2"
  \ar@{<:} "f1";"f2"^{\phi(g')}
  \ar@{<=} "b1";"b2"_{\phi(g)}
}
\eeq
where we have introduced $g' = \d(h) g$. In other words:
\beq \label{pillow}
\left[ \unit_\phi \circ \rho_1(u) \right] \cdot \phi(g) =
\phi(g') \cdot \left[\rho_2(u) \circ \unit_\phi \right]
\eeq
where we have again used whiskering to glue together the 2-morphisms on
the front and top, and similarly the bottom and back.
\end{itemize}

Now a word about notation is required.  While an intertwiner from
$\rho_1$ to $\rho_2$ is really a pair consisting of a morphism $\phi
\maps V_1 \to V_2$ and a family of 2-morphisms $\phi(g)$, for
efficiency we refer to an intertwiner simply as $\phi$, and denote it by
$\phi \maps \rho_1 \to \rho_2$.  This should not cause any confusion.

So far, we have described representation of 2-groups as {\em
2-functors} and intertwiners as {\em pseudonatural transformations}.
As mentioned earlier, there are also things going between
pseudonatural transformations, called {\em modifications}.  The
following definition should thus come as no surprise:
%
%
\begin{defn}
Given a 2-group $\G$, a 2-category $\tc$, representations $\rho_1$ and
$\rho_2$ of $G$ in $\tc$,
and intertwiners $\phi,\psi \maps \rho \to \rho'$,
a {\bf 2-intertwiner} $m \maps \phi \To \psi$ is a modification from
$\phi$ to $\psi$.
\end{defn}
%
%
Let us say what modifications amount to in this case.
A modification $m \maps \phi \To \psi$ is a 2-morphism
\beq \label{2int}
  \xymatrix{
  V_1 \ar@/^2ex/[rr]^{\phi}="g1"\ar@/_2ex/[rr]_{\psi}="g2"&& V_2
  \ar@{=>}^{m} "g1"+<0ex,-2.5ex>;"g2"+<0ex,2.5ex>
}
\eeq
in $\tc$ such that the following pillow diagram:
\beq \label{2pillow}
\xymatrix{
  V_1\ar[rrr]^{\rho_1(g)}
    \ar@{}[rrr]^(0.85){}="fx"\ar@/_3ex/[ddd]_{\phi}="a"\ar@/^3ex/[ddd]^{\psi}="b"
  &&&V_1\ar@{-->}@/_3ex/[ddd]_{\phi}="c"\ar@/^3ex/[ddd]^{\psi}="d"\\
  \\
  \\
  V_2\ar[rrr]_{\rho_2(g)}\ar@{}[rrr]_(0.15){}="fy"&&&V_2\\
  \ar@{=>}^{m} "a"+<2.5ex,0pt>;"b"+<-2.5ex,0pt>
  \ar@{:>}^{m} "c"+<2.5ex,0pt>;"d"+<-2.5ex,0pt>
  \ar@{} "fy";"c"|(0.3){}="f1"
  \ar@{} "fy";"c"|(0.7){}="f2"
  \ar@{} "b";"fx"|(0.3){}="b1"
  \ar@{} "b";"fx"|(0.7){}="b2"
  \ar@{=>} "f1";"f2"^{\psi(g)}
  \ar@{:>} "b1";"b2"_{\phi(g)}
}
\eeq
commutes.  Equating the front and left with the back and right,
this means precisely that:
\beq \label{eq_pillow}
\psi(g) \cdot \left[\unit_{\rho_2(g)} \circ m \right] =
\left[m \circ \unit_{\rho_1(g)}\right] \cdot \phi(g)
\eeq
where we have again used whiskering to attach the morphisms
$\rho_i(g)$ to the 2-morphism $m$.

It is helpful to compare this diagram with the condition shown in
(\ref{pillowdiagram}).  One important difference is that in that case,
we had a ``pillow'' for each element $g \in G$ and $h \in H$, whereas
here we have one only for each $g \in G$.  For a intertwiner, the
pillow involves 2-morphisms between the maps given by representations.
Here the condition states that we have a fixed 2-morphism $m$ between
morphisms $I$ and $J$ between representation spaces, making the
given diagram commute for each $g$.  This is what representation
theory of ordinary groups would lead us to expect from an intertwiner.

\subsubsection{The 2-category of representations}
\label{sec:2Rep}

Just as any group $G$ gives a category ${\bf Rep}(G)$ with
representations as objects and intertwiners as morphisms, any 2-group
$\G$ gives a 2-category $\Rep(\G)$ with representations as objects,
intertwiners as morphisms, 2-intertwiners as 2-morphisms. It is worth
describing the structure of this 2-category explicitly.  In particular,
let us describe the rules for composing intertwiners and for
vertically and horizontally composing 2-intertwiners:

\begin{itemize}
\item
First, given a composable pair of intertwiners:
\[
\xymatrix{\rho_1 \ar[r]^{\phi} & \rho_2 \ar[r]^{\psi} & \rho_3}
\]
we wish to define their composite, which will be an intertwiner
from $\rho_1$ to $\rho_3$.  Recall that this intertwiner is a pair
consisting of a morphism $\xi \maps V_1 \to V_3$ in $\tc$ together
with a family of 2-morphisms $\xi(g)$.  We define $\xi$ to
be the composite $\psi \phi$,
and for any $g \in G$ we
define $\xi(g)$ by gluing together the diagrams (\ref{1int}) for
$\phi(g)$ and $\psi(g)$ in the obvious way:

\beq
\label{compose1int}
\newsavebox{\BoxA}
\savebox{\BoxA}{
\xymatrix{
  V_1\ar[rr]^{\rho_1(g)}\ar[dd]_{\phi}&&V_1\ar[dd]^{\phi} \\
  \\
  V_2 \ar[rr]_{\rho_2(g)}\ar[dd]_{\psi}&&V_2\ar[dd]^{\psi}
  \ar@{<=}_{\phi(g)} "1,3"+<-4ex,-4ex>; "3,1"+<4ex,4ex> \\
  \\
  V_3 \ar[rr]_{\rho_3(g)}&&V_2
  \ar@{<=}_{\psi(g)} "3,3"+<-4ex,-4ex>; "5,1"+<4ex,4ex>
}}
\newsavebox{\BoxB}
\savebox{\BoxB}{
\xymatrix{
  V_1\ar[rr]^{\rho_1(g)}\ar[dd]_{\xi}&&V_1 \ar[dd]^{\xi}\\
  \\
  V_3 \ar[rr]_{\rho_3(g)}&&V_3
  \ar@{<=}_{\xi(g)} "1,3"+<-4ex,-4ex>; "3,1"+<4ex,4ex>
}}
\xy
(-22,0)*{\usebox{\BoxB}};
(0,0)*{:=};
(22,0)*{\usebox{\BoxA}};
\endxy
\eeq
The diagram on the left hand side is once again evaluated with
the help of whiskering: we whisker $\phi(g)$ on one side by $\psi$
and $\psi(g)$ on the other side by $\phi$, then vertically compose the
resulting 2-morphisms.  In summary:
\beq
\label{compo_1int} \xi = \psi \phi,
\qquad
\xi(g) =
\left[\unit_\psi \circ \phi(g)\right]
\cdot \left[\psi(g) \circ \unit_\phi \right]
\eeq
By some calculations best done using diagrams, one
can check that these formulas define an intertwiner: relations
(\ref{1int}), (\ref{norm1int}), (\ref{1interdiag}) and
(\ref{pillowdiagram}) follow from the corresponding relations
for $\psi$ and $\phi$.

\item
Next, suppose we have a vertically composable pair of
2-intertwiners:
$$
\xymatrix{
   \rho_1\ar@/^4ex/[rr]^{\phi}="g1"\ar[rr]^(0.35){\psi}\ar@{}[rr]|{}="g2"
  \ar@/_4ex/[rr]_{\xi}="g3"&&\rho_2
  \ar@{=>}^{m} "g1"+<0ex,-2ex>;"g2"+<0ex,1ex>
  \ar@{=>}^{n} "g2"+<0ex,-1ex>;"g3"+<0ex,2ex>
}
$$
Then the 2-intertwiners $m$ and $n$ can be vertically composed using
vertical composition in $\tc$.  With some further calculations one
one check that the relation (\ref{eq_pillow}) for $n \cdot m \maps \phi
\To \xi$ follows from the corresponding relations for $m$ and $n$.

\item
Finally, consider a horizontally composable pair of 2-intertwiners:
$$
 \xymatrix{
  \rho_1 \ar@/^2ex/[rr]^{\phi}="g1"\ar@/_2ex/[rr]_{\phi'}="g2"&&\rho_2
  \ar@{=>}^{m} "g1"+<0ex,-2.5ex>;"g2"+<0ex,2.5ex>
  \ar@/^2ex/[rr]^{\psi}="g1"\ar@/_2ex/[rr]_{\psi'}="g2"&&\rho_3
  \ar@{=>}^{n} "g1"+<0ex,-2.5ex>;"g2"+<0ex,2.5ex>
}
$$
Then $m$ and $n$ can be composed using horizontal composition in
$\tc$.  With more calculations, one can check that the result
$n \circ m$ defines a 2-intertwiner: it satisfies
relation (\ref{eq_pillow}) because $n$ and $m$ satisfy the
corresponding relations.

\end{itemize}
All the calculations required
above are well-known in 2-category theory \cite{KS}.  Quite
generally, these calculations show that for {\it any}
2-categories ${\mathcal X}$ and ${\mathcal Y}$, there is
a 2-category with:
\begin{itemize}
\item 2-functors $\rho \maps {\mathcal X} \to {\mathcal Y}$
as objects,
\item pseudonatural transformations between these as morphisms,
\item modifications between these as 2-morphisms.
\end{itemize}
We are just considering the case ${\mathcal X} = \G$,
${\mathcal Y} = \tc$.

We conclude our description of $\Rep(\G)$ by discussing invertibility
for intertwiners and 2-intertwiners; this will allow us to introduce
natural equivalence relations for representations and intertwiners.

We first need to fill a small gap in our description of the 2-category
$\Rep(\G)$: we need to describe the identity morphisms and
2-morphisms.   Every representation $\rho$, with representation
space $V$, has its \textbf{identity intertwiner} given by the identity
morphism $\unit_V \maps V \to V$ in $\tc$, together with for each
$g$ the identity 2-morphism
$$\unit_{\rho(g)} \maps \rho(g) \unit_V \stackrel{\sim}{\longrightarrow}
\unit_V \rho(g)$$
Also, every intertwiner $\phi$ has its \textbf{identity 2-intertwiner},
given by the identity 2-morphism $\unit_{\phi}$ in $\tc$.

We define a 2-intertwiner $m \maps \phi \To \psi$
to be {\bf invertible} (for vertical composition) if there exists
$n \maps \psi \To \phi$ such that
\[
n \cdot m = \unit_{\phi} \quad \mbox{and} \quad m \cdot n = \unit_{\psi}
\]
Similarly, we define a intertwiner $\phi \maps \rho_1 \to \rho_2$
to be {\bf strictly invertible} if there exists an
intertwiner $\psi \maps \rho_2 \to \rho_1$ with
\beq \label{invert1mor}
\psi \phi = \unit_{\rho_1} \quad \mbox{and} \quad
\phi \psi = \unit_{\rho_2}
\eeq
However, it is better to relax the notion of invertibility for
intertwiners by requiring that the equalities (\ref{invert1mor})
hold only {\it up to invertible 2-intertwiners}.  In this case we say
that $\phi$ is {\bf weakly invertible}, or simply {\bf invertible}.

As for ordinary groups, we often consider equivalence classes of representations,
rather than representations themselves:

%
%
\begin{defn} \label{equivalrep}
We say that two representations $\rho_1$ and $\rho_2$ of a 2-group
are {\bf equivalent}, and write $\rho_1 \simeq \rho_2$, when there
exists a weakly invertible intertwiner between them.
\end{defn}
%
%
In the representation theory of 2-groups, however, where an extra layer
of categorical structure is added, it is also natural to consider
equivalence classes of intertwiners:
%
%
\begin{defn} \label{equival1int}
We say two intertwiners $\psi, \phi \maps \rho_1 \to \rho_2$
are {\bf equivalent}, and write $\phi \simeq \psi$, when there
exists an invertible 2-intertwiner between them.
\end{defn}
%
%

Sometimes it is useful to relax this notion of equivalence to 
include pairs of intertwiners that are not strictly parallel.  Namely, 
we call intertwiners $\phi \maps \rho_1 \to \rho_2$ and 
$\psi \maps \rho'_1 \to \rho'_2$ `equivalent' if there are invertible
intertwiners $\rho_i \to \rho'_i$ such that
\[
 \rho_1 \stackrel{\phi}{\to} \rho_2 \stackrel{\sim}{\to} \rho'_2
 \qquad \text{and} \qquad
  \rho_1 \stackrel{\sim}{\to}  \rho'_1 \stackrel{\psi}{\to} \rho'_2
\]
are equivalent,  in the sense of the previous definition.

A major task of 2-group representation theory is to
classify the representations and intertwiners up to equivalence.
Of course, one can only do this concretely after choosing a 2-category
in which to represent a given 2-group.  We turn to this task next.

%
\section{Measurable categories}
\label{2vs}
%

We have described the passage from groups to 2-groups, and from
representations to 2-{rep\-re\-sen\-ta\-tions}. Having presented these
definitions in a fairly abstract form, our next objective is to
describe a suitable target 2-category for representations of 2-groups.
Just as ordinary groups are typically represented on vector spaces,
2-groups can be represented on higher analogues called `2-vector
spaces'.  The idea of a 2-vector space can be formalized in several
ways.  In this section we describe the general idea of 2-vector
spaces, then focus on a particular formalism: the 2-category $\me$
defined by Yetter \cite{Yetter2}.

\subsection{From vector spaces to 2-vector spaces}
\label{sec:2vect}

To understand 2-vector spaces, it is helpful first to remember the
naive point of view on linear algebra that vectors are lists of
numbers, operators are matrices.  Namely, any finite dimensional
complex vector space is isomorphic to $\C^N$ for some natural number
$N$, and a linear map
\[
      T\maps \C^M \to \C^N
\]
is an $N\times M$ matrix of complex numbers $T_{n,m}$, where
$n\in\{1,\ldots,N\}$, $m\in\{1,\ldots,M\}$.  Composition of operators
is accomplished by matrix multiplication:
\[
       (UT)_{k,m} = \sum_{n=1}^N U_{k,n} T_{n,m}
\]
for $T\maps \C^M \to \C^N$ and $U\maps \C^N \to \C^K$.

As a setting for doing linear algebra, we can form a category whose objects
are just the sets $\C^N$ and whose morphisms are $N\times M$ matrices.  This
category is smaller than the category $\Vect$ of {\em all} finite dimensional
vector spaces, but it is {\em equivalent} to $\Vect$.   This is why
one can accomplish the same things with matrices as with abstract linear
maps---an oft used fact in practical computations.

Kapranov and Voevodsky \cite{KV} observed that we can `categorify'
this naive version of the category of vector spaces and define a
2-category of `2-vector spaces'.  When we categorify a concept, we
replace sets with categories.  In this case, we replace the set $\C$
of complex numbers, along with its usual product and sum operations,
by the category $\Vect$ of complex vector spaces, with its tensor
product and direct sum.  Thus a `2-vector' is a list, not of numbers,
but of vector spaces.  Since we can define maps between such lists
they form, not just a set, but a category: a `2-vector space'.  A
morphism between 2-vector spaces is a matrix, not of numbers, but of
vector spaces.  We also get another layer of structure: {\it
2-morphisms}.  These are matrices of linear maps.

More precisely, there is a 2-category denoted $\twoVe$ defined as follows:

\subsubsection*{Objects}

The objects of $\twoVe$ are the categories
\[
       \Vect^0, \Vect^1, \Vect^2, \Vect^3, \ldots
\]
where $\Vect^N$ denotes the $N$-fold cartesian product.  Note in
particular that the zero-dimensional 2-vector space $\Vect^0$ has just
one object and one morphism.

\subsubsection*{Morphisms}

Given 2-vector spaces $\Vect^M$ and $\Vect^N$, a morphism
\[
     T\maps \Vect^M \to \Vect^N
\]
is given by an $N\times M$ matrix of complex vector spaces $T_{n,m}$,
where $n\in\{1,\ldots,N\}$, $m\in\{1,\ldots,M\}$.  Composition is
accomplished by matrix multiplication, as in ordinary linear algebra,
but using tensor product and direct sum:
\beq
\label{matmult}
       (UT)_{k,m} = \bigoplus_{n=1}^N U_{k,n}\otimes T_{n,m}
\eeq
for $T\maps \Vect^M \to \Vect^N$ and $U\maps \Vect^N \to \Vect^K$.

\subsubsection*{2-Morphisms}

Given morphisms $T,T'\maps \Vect^M \to \Vect^N$, a 2-morphism $\alpha$
between these:
\[
  \xymatrix{
  \Vect^M\ar@/^2ex/[rr]^{T}="g1"\ar@/_2ex/[rr]_{T'}="g2"&& \Vect^N
  \ar@{=>}^{\alpha} "g1"+<0ex,-2.5ex>;"g2"+<0ex,2.5ex>
}
\]
is an $N\times M$ matrix of linear maps of vector spaces, with components
\[
       \alpha_{n,m}\maps T_{n,m} \to T'_{n,m}.
\]
Such 2-morphisms can be composed \textit{vertically}:
\[
\xymatrix{
 \Vect^M\ar@/^5ex/[rr]^{T}="g1"\ar[rr]^(0.35){T'}\ar@{}[rr]|{}="g2"
  \ar@/_5ex/[rr]_{T''}="g3"&&\Vect^N
  \ar@{=>}^{\alpha} "g1"+<0ex,-2ex>;"g2"+<0ex,1ex>
  \ar@{=>}^{\alpha'} "g2"+<0ex,-1ex>;"g3"+<0ex,2ex>
}
\]
simply by composing componentwise the linear maps:
\beq\label{vertcomp}
     (\alpha' \cdot \alpha)_{n,m} = {\alpha'}_{n,m}\alpha_{n,m}.
\eeq
They can also be composed \textit{horizontally}:
\[
  \xymatrix{
  \Vect^N\ar@/^2ex/[rr]^{T}="g1"\ar@/_2ex/[rr]_{T'}="g2"&& \Vect^M
  \ar@/^2ex/[rr]^{U}="g3"\ar@/_2ex/[rr]_{U'}="g4"&& \Vect^K
  \ar@{=>}^{\alpha} "g1"+<0ex,-2.5ex>;"g2"+<0ex,2.5ex>
  \ar@{=>}^{\beta} "g3"+<0ex,-2.5ex>;"g4"+<0ex,2.5ex>
}
\]
analogously with (\ref{matmult}), by using `matrix multiplication'
with respect to tensor product and direct sum of maps:
\beq\label{horcomp} (\beta\circ \alpha)_{k,m} = \bigoplus_{n=1}^N
\beta_{k,n}\otimes \alpha_{n,m}.  \eeq

While simple in spirit, this definition of $\twoVe$ is problematic for
a couple of reasons.  First, composition of morphisms is not strictly
associative, since the direct sum and tensor product of vector spaces
satisfy the associative and distributive laws only up to isomorphism,
and these laws are used in proving the associativity of matrix
multiplication.  So, $\twoVe$ as just defined is not a 2-category, but
only a `weak' 2-category, or `bicategory'.  These are a bit more
complicated, but luckily any bicategory is equivalent, in a precise
sense, to some 2-category.  The next section gives a concrete
description of a such a 2-category.  (See also the work of Elgueta
\cite{Elgueta1}.)

The above definition of $\twoVe$ is also somewhat naive, since it
categorifies a naive version of $\Vect$ where the only vector spaces
are those of the form $\C^N$.  A more sophisticated approach involves
`abstract' 2-vector spaces.  One can define these axiomatically by
listing properties of a category that guarantee that it is equivalent
to $\Vect^N$ (see Def.\ 2.12 in \cite{Neuchl}, and also \cite{Yetter1}).
A cruder way to accomplish the same effect is to {\it define} an
abstract 2-vector space to be a category equivalent to $\Vect^N$.  We
take this approach in the next section, because we do not yet know an
axiomatic approach to measurable categories, and we wish to prepare
the reader for our discussion of those.

\subsection{Categorical perspective on 2-vector spaces}
\label{2vectcat}

In this section we give a definition of $\twoVe$ which involves
treating it as a sub-2-category of the 2-category $\Cat$, in which
objects, morphisms, and 2-morphisms are categories, functors, and
natural transformations, respectively.  This approach addresses both
problems mentioned at the end of the last subsection.  Similar
ideas will be very useful in our study of measurable categories
in the sections to come.

In this approach the objects of $\twoVe$ are `linear categories'
that are `linearly equivalent' to $\Vect^N$ for some $N$.  The
morphisms are `linear functors' between such categories, and the
2-morphisms are natural transformations.

Let us define the three quoted terms.  First, a {\bf linear category}
is a category where for each pair of objects $x$ and $y$, the set of
morphisms from $x$ to $y$ is equipped with the structure of a
finite-dimensional complex vector space, and composition of morphisms
is a bilinear operation.  For example, $\Vect^N$ is a linear category.

Second, a functor $F \maps \VV \to \VV'$ between linear categories is a
{\bf linear equivalence} if it is an equivalence that maps morphisms
to morphisms in a linear way.  We define a {\bf 2-vector space} to be
a linear category that is linearly equivalent to $\Vect^N$ for some
$N$.  For example, given a category $V$ and an equivalence $F \maps \VV
\to \Vect^N$, we can use this equivalence to equip $\VV$ with the
structure of a linear category; then $F$ becomes a linear equivalence
and $V$ becomes a 2-vector space.

Third, note that any $N \times M$ matrix of
vector spaces $T_{n,m}$ gives a functor $T \maps \Vect^M \to \Vect^N$
as follows.  For an object $V\in \Vect^M$, we define $TV \in \Vect^N$
by
\[
       (TV)_{n} = \bigoplus_{m=1}^M T_{n,m} \otimes V_m.
\]
For a morphism $\phi$ in $\Vect^M$, we define $T\phi$ by:
\[
       (T\phi)_{n} = \bigoplus_{m=1}^M \unit_{T_{n,m}} \otimes \phi_m
\]
where $ \unit_{T_{n,m}} $ denotes the identity map on the vector space
$T_{n,m}$.  It is straightforward to check that these operations
define a functor.  We call such a functor from $\Vect^N$ to $\Vect^M$
a {\bf matrix functor}.  More generally, given 2-vector spaces $\VV$
and $\VV'$, we define a \textbf{linear functor} from $\VV$ to $\VV'$
to be any functor naturally isomorphic to a composite
\[
\xymatrix{
   \VV \ar[r]^(.4){F}&
   \Vect^M \ar[r]^{T}&
   \Vect^N \ar[r]^(.6){G}&
   \VV'
}
\]
where $T$ is a matrix functor and $F,G$ are linear equivalences.

These definitions may seem complicated, but unlike the naive
definitions in the previous section, they give a 2-category:

\begin{theo} \label{twoVe} There is a sub-2-category $\twoVe$
of $\Cat$ where the objects are 2-vector spaces, the morphisms
are linear functors, and the 2-morphisms are natural transformations.
\end{theo}

The proof of this result will serve as the pattern for a similar
argument for measurable categories.  We break it into a series of
lemmas.  It is easy to see that identity functors and identity natural
transformations are linear.  It is obvious that natural
transformations are closed under vertical and horizontal composition.
So, we only need to check that linear functors are closed under
composition.  This is Lemma \ref{composition.3}.

\begin{lemma}
\label{composition.0}
A composite of matrix functors is naturally isomorphic to a matrix
functor.
\end{lemma}

\begin{proof} Suppose $T\maps
\Vect^M \to \Vect^N$ and $U\maps \Vect^N \to \Vect^K$ are matrix
functors.  Their composite $UT$ applied to an object $V \in \Vect^M$ gives
an object $UTV$ with components
\[
      (UTV)_k  = \bigoplus_{n=1}^N U_{k,n} \otimes
                  \left( \bigoplus_{m=1}^M T_{n,m} \otimes V_m\right)
\]
but this is naturally isomorphic to
\[
     \bigoplus_{m=1}^M
     \left( \bigoplus_{n=1}^N U_{k,n}\otimes T_{n,m} \right) \otimes V_{m}
\]
so $UT$ is naturally isomorphic to the matrix functor defined by
formula (\ref{matmult}).
\end{proof}

\begin{lemma}
\label{composition.1}
If $F \maps \Vect^N \to \Vect^M$ is a linear equivalence,
then $N = M$ and $F$ is a linear functor.
\end{lemma}

\begin{proof}
Let $e_i$ be the standard basis for $\Vect^N$:
\[       e_i = (0, \dots,\underbrace{\C}_{\mbox{$i$th place}}, \dots, 0 ) . \]
Since an equivalence maps indecomposable objects to indecomposable
objects, we have $F(e_i) \cong e_{\sigma(i)}$ for some function
$\sigma$.  This function must be a permutation, since $F$ has a weak
inverse.  Let $\tilde{F}$ be the matrix functor corresponding to the
permutation matrix associated to $\sigma$.  One can check that $F$ is
naturally isomorphic to $\tilde{F}$, hence a linear functor.  Checking
this makes crucial use of the fact that $F$ be a {\it linear}
equivalence: for example, taking the complex conjugate of a vector
space defines an equivalence $K \maps \Vect \to \Vect$ that is not
a matrix functor.  We leave the details to the reader.
\end{proof}

\begin{lemma}
\label{composition.2}
If $T\maps \VV \to \VV'$ is a linear functor and $F\maps \VV \to \Vect^M$,
$G\maps \Vect^N \to \VV'$ are {\rm arbitrary} linear equivalences, then
$T$ is naturally isomorphic to the composite
\[
\xymatrix{
   \VV \ar[r]^(.4){F}&
   \Vect^M \ar[r]^{\tilde{T}}&
   \Vect^N \ar[r]^(.6){G}&
   \VV'
}
\]
for some matrix functor $\tilde{T}$.
\end{lemma}

\begin{proof}
Since $T$ is linear we know there exist linear equivalences $F'\maps \VV
\to \Vect^{M'}$ and $G'\maps \Vect^{N'} \to \VV'$ such that $T$ is
naturally isomorphic to the composite
\[
\xymatrix{
   \VV \ar[r]^(.4){F'}&
   \Vect^{M'} \ar[r]^{\tilde{T}'}&
   \Vect^{N'} \ar[r]^(.6){G'}&
   \VV'
}
\]
for some matrix functor $\tilde{T}'$.  We have $M' = M$ and
$N' = N$ by Lemma \ref{composition.1}.  So, let $\tilde{T}$ be
the composite
\[
\xymatrix{
   \Vect^M \ar[r]^(.6){\bar{F}}&
   \VV \ar[r]^(.4){F'}&
   \Vect^M \ar[r]^{\tilde{T}'}&
   \Vect^N \ar[r]^(.6){G'}&
   \VV' \ar[r]^(.4){\bar{G}}&
  \Vect^N
}
\]
where $\bar{F}$ and $\bar{G}$ are weak inverses for $F$ and $G$.
Since $F' \bar{F} \maps \Vect^M \to \Vect^M$ and $\bar{G} G' \maps
\Vect^N \to \Vect^N$ are linear equivalences, they are naturally
isomorphic to matrix functors by Lemma \ref{composition.1}.  Since
$\tilde{T}$ is a composite of functors that are naturally isomorphic
to matrix functors, $\tilde{T}$ itself is naturally isomorphic to a
matrix functor by Lemma \ref{composition.0}.  Note that the composite
\[
\xymatrix{
   \VV \ar[r]^(.4){F}&
   \Vect^M \ar[r]^{\tilde T}&
   \Vect^N \ar[r]^(.6){G}&
   \VV'
}
\]
is naturally isomorphic to $T$.  Since $F$ and $G$ are linear equivalences
and $\tilde{T}$ is naturally isomorphic to a matrix functor,
it follows that $T$ is a linear functor.
\end{proof}

\begin{lemma}
\label{composition.3}
A composite of linear functors is linear.
\end{lemma}

\begin{proof}
Suppose we have a composable pair of linear functors
$T\maps \VV \to \VV'$ and $U\maps \VV' \to \VV''$.  By definition,
$T$ is naturally isomorphic to a composite
\[
\xymatrix{
   \VV \ar[r]^(.4){F}&
   \Vect^L \ar[r]^{\tilde{T}}&
   \Vect^M \ar[r]^(.6){G}&
   \VV'
}
\]
where $\tilde{T}$ is a matrix functor, and $F$ and $G$ are linear
equivalences.  By Lemma \ref{composition.2}, $U$ is naturally isomorphic
to a composite
\[
\xymatrix{
   \VV' \ar[r]^(.4){\bar{G}}&
   \Vect^M \ar[r]^{\tilde{U}}&
   \Vect^N \ar[r]^(.55){H}&
   \VV''
}
\]
where $\tilde{U}$ is a matrix functor, $\bar{G}$ is a weak inverse for
$G$, and $H$ is a linear equivalence.  The composite $UT$ is thus
naturally isomorphic to
\[
\xymatrix{
   \VV \ar[r]^(.4){F}&
   \Vect^L \ar[r]^{\tilde{U}\tilde{T}}&
   \Vect^N \ar[r]^(.55){H}&
   \VV''
}
\]
Since $\tilde{U}\tilde{T}$ is naturally isomorphic to a matrix functor by
Lemma \ref{composition.0}, it follows that $UT$ is a linear functor.
\end{proof}

These results justify the naive recipe for composing 1-morphisms using
matrix multiplication, namely equation (\ref{matmult}).  First,
Lemma \ref{composition.0} shows that the composite of matrix functors
is naturally isomorphic to their matrix product as given by equation
(\ref{matmult}).  More generally, given any linear functors
$T \maps \Vect^L \to \Vect^M$ and $U \maps
\Vect^M \to \Vect^N$, we can choose matrix functors naturally isomorphic
to these, and the composite $UT$ will be naturally isomorphic to the
matrix product of these matrix functors.  Finally, we can reduce the
job of composing linear functors between {\it arbitrary} 2-vector spaces
to matrix multiplication by choosing linear equivalences between
these 2-vector spaces and some of the form $\Vect^N$.

Similar results hold for natural transformations.  Any $N \times M$
matrix of linear operators $\alpha_{n,m} \maps T_{n,m} \to T'_{n,m}$
determines a natural transformation between the matrix functors
$T,T'\maps \Vect^M \to \Vect^N$.  This natural transformation gives,
for each object $V\in \Vect^M$, a morphism $\alpha_{{}_V}\maps TV \to
T'V$ with components
\[
     (\alpha_{{}_V})_n\maps
     \bigoplus_{m=1}^M T_{n,m} \otimes V_m \to
     \bigoplus_{m=1}^M  T'_{n,m} \otimes V_m
\]
given by
\[
    (\alpha_{{}_V})_n = \bigoplus_{m=1}^M \alpha_{n,m} \otimes \unit_{V_m}.
\]
We call a natural transformation of this sort a {\bf matrix natural
transformation}.  However:

\begin{theo}
\label{composition.4}
Any natural transformation between matrix functors is a matrix natural
transformation.
\end{theo}

\begin{proof}
Given matrix functors $T,T'\maps \Vect^M \to
\Vect^N$, a natural transformation $\alpha \maps T \To T'$ gives
for each basis object $e_m \in \Vect^M$ a morphism in $\Vect^N$
with components
\[    (\alpha_{e_m})_n \maps T_{n,m} \otimes \C \to T'_{n,m} \otimes \C. \]
Using the natural isomorphism between a vector space and that vector
space tensored with $\C$, these can be reinterpreted as operators
\[     \alpha_{n,m} \maps T_{n,m} \to T'_{n,m}  .\]
These operators define a matrix natural transformation from $T$ to $T'$,
and one can check using naturality that this equals $\alpha$.
\end{proof}

One can check that vertical composition of matrix natural
transformations is given by the matrix formula of the previous
section, namely formula (\ref{vertcomp}).  Similarly, the horizontal
composite of matrix natural transformations is `essentially' given by
formula (\ref{horcomp}).  So, while these matrix formulas are a bit
naive, they are useful tools when properly interpreted.

\subsection{From 2-vector spaces to measurable categories}
\label{measurable_categories}

In the previous sections, we saw the 2-category $\twoVe$ of
Kapranov--Voevodsky 2-vector spaces as a categorification of $\Vect$,
the category of finite-dimensional vector spaces.  While one can
certainly study representations of 2-groups in $\twoVe$
\cite{BarrettMackaay,Elgueta2}, our goal is to describe
representations of 2-groups in something more akin to
infinite-dimensional 2-{\em Hilbert} spaces.  Such objects should be
roughly like `$\Hilb^X$', where $\Hilb$ is the category of Hilbert
spaces and $X$ may now be an infinite index set.  In fact, for our
purposes, $X$ should have at least the structure of a measurable
space.  This allows one to categorify Hilbert spaces $L^2(X,\mu)$ in
such a way that measurable functions are replaced by `measurable
fields of Hilbert spaces', and integrals of functions are replaced by
`direct integrals' of such fields.

We can construct a chart like the one in the introduction, outlining the
basic strategy for categorification:

\vskip 1em
\begin{center}
{\small
\begin{tabular}{c|c}                              \hline
ordinary             &   higher              \\
$L^2$ spaces        &  $L^2$ spaces       \\     \hline
                    &                        \\
$\C$        &  $\Hilb$                       \\
$+$         &  $\oplus$                      \\
$\times$    &  $\otimes$                     \\
$0$         &  $\{0\}$                       \\
$1$         &  $\C$                          \\
measurable functions       &  measurable fields of Hilbert spaces                     \\
$\int$ \quad (integral)    &  $\direct$ \quad (direct integral)    \\
                    &                        \\
\hline
\end{tabular}} \vskip 1em
\end{center}

Various alternatives spring from this basic idea. In this section and
the following one, we provide a concrete description of one possible
categorification of $L^2$ spaces: `measurable categories' as defined
by Yetter \cite{Yetter2}, which provide a foundation for earlier work
by Crane, Sheppeard, and Yetter \cite{CraneSheppeard, CraneYetter}.

Measurable categories do not provide a full-fledged categorification
of the concept of Hilbert space, so they do not deserve to be called
`2-Hilbert spaces'.  Indeed, {\it finite-dimensional} 2-Hilbert spaces
are well understood \cite{Baez2,Bartlett}, and they have a bit more
structure than measurable categories with a finite basis of objects.
Namely, we can take the `inner product' of two objects in such a
2-Hilbert space and get a Hilbert space.  We expect something similar
in an infinite-dimensional 2-Hilbert space, and it happens in many
interesting examples, but the definition of measurable category lacks
this feature.  So, our work here can be seen as a stepping-stone
towards a theory of unitary representations of 2-groups on
infinite-dimensional 2-Hilbert spaces.  See Section \ref{conclusion}
for a bit more on this issue.

The goal of this section is to construct a 2-category of measurable
categories, denoted $\me$.  This requires some work, in part because
we do not have an intrinsic characterization of measurable categories.
We also give concrete practical formulas for composing
morphisms and 2-morphisms in $\me$.  This will equip the reader with the
tools necessary for calculations in the representation theory
developed in Section \ref{MeasRep}.  But first we need some
preliminaries in analysis.  For basic results and standing
assumptions the reader may also turn to Appendix \ref{tools}.

%
\subsubsection{Measurable fields and direct integrals}
%

We present here some essential analytic tools: measurable fields of
Hilbert spaces and operators, their measure-classes and direct
integrals, and measurable families of measures.

We have explained the categorical motivation for generalizing {\em
functions} on a measurable space to `{\em fields of Hilbert spaces}'
on a measurable space.  But one cannot simply assign an arbitrary
Hilbert space to each point in a measurable space $X$ and expect to
perform operations that make good analytic sense.  Fortunately,
`measurable fields' of Hilbert spaces have been studied in detail---see
especially the book by Dixmier \cite{Dixmier}.  Algebraists may
view these as representations of abelian von Neumann algebras on
Hilbert spaces, as explained by Dixmier and also Arveson \cite[Chap.\
2.2]{Arveson}.  Geometers may instead prefer to view them as
`measurable bundles of Hilbert spaces', following the treatment of
Mackey \cite{Mackey1978}.  Measurable fields of Hilbert spaces
have also been studied from a category-theoretic perspective by
Yetter \cite{Yetter2}.

It will be convenient to impose some simplifying assumptions.  Our
measurable spaces will all be `standard Borel spaces' and our measures
will always be $\sigma$-finite and positive.  Standard Borel spaces
can be characterized in several ways:

\begin{lemma}
\label{lem:standard_Borel} Let $(X,\mathcal{B})$ be a measurable space,
i.e.\ a set $X$ equipped with a $\sigma$-algebra of subsets $\mathcal{B}$.
Then the following are equivalent:
\begin{enumerate}
\item $X$ can be given the structure of a separable complete metric
space in such a way that $\mathcal{B}$ is the $\sigma$-algebra of Borel
subsets of $X$.
\item $X$ can be given the structure of a second-countable, locally
compact Hausdorff space in such a way that $\mathcal{B}$ is the
$\sigma$-algebra of Borel subsets of $X$.
\item $(X,\mathcal{B})$ is isomorphic to one of the following:
\begin{itemize}
\item a finite set with its $\sigma$-algebra of all subsets;
\item a countably infinite set with its $\sigma$-algebra of all subsets;
\item $[0,1]$ with its $\sigma$-algebra of Borel subsets.
\end{itemize}
A measurable space satisfying any of these
equivalent conditions is called a {\bf standard Borel space}.
\end{enumerate}
\end{lemma}

\begin{proof} It is clear that 3) implies 2).  To see that 2)
implies 1), we need to check that every second-countable locally
compact Hausdorff space $X$ can be made into a separable complete
metric space.  For this, note that the one-point compactification of
$X$, say $X^+$, is a second-countable compact Hausdorff space, which
admits a metric by Urysohn's metrization theorem.  Since $X^+$ is
compact this metric is complete.  Finally, any open subset of
separable complete metric space can be given a new metric giving it
the same topology, where the new metric is separable and complete
\cite[Chap.\ IX, \S 6.1, Prop.\ 2]{Bourbaki}.  Finally, that 1)
implies 3) follows from two classic results of Kuratowski.
Namely: two standard Borel spaces (defined using condition 1)
are isomorphic if and only if they have the same cardinality,
and any uncountable standard Borel space has the cardinality
of the continuum \cite[Chap.\ I, Thms.\ 2.8 and 2.13]{Parthasarathy}.
\end{proof}

The following definitions will be handy:

\begin{defn}
\label{defn:measurable_space}
By a {\bf measurable space} we mean a standard Borel
space $(X,\mathcal{B})$.  We call sets in $\mathcal{B}$ {\bf measurable}.
Given spaces $X$ and $Y$, a map $f \maps X \to Y$ is {\bf measurable}
if $f^{-1}(S)$ is measurable whenever $S \subseteq Y$ is measurable.
\end{defn}

\begin{defn}
\label{defn:measure}
By a {\bf measure} on a measurable space $(X, \mathcal{B})$ we mean a
$\sigma$-finite measure, i.e.\ a countably additive map $\mu \maps
\mathcal{B} \to [0,+\infty]$ for which $X$ is a countable union of
$S_i \in \mathcal{B}$ with $\mu(S_i) < \infty$.
\end{defn}

A key idea is that a measurable field of Hilbert spaces should
know what its `measurable sections' are.  That is, there should be
preferred ways of selecting one vector from the Hilbert space at each
point; these preferred sections should satisfy some properties, given
below, to guarantee reasonable measure-theoretic behavior:

\begin{defn} Let $X$ be a measurable space.  A  {\bf measurable field
of Hilbert spaces} $\H$ on $X$ is an assignment of a Hilbert space $\H_x$
to each $x\in X$, together with a subspace $\M_\H \subseteq \prod_x \H_x$
called the {\bf measurable sections} of $\H$, satisfying the properties:
\begin{itemize}
\item $\forall \xi \in \M_\H$, the function $x\mapsto \| \xi_x \|_{\H_x}$ is
measurable.
\item For any $\eta \in \prod_x \H_x$ such that $x\mapsto \langle
\eta_x,\xi_x\rangle_{\H_x}$ is measurable for all $\xi \in \M_\H$, we
have $\eta \in \M_\H$.
\item There is a sequence $\xi_i\in \M_\H$ such that
$\{(\xi_i)_x\}_{i=1}^\infty$ is dense in $\H_x$ for all $x\in X$.
\end{itemize}
\end{defn}

\begin{defn}
Let $\H$ and $\H'$ be measurable fields of Hilbert spaces on $X$.
A {\bf measurable field of bounded linear operators} $\phi \maps \H
\to \K$ on $X$ is an $X$-indexed family of bounded operators
$\phi_x\maps \H_x \to \H_x'$ such that $\xi\in \M_\H$ implies
$\phi(\xi)\in \M_{\H'}$, where $\phi(\xi)_x := \phi_x(\xi_x)$.
\end{defn}

Given a positive measure $\mu$ on $X$, measurable fields can be
integrated. The integral of a function gives an element of $\C$; the
integral of a field of Hilbert spaces gives an object of $\Hilb$.
Formally, we have the following definition:
\begin{defn} \label{MeasHilb}
Let $\H$ be a measurable field of Hilbert spaces on a measurable space
$X$; let $\langle \cdot,\cdot \rangle_x$ denote the inner product in
$\H_x$, and $\|\cdot\|_x$ the induced norm.  The {\bf direct integral}
\[
     \direct_X \extd\mu(x)\, \H_x
\]
of $\H$ with respect to the measure $\mu$ is the Hilbert space of all
$\mu$-\alme\ equivalence classes of measurable $L^2$ sections of $\H$,
that is, sections $\psi\in \M_{\H}$ such that
     \[
       \int_X \extd\mu(x)\, \| \psi_x \|^2_{x} < \infty,
     \]
with inner product given by
  \[
       \langle\psi,\psi'\rangle =
  \int_X \extd\mu(x)\, \langle \psi_x,\psi'_x \rangle_x.
  \]
for $\psi,\psi' \in \direct_X \! \extd\mu\, \H$.
\end{defn}
That the inner product is well defined for $L^2$ sections follows by
polarization.  Of course, for $\direct_X \! \extd\mu\, \H$ to be a
Hilbert space as claimed in the definition, one must also check that
it is Cauchy-complete with respect to the induced norm.  This is
indeed the case \cite[Part II Ch.~1 Prop.~5]{Dixmier}.  We often
denote an element of the direct integral of $\H$ by
\[
       \direct_X \extd\mu(x) \, \psi_x
\]
where $\psi_x\in \H_x$ is defined up to $\mu$-\alme\ equality.

We also have a corresponding notion of direct integral for fields of
linear operators:
\begin{defn}
Suppose $\phi\maps \H \to \H'$ is a $\mu$-essentially bounded
measurable field of linear operators on $X$.  The {\bf direct
integral} of $\phi$ is the linear operator acting pointwise on
sections:
\[
\begin{array}{cccc}
\displaystyle \direct_X \extd\mu(x)\, \phi_x \maps&
 \displaystyle \direct_X \extd\mu(x)\, \H_x &\to&
 \displaystyle \direct_X \extd\mu(x)\, \H'_x \\ \\
 &  \direct_X \extd\mu(x)\, \psi_x &\mapsto&
      \direct_X \extd\mu(x)\, \phi_x(\psi_x)
\end{array}
\]
\end{defn}
Note requiring that the field be $\mu$-essentially bounded---{\it
i.e.}\ that the operator norms $\|\phi_x\|$ have a common bound for
$\mu$-almost every $x$---guarantees that the image lies in the direct
integral of $\H'$, since
\[
   \int_X \extd\mu(x)\, \|\phi_x(\psi_x)\|^2_{\H'_x} \;\leq\;
   {\rm ess} \, \sup_{x'}
\|\phi_{x'}\|^2 \int \extd\mu(x)\, \|\psi_x\|^2_{\H_x} \;<\; \infty.
\]
Notice that direct integrals indeed generalize direct sums: in the
case where $X$ is a finite set and $\mu$ is counting measure, direct
integrals of Hilbert spaces and operators simply reduce to direct
sums.

In ordinary integration theory, one typically identifies functions
that coincide almost everywhere with respect to the relevant measure.
This is also useful for the measurable fields defined above, for the
same reasons.  To make `\alme-equivalence of measurable fields'
precise, we first need a notion of `restriction'.

If $A\subseteq X$ is a measurable set, any measurable field $\H$ of
Hilbert spaces on $X$ induces a field $\H|_{A}$ on $A$, called the
{\bf restriction} of $\H$ to $A$.  The restricted field is constructed
in the obvious way: we let $(\H|_A)_x =\H_x$ for each $x\in A$, and
define the measurable sections to be the restrictions of measurable
sections on $X$: $\M_{\H|_A} = \{\psi|_A : \psi \in \M_\H\}$.  It is
straightforward to check that $(\H|_A, \M_{\H|_A})$ indeed defines a
measurable field.\footnote{The first and third axioms in the
definition are obvious.  To check the second, pick $\eta \in
\prod_{x\in A} \H_x$ such that $x \mapsto \langle \eta_x,
\xi_x\rangle$ is a measurable function on $A$ for every $\xi \in
\M|_{\H|_A}$.  Extend $\eta$ to $\tilde \eta \in \prod_{x\in X} \H_x$
by setting
\[
   \tilde \eta_x =
\left\{ \begin{array}{cc} \eta_x & x\in A \\
0 & x\not\in A. \end{array}\right.
\]
}

Similarly, if $\phi\maps \H \to \K$ is a field of linear operators,
its {\bf restriction} to a measurable subset $A\subseteq X$ is the
obvious $A$-indexed family of operators \( \phi|_A \maps \H|_A \to
\K|_A \) given by $(\phi|_A)_x = \phi_x$ for each $x$ in $A$.  It is
easy to check that $\xi\in \M_{\H|_A}$ implies $\phi(\xi)\in
\M_{\K|_A}$, so $\phi|_A$ defines a measurable field on $A$.

We say two measurable fields of Hilbert spaces on $X$ are {\bf \boldmath
$\mu$-almost everwhere equivalent} if they have equal restrictions to
some measurable $A\subseteq X$ with $\mu(X-A) = 0$.  This is obviously
an equivalence relation, and an equivalence class is called a {\bf \boldmath
$\mu$-class of measurable fields}.  Two fields in the same $\mu$-class
have canonically isomorphic direct integrals, so the direct integral
of a $\mu$-class makes sense.

Equivalence classes of measurable fields of linear operators work
similarly, but with one subtlety.  First suppose we have two
measurable fields of Hilbert spaces $T_x$ and $U_x$ on $X$, and a
measurable field of operators $\alpha_x\maps T_x \to U_x$.  Given a
measure $\mu$, one can clearly identify two such $\alpha$ if they
coincide outside a set of $\mu$-measure 0, thus defining a notion of
{\bf \boldmath $\mu$-class of fields of operators} from $T$ to $U$.
So far $T$ and $U$ are fixed, but now we wish to take equivalence
classes of them as well.  In fact, it is often useful to pass to
$t$-classes of $T$ and $u$-classes of $U$, where $t$ and $u$ are in
general {\em different} measures on $X$.  We then ask what sort of
measure $\mu$ must be for the $\mu$-class of $\alpha$ to pass to a
well defined map
\[
   [\alpha_x]_\mu \maps [T_x]_t \to [U_x]_u,
\]
where brackets denote the relevant classes.  This works if and only if
each $t$-null set and each $u$-null set is also $\mu$-null.  Thus we
require \beq
\label{abso-cts}
    {\mu} \ll t \quad \text{ and }\quad \mu \ll u, \eeq where `$\ll$'
denotes absolute continuity of measures.  Given a measure $\mu$
satisfying these properties, it makes sense to speak of the {\bf
$\mu$-class of fields of operators} from a $t$-class of fields of
Hilbert spaces to a $u$-class of fields of Hilbert spaces.  In
practice, one would like to pick $\mu$ to be maximal with respect to
the required properties (\ref{abso-cts}), so that $\mu$-\alme\
equivalence is the transitive closure of $u$-\alme\ and $t$-\alme\
equivalences.

In fact, if $t$ and $u$ are both $\sigma$-finite measures, there is a
natural choice for which measure $\mu$ to take in the above
construction: the `geometric mean measure' $\sqrt{tu}$ of the measures
$t$ and $u$.  The notion of geometric mean measure is discussed in
Appendix \ref{apx:gmean}, but the basic idea is as follows.  If $t$ is
absolutely continuous with respect to $u$, denoted $t\ll u$, then we
have the Radon--Nikodym derivative $\rnd{t}{u}$.  More generally, even
when $t$ is not absolutely continuous with respect to $u$, we will use
the notation
\[
      \rnd{t}{u} := \rnd{t^{u}}{u}
\]
where $t^{u}$ is the absolutely continuous part of the Lebesgue
decomposition of $t$ with respect to $u$. An important fact, proved in
Appendix \ref{apx:gmean}, is that $$\sqrnd{t}{u} \extd u =
\sqrnd{u}{t} \extd t,$$ so we can define the {\bf geometric mean
measure}, denoted $\gmean{t}{u}$ or simply $\sqrt{tu}$, using either
of these expressions.

Every set of $t$-measure or $u$-measure zero also has
$\sqrt{tu}$-measure zero. That is,
\[
    \sqrt{tu} \ll t \quad \text{ and }\quad  \sqrt{tu} \ll u.
\]
In fact, every $\sqrt{tu}$-null set is the union of a $t$-null set and
a $u$-null set, as we show in Appendix \ref{apx:gmean}.  This means
$\sqrt{tu}$ is a measure that is maximal with respect to
(\ref{abso-cts}).

Recall that we are assuming our measures are $\sigma$-finite.
Using this, one can show that
\beq \label{chainrule}
\rnd{t}{u}\rnd{u}{t} = 1 \qquad \text{$\sqrt{tu}$-\alme}
\eeq
This rule, obvious when the two measures are equivalent, is proved in
Appendix \ref{apx:gmean}.

We shall need one more type of `field', which may be thought of as
`measurable fields of {\em measures}'.  In general, these involve two
measure spaces: they are certain families $\mu_y$ of measures on a
measurable space $X$, indexed by elements of a measurable space $Y$.
We first introduce the notion of fibered measure distribution
\cite{Yetter2}:
\begin{defn}
Suppose $X$ and $Y$ are measurable spaces and every one-point set
of $Y$ is measurable.  Then a {\boldmath {\bf $Y$-fibered measure
distribution on $Y\times X$}} is a $Y$\!-indexed family of measures
$\bar{\mu}_y$ on $Y\times X$ satisfying the properties:
\begin{itemize}
\item $\bar{\mu}_y$ is supported on $\{y\}\times X$: that is,
$\bar{\mu}_y((Y-\{y\})\times X ) = 0$
\item For every measurable $A\subseteq Y\times X$, the function
$y\mapsto \bar{\mu}_y(A)$ is measurable
\item The family is uniformly finite: that is, there exists
a constant $M$ such that for all $y \in Y$, $\bar{\mu}_y(X) < M$.
\end{itemize}
\end{defn}

Any fibered measure distributions gives rise to a $Y$\!-indexed family of
measures on $X$:
\begin{defn}
Given measurable spaces $X$ and $Y$, $\mu_y$ is a {\boldmath
{\bf $Y$\!-indexed measurable family of measures on $X$}}
if it is induced by a $Y$-fibered measure
distribution $\bar{\mu}_y$ on $Y\times X$; that is, if
\[
\mu_y(A) = \bar{\mu}_y(Y \times A)
\]
for every measurable $A\subseteq X$.
\end{defn}
Notice that, if $\bar\mu_y$ is the fibered measure distribution
associated to the measurable family $\mu_y$, we have \beq \bar\mu_y =
\delta_y \tensor \mu_y \eeq as measures on $Y\times X$, where for each
$y\in Y$, $\delta_y$ is the Dirac measure concentrated at $y$.

By itself, a fibered measure distribution $\bar\mu_y$ on $Y\times X$
is {\em not} a measure on $Y\times X$.  However, taken together with a
suitable measure $\nu$ on $Y$, it may yield a measure
$\lambda$ on $Y\times X$: \beq
\label{disint}
      \lambda = \int_Y \extd\nu \, (\delta_y\tensor \mu_y) \eeq
Because this measure $\lambda$ is obtained from $\mu_y$ by {\em
integration} with respect to $\nu$, the measurable family $\mu_y$ is
also called the {\bf disintegration} of $\lambda$ with respect to
$\nu$.  It is often the {\em dis}integration problem one is interested
in: given a measure $\lambda$ on a product space and a measure $\nu$
on one of the factors, can $\lambda$ be written as an integral of some
measurable family of measures on the other factor, as in
(\ref{disint}).  Conditions for the disintegration problem to have a
solution are given by the `disintegration theorem':

\begin{theo} [Disintegration Theorem] \label{disintegration.theo}
Suppose $X$ and $Y$ are measurable spaces.  Then a
measure $\lambda$ on $Y \times X$ has a disintegration $\mu_y$ with
respect to the measure $\nu$ on $Y$ if and only if $\nu(U) = 0$
implies $\lambda(U\times X)= 0$ for every measurable $U\subseteq Y$.
When this is the case, the measures $\mu_y$ are determined uniquely
for $\nu$-almost every $y$.
\end{theo}

\begin{proof}
Graf and Mauldin \cite{GrafMauldin} state a theorem due to Maharam
\cite{Maharam} that easily implies a stronger version of this result:
namely, that the conclusions hold whenever $X$ and $Y$ are Lusin
spaces.  Recall that a topological space space homeomorphic to separable
complete metric space is called a {\bf Polish space}, while
more generally a \textbf{Lusin space} is a topological space that is
the image of a Polish space under a continuous bijection.  By
Lemma \ref{lem:standard_Borel}, every measurable space we consider ---
i.e., every standard Borel space---is isomorphic to
some Polish space equipped with its $\sigma$-algebra of Borel sets.
\end{proof}

%
\subsubsection{The 2-category of measurable categories: $\me$}
%

We are now in a position to give a definition of the 2-category $\me$
introduced in the work of Crane and Yetter \cite{CraneYetter,
Yetter2}.  The aim of this section is essentially practical: we give
concrete descriptions of the objects, morphisms, and 2-morphisms of
$\me$, and formulae for the composition laws.  These formulae will be
analogous to those presented in the finite-dimensional case in Section
\ref{sec:2vect}, which the current section parallels.

Before diving into the technical details, let us sketch the basic idea
behind the 2-category $\me$:

\begin{itemize}
\item
The objects of $\me$ are `measurable categories', which are categories
somewhat analogous to Hilbert spaces.  The most important sort of
example is the category $H^X$ whose objects are measurable fields of
Hilbert spaces on the measurable space $X$, and whose morphisms are
measurable fields of bounded operators.  If $X$ is a finite set with
$n$ elements, then $H^X \cong \Hilb^n$.  So, $H^X$ generalizes
$\Hilb^n$ to situations where $X$ is a measurable space instead of a
finite set.
\item
The morphisms of $\me$ are `measurable functors'.  The most important
examples are `matrix functors' $T \maps H^X \to H^Y$.  Such a functor
is constructed using a field of Hilbert spaces on $X \times Y$, which
we also denote by $T$.  When $X$ and $Y$ are finite sets, such field
is simply a matrix of Hilbert spaces.  But in general, to construct a
matrix functor $T \maps H^X \to H^Y$ we also need a $Y$\!-indexed
measure on $X$.
\item
The 2-morphisms of $\me$ are `measurable natural transformations'.
The most important examples are `matrix natural transformations'
$\alpha \maps T \to T'$ between matrix functors $T, T' \maps H^X \to
H^Y$.  Such a natural transformation is constructed using a
uniformly bounded field of linear operators
$\alpha_{y,x} \maps T_{y,x} \to T'_{y,x}$.
\end{itemize}

Here we have sketchily described the most important objects, morphisms
and 2-morphisms in $\me$.  However, following our treatment of
$\twoVe$ in Section \ref{2vectcat}, we need to make $\me$ bigger to
obtain a 2-category instead of a bicategory.  To do this, we include
as objects of $\me$ certain categories that are {\it equivalent} to
categories of the form $H^X$, and include as morphisms certain
functors that are {\it naturally isomorphic} to matrix functors.

\subsubsection*{Objects}

Given a measurable space $X$, there is a category $H^X$ with:
\begin{itemize}
\item  measurable fields of Hilbert spaces on $X$ as objects;
\item bounded measurable fields of linear operators on $X$ as
morphisms.
\end{itemize}
Objects of the 2-category $\me$ are `measurable categories'---that
is, `$C^*$\!-categories' that are `$C^*$\!-equivalent' to $H^X$ for
some $X$.  Let us make this precise:

\begin{defn}
A {\bf Banach category} is a category $C$ enriched over
Banach spaces, meaning that for any pair of objects $x,y \in C$,
the set of morphisms from $x$ to $y$ is equipped with the structure
of a Banach space, composition is bilinear, and
\[
     \| fg \| \leq \| f \| \| g \|
\]
for every pair of composable morphisms $f,g$ in $C$.
\end{defn}

\begin{defn}
A {\bf\boldmath Banach $\ast$-category} is a Banach category in which each
morphism $f \maps x \to y$ has an associated morphism $f^* \maps y \to x$,
such that:
\begin{itemize}
\item each map $\hom(x,y)\to\hom(y,x)$ given by $f\mapsto f^*$ is conjugate linear;
\item $(gf)^* = f^* g^*$,  $1_x^* = 1_x$, and $f^{**} = f$, 
for every object $x$ and pair of composable morphisms $f,g$; 
\item for any morphism $f\maps x \to y$, there exists a morphism
$g\maps x \to x$ such that $f^*f = g^*g$;
\item $f^*f = 0$ if and only if $f=0$.    
\end{itemize}
\end{defn}

\begin{defn}
A {\bf\boldmath $C^*$\!-category} is a Banach $\ast$-category such that
for each morphism $f\maps x \to y$, 
\[
     \| f^* f \| = \| f \|^2 .
\]
\end{defn}

Note that for each object $x$ in a $C^*$\!-category, its endomorphisms
form a $C^*$\!-algebra.  Note also that for any measurable
space $X$, $H^X$ is a $C^*$\!-category, where the norm of any
bounded measurable field of operators $\phi \maps \H \to \K$ is
\[          \|\phi\| = \sup_{x \in X} \|\phi_x \|  \]
and we define the $\ast$ operation pointwise:
\[          (\phi^*)_x = (\phi_x)^*  \]
where the right-hand side is the Hilbert space adjoint of the operator
$\phi_x$.

\begin{defn}
A functor $F \maps C \to C'$ between $C^*$\!-categories
is a {\bf\boldmath $C^*$\!-functor} if it maps morphisms to morphisms
in a linear way, and satisfies
\[
        F(f^*) = F(f)^*
\]
for every morphism $f$ in $C$.
\end{defn}
Using the fact that a $*$-homomorphism between unital $C^*$\!-algebras is
automatically norm-decreasing, we can show that any $C^*$\!-functor
satisfies
\[
     \|F(f) \| \leq \|f\| .
\]

\begin{defn}
Given $C^*$\!-categories $C$ and $C'$,
a natural transformation $\alpha \maps F \To F'$ between functors
$F,F' \maps C \to C'$ is {\bf bounded} if for some constant $K$ we have
\[   \|  \alpha_x \| \le K  \]
for all $x \in C$.  If there is a bounded natural isomorphism between
functors between $C^*$\!-categories, we say they are {\bf boundedly
naturally isomorphic}.
\end{defn}

\begin{defn}
A $C^*$\!-functor $F \maps C \to C'$ is a {\bf \boldmath $C^*$\!-equivalence}
if there is a $C^*$\!-functor $\bar{F} \maps C' \to C$ such that
$\bar{F} F$ and $F \bar{F}$ are boundedly naturally isomorphic to
identity functors.
\end{defn}

\begin{defn}
A {\bf measurable category} is a $C^*$\!-category that is $C^*$\!-equivalent
to $H^X$ for some measurable space $X$.
\end{defn}

\subsubsection*{Morphisms}
\label{morphisms}

The morphisms of $\me$ are `measurable functors'.  The most important
measurable functors are the `matrix functors', so we begin with these.
Given two objects $H^X$ and $H^Y$ in $\me$, we can construct a functor
\[
\xymatrix{
   H^X \ar[r]^{T,t}&  H^Y
}
\]
from the following data:
\begin{itemize}
\item
a uniformly finite $Y$\!-indexed measurable family $t_y$ of measures on $X$,
\item
a $t$-class of measurable fields of Hilbert spaces $T$ on $Y\times X$,
such that $t$ is concentrated on the support of $T$; that is, for each
$y\in Y$, $t_y(\{x\in X : T_{y,x} = 0\}) = 0. $
\end{itemize}
Here by {\boldmath \bf $t$-class} we mean a $t_y$-class for each
$y$, as defined in the previous section.

For brevity, we will sometimes denote the functor constructed from
these data simply by $T$.  This functor maps any object
$\H\in H^X$---a measurable field of Hilbert spaces on $X$---to the
object $T\H\in H^Y$ given by
\[
     (T\H)_y = \direct_X \extd t_y\, T_{y,x}\tensor \H_x.
\]
Similarly, it maps any morphism $\phi\maps\H\to \H'$ to the morphism
$T\phi\maps T\H\to T\H'$ given by the direct integral of operators
\[
    (T\phi)_y = \direct_X \extd t_y\, \unit_{T_{y,x}}\tensor \phi_x
\]
where $\unit_{T_{y,x}}$ denotes the identity operator on $T_{y,x}$.
Note that $T$ is a $C^*$\!-functor.  

\begin{defn} Given measurable spaces $X$ and $Y$, a functor
$T \maps H^X \to H^Y$ of the above sort is called a {\bf matrix
functor}.
\end{defn}

Starting from matrix functors, we can define measurable functors
in general:
\begin{defn} Given objects $\HH, \HH' \in \me$, a {\bf measurable
functor} from $\HH$ to $\HH'$ is a $C^*$\!-functor that is boundedly
naturally isomorphic to a composite
\[
\xymatrix{
   \HH \ar[r]^{F}&
   H^X \ar[r]^{T}&
   H^Y \ar[r]^{G}&
   \HH'
}
\]
where $T$ is a matrix functor and the first and last functors are
$C^*$\!-equivalences.
\end{defn}

In Section \ref{2-category} we use results of Yetter to show
that the composite of measurable functors is measurable.  A key
step is showing that the composite of two matrix functors:
\[
\xymatrix{
   H^X \ar[r]^{T,t}&  H^Y \ar[r]^{U,u} & H^Z
}
\]
is boundedly naturally isomorphic to a matrix functor
\[
\xymatrix{
   H^X \ar[r]^{UT,ut}& H^Z.
}
\]
Let us sketch how this step goes, since we will need explicit formulas
for $UT$ and $ut$.  Picking any object $\H \in H^X$, we have
\begin{align*}
  (UT\H)_z&=\direct_Y \extd u_z \, U_{z,y}\tensor (T\H)_y \\
      &=\direct_Y \extd u_z \, U_{z,y}\tensor
        \left(\direct_X \extd t_y T_{y,x}\tensor \H_x \right)
\end{align*}
To express this in terms of a matrix functor, we will write it as
direct integral over $X$ with respect to a $Z$-indexed family
of measures on $X$ denoted $ut$, defined by:
\beq \label{meas1compo}
(ut)_z = \int_Y \extd u_z(y)\, t_y.
\eeq
To do this we use the disintegration theorem, Thm.\
\ref{disintegration.theo}, to obtain a field of measures
$k_{z,x}$ such that
\beq
\label{swapping_disintegrations}
\int_X \extd (ut)_z(x)\, (k_{z,x}\tensor \delta_x) =
\int_Y \extd u_z(y)\, (\delta_y\tensor t_y).
\eeq
as measures on $Y\times X$.  That is, $k_{z,x}$ and $t_y$ are,
respectively, the $X$- and $Y$-disintegrations of the same
measure on $X\times Y$, with respect to the measures $(ut)_z$ on $X$
and $u_z$ on $Y$.  The measures $k_{y,x}$ are determined uniquely for
all $z$ and $(ut)_z$-almost every $x$.  With these definitions, it
follows that there is a bounded natural isomorphism
\begin{align}
\label{bounded.natural.iso}
  (UT\H)_z &\cong
\direct_X \extd (ut)_z
\left(\direct_Y \extd k_{z,x} U_{z,y} \tensor T_{y,x}\right) \tensor \H_x \\
&=
\direct_X \extd (ut)_z (U T)_{z,x} \tensor \H_x
\end{align}
where
\beq \label{convol_1maps}
(U T)_{z,x} = \direct _Y \extd k_{z,x}(y) \,U_{z,y}\otimes T_{y,z},
\eeq
This formula for $U T$ is analogous to (\ref{matmult}).  We refer to
Yetter \cite{Yetter2} for proofs that the family of measures $ut$ and
the field of Hilbert spaces $U T$ are measurable, and hence define a
matrix functor.

It is often convenient to use an alternative form of
(\ref{swapping_disintegrations}) in terms of integrals of functions:
for every measurable function $F$ on $Y\times X$ and for all $z \in
Z$,
\beq
\label{integration_meas}
\int_X \extd (ut)_z(x) \int_Y \extd k_{z,x}(y) \,F(y,x)
= \int_Y \extd u_z(y) \int_X \extd t_y(x)
\,F(y,x).
\eeq
This can be thought of as a sort of `Fubini theorem', since it lets us
change the order of integration, but here the measure on one factor in
the product is parameterized by the other factor.

Besides composition of morphisms in $\me$, we also need identity
morphisms.  Given an object $H^X$, to show its identity functor
$\unit_X \maps H^X \to H^X$ is a matrix functor we need an $X$-indexed family
of measures on $X$, and a field of Hilbert spaces on $X\times X$.
Denote the coordinates of $X\times X$ by $(x',x)$.  The family of
measures assigns to each $x'\in X$ the unit Dirac measure concentrated
at the point $x'$:
\[
\delta_{x'}(A) =
  \begin{cases}
      1 \quad  \mbox{if} \quad x' \in A \\
      0 \quad  \mbox{otherwise}
  \end{cases}
  \qquad \mbox{for every measurable set} \, A \subseteq X
\]
The field of Hilbert spaces on $X\times X$ is the constant field
$(\unit_X)_{x',x}=\C$.  It is simple to check that this acts as both
left and right identity for composition.  Let us check that it is a
right identity by forming this composite:
\[
  \xymatrix{H^X \ar[r]^{\unit_X,\delta} & H^X \ar[r]^{T,t} &H^Y }
\]
One can check that the composite measure is:
\[
    (t\delta)_y = \direct_X \extd t_y(x') \delta_x' = t_y,
\]
and hence, using (\ref{integration_meas}),
\[
    k_{y,x} = \delta_x.
\]
We can then calculate the field of operators:
\begin{align*}
        (T\unit_X)_{y,x} &\cong
\direct_X \extd\delta_x(x')\; T_{y,x'} \tensor \C = T_{y,x}.
\end{align*}

\subsubsection*{2-Morphisms}
\label{2-morphisms}

The 2-morphisms in $\me$ are `measurable natural transformations'.
The most important of these are the `matrix natural transformations'.
Given two matrix functors $(T, t)$ and $(T', t')$, we can construct a
natural transformation between them from a $\sqrt{t t'}$-class of
bounded measurable fields of linear operators $$\alpha_{y,x}\maps
T_{y,x} \arr T'_{y,x}$$ on $Y\times X$.  Here by a {\bf \boldmath
$\sqrt{t t'}$-class}, we mean a $\sqrt{t_y t'_y}$-class for each $y$,
where the $\sqrt{t_y t'_y}$ is the geometric mean of the measures
$t_y$ and $t'_y$.  By {\bf bounded}, we mean $\alpha_{y,x}$ have a
common bound for all $y$ and $\sqrt{t_y t'_y}$-almost every $x$.

We denote the natural transformation constructed from
these data simply by $\alpha$.  This natural transformation assigns to
each object $\H\in H^X$ the morphism $\alpha_\H\maps T\H \to T'\H$ in $H^Y$
with components:
\[
\begin{array}{cccc}
   (\alpha_\H)_y \maps &  \displaystyle
     \direct_X \extd t_y  \; T_{y,x} \tensor {\H_x} & \to &
                                         \displaystyle
    \direct_X \extd t'_y  \; T'_{y,x} \tensor {\H_x} \\ \\
      &   \direct_X \extd t_y  \; \psi_{y,x} & \mapsto &
  \direct_X \extd t'_y  \; [\tilde\alpha_{y,x}\tensor \unit_{\H_x}](\psi_{y,x})
\end{array}
\]
where $\tilde\alpha$ is the rescaled field
\beq
\label{rescaling}
\tilde\alpha  = \sqrnd{t_y}{t'_y}\,\alpha .
\eeq

To check that $\alpha_\H$ is well defined, pick $\psi \in T\H$ and compute
\begin{align*}
\int_X \extd t'_y
   \|
         [\tilde\alpha_{y,x}\tensor \unit_{\H_x}](\psi_{y,x})
   \|^2
&=
\int_X \extd t^c_y
   \|
         [\alpha_{y,x}\tensor \unit_{\H_x}](\psi_{y,x})
   \|^2 \\
&\leq {\rm ess\, } \sup_{x'} \|\alpha_{y,x'}\|^2
\,\direct _X \extd t_y \|\psi_{y,x} \|^2 < \infty
\end{align*}
where $t^c_y$ is the absolutely continuous part of the Lebesgue
decomposition of $t_y$ with respect to $t'_y$; note that, since
$t^c_y$ is equivalent to $\sqrt{t_y t'_y}$, the field $\alpha$ is
essentially bounded with respect to $t^c_y$. This inequality shows
that the image $(\alpha_\H)_y(\psi)$ belongs to $(T'\H)_y$, and that
$\alpha_\H$ is a field of bounded linear maps, as required. Note also
that the direct integral defining the image does not depend on the
chosen representative of $\alpha$.

To check that $\alpha$ is natural it suffices to choose a morphism
$\phi\maps \H \to \H'$ in $H^X$ and show that the naturality square
\[
\xymatrix{
  T\H \ar[rr]^{T\phi}\ar[dd]_{\alpha_H}&&T\H' \ar[dd]^{\alpha_{\H'}}\\
  \\
  T'\H \ar[rr]_{T'\phi}&&T'\H'
}
\]
commutes; that is,
\[
           \alpha_{\H'}\, (T\phi) = (T'\phi)\,  \alpha_\H .
\]
To check this, apply the operator on the left to
$\psi \in T\H$ and calculate:
\begin{align*}
  (\alpha_{\H'})_y  (T\phi)_y (\psi_y)
      &= (\alpha_{\H'})_y
 \left( \direct_X \extd t_y(x)
[ \unit_{T_{y,x}} \tensor \phi_x](\psi_y) \right) \\
      &= \direct_X \extd t'_y(x) [\tilde\alpha_{y,x} \tensor \unit_{H'_x}]
 [\unit_{T_{y,x}}\tensor \phi_x](\psi_y) \\
      &= \direct_X \extd{t'_y}(x)
 [\unit_{T'_{y,x}} \tensor \phi_x]
[ \alpha_{y,x} \tensor \unit_{\H_x}](\psi_{y})
      =  (T'\phi)_y  (\alpha_\H)_y(\psi_y).
\end{align*}

\begin{defn} Given measurable spaces $X$ and $Y$ and matrix functors
$T, T' \maps H^X \to H^Y$, a natural transformation $\alpha \maps T
\To T'$ of the above sort is called a {\bf matrix natural transformation}.
\end{defn}

However, in analogy to Thm.\ \ref{composition.4}, we have:

\begin{theo}
\label{magic.thm}
Given measurable spaces $X$ and $Y$ and matrix functors
$T, T' \maps H^X \to H^Y$, every bounded natural transformation
$\alpha \maps T  \To T'$ is a matrix natural transformation, and
conversely.
\end{theo}

\begin{proof} The converse is easy.  So,
suppose $T, T' \maps H^X \to H^Y$ are matrix natural
transformations and $\alpha \maps T \To T'$ is a bounded natural
transformation. Denote by $t$ and $t'$ the families of measures of the two
matrix functors. We will show that $\alpha$ is a matrix
natural transformation in three steps.  We begin by assuming that for
each $y \in Y$, $t_y = t'_y$; we then extend the result to the case
where the measures are only equivalent $t_y \sim t'_y$; then finally we
treat the general case.

Assume first $t = t'$. Let $\J$ be the measurable field
of Hilbert spaces on $X$ with
\[          \J_x = \C \qquad \textrm{for all } x \in X . \]
Then $T\J$ and $T'\J$ are measurable fields of Hilbert spaces on $Y$
with canonical isomorphisms
\beq \label{isoTJ}          (T\J)_y \cong \direct_X \extd t_y \; T_{x,y}, \qquad
           (T'\J)_y \cong \direct_X \extd t_y \; T'_{x,y}
\eeq
Using these, we may think of $\alpha_\J$ as a measurable field of
operators on $Y$ with
\[          (\alpha_\J)_y \maps \direct_X \extd  t_y \; T_{x,y} \to
                                \direct_X \extd t_y \; T'_{x,y} .\]

We now show that for any fixed $y \in Y$
there is a bounded measurable field of operators on $X$, say
\[          \alpha_{y,x} \maps T_{x,y} \to T'_{x,y},  \]
with the property that
\begin{equation}
\label{decomposable.operator}
   (\alpha_\J)_y \maps
   \direct_X \extd t_y  \; \psi_{y,x}
   \mapsto
   \direct_X \extd t_y  \; \alpha_{y,x} (\psi_{y,x})
\end{equation}
for any measurable field of vectors $\psi_{y,x} \in T_{y,x}$. For this, 
note that any  measurable bounded function $f$ on $X$ defines a 
morphism $$f\maps  \J \to \J$$ in $H^X$, mapping a vector field 
$\psi_x$ to $f(x) \psi_x$.
The functors $T$ and $T'$ map $f$ to the some morphisms
\[ T_f \maps T\J \to T\J \qquad \mbox{and} \qquad  T'_f \maps T'\J \to T'\J \]
in $H^Y$.
Using the canonical isomorphisms (\ref{isoTJ}), we may think of $T_f$ 
as a measurable field of {\bf multiplication operators} on $Y$ with
\[
\begin{array}{cccl}
   \displaystyle  (T_f)_y \maps & \direct_X \extd t_y \; T_{y,x} &\to&
   \displaystyle              \direct_X \extd t_y \; T_{y,x} \\ \\
                &  \direct_X \extd t_y  \; \psi_{y,x} &\mapsto&
                   \direct_X \extd t_y  \; f(x) \psi_{y,x}
\end{array}
\]
and similarly for $T'_f$.
The naturality of $\alpha$ implies that the square
\[
\xymatrix{
  T\J \ar[rr]^{T_f}\ar[dd]_{\alpha_{\J}}&& T\J \ar[dd]^{\alpha_{\J}}\\
  \\
  T'\J \ar[rr]_{T'_f}&& T'\J
}
\]
commutes; unraveling this condition it follows that, for each $y\in Y$,
\[           (\alpha_\J)_y \; (T_f)_y = (T'_f)_y \; (\alpha_\J)_y  .\]

Now we use this result:

\begin{lemma}
\label{magic.lemma}
Suppose $X$ is a measurable space and $\mu$ is a measure on $X$
Suppose $T$ and $T'$ are measurable fields of Hilbert spaces on $X$ and
\[          \alpha \maps \direct_X \extd \mu \; T_x \to
\direct_X \extd \mu \; T'_x \]
is a bounded linear operator such that
\[           \alpha \, T_f = T'_f \, \beta  \]
for every $f \in L^\infty(X,\mu)$, where $T_f$ and $T'_f$ are
multiplication operators as above.  Then there exists a uniformly
bounded measurable field of operators
\[           \alpha_x \maps T_x \to T'_x \]
such that
\[ \displaystyle{
   \alpha \maps
   \direct_X \extd \mu \; \psi_{x}
   \mapsto
   \direct_X \extd \mu  \; \alpha_x (\psi_{x})  . }
\]
\end{lemma}

\begin{proof.within.proof}
This can be found in Dixmier's book \cite[Part II Chap.~2 Thm.~1]{Dixmier}.
\end{proof.within.proof}

It follows that for any $y \in Y$
there is a uniformly bounded measurable field of operators on
$X$, say
\[          \alpha_{y,x} \maps T_{x,y} \to T'_{x,y},  \]
satisfying Eq.\ \ref{decomposable.operator}.

Next note that as we let $y$ vary, $\alpha_{y,x}$ defines a uniformly
bounded measurable field of operators on $X \times Y$.  The uniform
boundedness follows from the fact that for all $y$,
\[        {\rm ess\, }
\sup_x \|\alpha_{y,x}\|  =   \|(\alpha_\J)_y \| \le K  \]
since $\alpha$ is a bounded natural transformation.
The measurability follows from the fact that $(\alpha_J)_y$ is
a measurable field of bounded operators on $Y$.

To conclude, we use this measurable field $\alpha_{y,x}$ to
prove that $\alpha$ is a matrix natural transformation.  For this,
we must show that for {\it any} measurable field $\H$ of Hilbert spaces
on $X$, we have
\[
   (\alpha_\H)_y \maps
\direct_X \extd t_y  \; \psi_{y,x}
\mapsto
\direct_X \extd t_y  \; [\alpha_{y,x}\tensor \unit_{\H_x}](\psi_{y,x})
\]

To prove this, first we consider the case where $\K$ is a constant
field of Hilbert spaces:
\[          \K_x = K \qquad \textrm{for all } x \in X , \]
for some Hilbert space $K$ of countably infinite dimension.  We handle this case
by choosing an orthonormal basis $e_j \in K$ and using this to define
inclusions
\[           i_j \maps \J \to \K, \qquad \psi_x \mapsto \psi_x e_j  \]
The naturality of $\alpha$ implies that the square
\[
\xymatrix{
  T\J \ar[rr]^{Ti_j}\ar[dd]_{\alpha_{\J}}&& T\K \ar[dd]^{\alpha_{\K}}\\
  \\
  T'\J \ar[rr]_{T'i_j}&& T'\K
}
\]
commutes; it follows that
\[           (\alpha_\K)_y \; (Ti_j)_y = (T'i_j)_y \; (\alpha_\J)_y  .\]
Since we already know $\alpha_\J$ is given by Eq.\
\ref{decomposable.operator}, writing any vector field in $\K$ in terms 
of the orthonormal basis $e_j$, we obtain that
\begin{equation}
\label{special.case}
   (\alpha_\K)_y \maps
\direct_X \extd t_y  \; \psi_{y,x}
\mapsto
\direct_X \extd t_y  \; [\alpha_{y,x}\tensor \unit_{K}](\psi_{y,x})
\end{equation}

Next, we use the fact that every measurable field $\H$ of
Hilbert spaces is isomorphic to a direct summand of $\K$
\cite[Part II, Chap.~1, Prop.~1]{Dixmier}.   So, we have a
projection
\[           p \maps \K \to \H . \]
The naturality of $\alpha$ implies that the square
\[
\xymatrix{
  T\K \ar[rr]^{Tp}\ar[dd]_{\alpha_{\K}}&& T\H \ar[dd]^{\alpha_{\H}}\\
  \\
  T'\J \ar[rr]_{T'p}&& T'\H
}
\]
commutes; it follows that
\[           (\alpha_\H)_y \; (Tp)_y = (T'p)_y \; (\alpha_\K)_y  .\]
Since we already know $\alpha_\K$ is given by Eq.\
\ref{special.case}, using the fact that any vector field in $\H$ is the 
image by $p$ of a vector field in $\K$, we obtain that
\[
   (\alpha_\H)_y \maps
\direct_X \extd t_y  \; \psi_{y,x}
\mapsto
\direct_X \extd t_y  \; [\alpha_{y,x}\tensor \unit_{\H_x}](\psi_{y,x})
\]

We have assumed so far that the matrix functors $T, T'$ are 
constructed from the same family of measures $t=t'$.
Next, let us relax this hypothesis and suppose that for each 
$y\in Y$, we have $t_y \sim t'_y$. Let $\tilde{T'}$ be the matrix 
functor constructed from the family of measures $t$ and the field 
of Hilbert space $T'$. The bounded measurable field of identity 
operators $\unit_{T'_{y,x}}$ defines a matrix natural transformation
\[r_{t, t'} \maps T \To \tilde{T'}.\]
This natural transformation assigns to any object $\H \in H^X$ a 
morphism $r_{t,t'\H } \maps T\H \to \tilde{T'}\H$  with components:
$$
(r_{t,t'\H} )_y \maps \direct \extd t'_y \psi_{y,x} 
\mapsto \direct \extd t_y \sqrnd{t'_y}{t_y} \psi_{y,x}
$$
Moreover, by equivalence of the measures, $r_{t,t'}$ is a 
natural isomorphism and $r_{t, t'}^{-1} = r_{t',t}$.

Suppose $\alpha \maps T \To T'$ is a bounded natural 
transformation. The composite $r_{t, t'} \alpha \maps T \to \tilde{T'}$ 
is a bounded natural transformation between matrix functors 
constructed from the same families of measures $t$. According to 
the result shown above, we know that this composite is a matrix 
measurable transformation, defined by some measurable field 
of operators
$$
\alpha_{y,x} \maps T_{y,x} \to T'_{y,x}
$$
Writing $\alpha = r_{t', t} ( r_{t, t'} \alpha)$, we conclude that 
$\alpha$ acts on each object $\H \in H^X$ as
$$
(\alpha_\H)_y \maps
\direct_X \extd t_y  \; \psi_{y,x}
\mapsto
\direct_X \extd t_y  \; [\tilde\alpha_{y,x}\tensor \unit_{\H_x}](\psi_{y,x})
$$
where $\tilde{\alpha}$ is the rescaled field
\[
\tilde{\alpha} = \sqrnd{t_y}{t'_y} \alpha
\]
This shows that $\alpha$ is a matrix natural transformation.

Finally, to prove the theorem in its full generality, we consider 
the Lebesgue decomposition of the measures $t_y$ and $t'_y$ 
with respect to each other (see Appendix  \ref{Lebesgue-Radon-Nykodym}):
\[
t = t^{t'} + \overline{t^{t'}}, \qquad t^{t'} \ll t' \quad \overline{t^{t'}} \perp t'
\]
and likewise,
\[
 t' = t'^{t} + \overline{t'^{t}}, \qquad t'^{t} \ll t \quad \overline{t'^{t}} \perp t
\]
where the subscript $y$ indexing the measures is dropped 
for clarity. Prop.\ref{fact1} shows that $t_y^{t'} \perp \overline{t_y^{t'}}$ 
and $t_y'^{t}\perp \overline{t_y'^{t}}$.
Moreover, Prop.\ref{mutual-decomposition} shows that $t_y^{t'} \sim t_y'^{t}$. 
Consequently, for each $y\in Y$, there are disjoint measurable sets 
$A_y, B_y$ and $B'_y$  such that $t_y^{t'}$ and $t_y'^{t}$ are supported 
on $A_y$, that is,
\[
 t_y^{t'}(S) = t_y^{t'}(S \cap A_y) \qquad  t_y'^{t}(S) = t_y'^{t}(S\cap A_y),
\]
for all measurable sets $S$; and such that  $\overline{t_y^{t'}}$ is 
supported on $B_y$, and $\overline{t_y'^{t}}$ is supported on $B'_y$.

Let $\tilde{T}$ be the matrix functor constructed from the family of 
measures $t^{t'}$ and the field of Hilbert spaces $T_{y,x}$; let 
$\tilde{T'}$ be the matrix functor constructed from the the family 
of measures $t'^{t}$ and the field of Hilbert spaces $T'_{y,x}$. 
The bounded measurable field of identity operators $\unit_{T_{y,x}}$ 
define matrix natural transformations:
\[
i \maps \tilde{T}  \To T, \quad p \maps T \To \tilde{T}
\]
Given any object $\H \in H^X$, we get a morphism 
$i_\H \maps \tilde{T}\H \to T\H$, whose components act as inclusions:
\[
(i_\H)_y \maps \direct \extd t^{t'}_y \, \psi_{y,x} \mapsto \direct \extd t_y \, \chi_{A_y}(x) \, \psi_{y,x}
\]
where  $\chi_{A}$ is the characteristic function of the set $A \subset X$:
\[
   \chi_{A}(x) =
\left\{ \begin{array}{cc} 1 & x\in A \\
0 & x\not\in A. \end{array}\right.
\]
We also get a morphism $p_\H \maps T\H \to \tilde{T}\H$, whose components act as projections:
\[
(p_\H)_y \maps \direct \extd t'_y \, \psi_{y,x} \mapsto \direct \extd t^{'t}_y \, \psi_{y,x}.
\]
Likewise, the bounded measurable field of identity operators $\unit_{T'_{y,x}}$ define an inclusion and a projection:
\[
i' \maps \tilde{T'}  \To T', \quad p' \maps T' \To \tilde{T'}
\]

Suppose $\alpha \maps T \To T'$ is a bounded natural transformation. 
The composite $p'\alpha i \maps \tilde{T} \To \tilde{T'}$ is then a bounded 
natural transformation between matrix functors constructed from 
equivalent families of measures.  According to the result shown above, 
we know that this composite is a matrix natural tranformation, defined 
by some measurable field of operators
$$
\alpha_{y,x} \maps T_{y,x} \to T'_{y,x}
$$

We will show below the equality of natural transformations:
\beq \label{magic-equality}
\alpha = i'[p' \alpha i] p
\eeq
This equality leads to our final result. Indeed, for any $\H \in H^X$ 
and each $y \in Y$, it yields:
\[
 (\alpha_\H)_y \maps \direct \extd t_y \, \psi_{y,x} \mapsto \direct \extd t'_y \,  \chi_{A_y}(x) \,\sqrnd{t_y^{t'}}{t_y'^t}\,[\alpha_{y,x} \otimes \unit_{\H_x} ](\psi_{y,x})
\]
and we conclude using the fact that, for all $y$ and $t'_y$-almost all $x$,
\[
 \chi_{A_y}(x) \,\sqrnd{t_y^{t'}}{t_y'^t} = \sqrnd{t^{t'}_y}{t'_y}.
\]

The equality (\ref{magic-equality}) follows from naturality of $\alpha$. 
In fact, naturality implies that, for any morphism $\phi \maps \H \to \H$, 
the square
\[
\xymatrix{
  T\H \ar[rr]^{T\phi}\ar[dd]_{\alpha_{\H}}&& T\H \ar[dd]^{\alpha_{\H}}\\
  \\
  T'\H \ar[rr]_{T'\phi}&& T'\H
}
\]
commutes. It follows that, for each $y\in Y$,
\[
(\alpha_\H)_y (T\phi)_y = (T'\phi)_y (\alpha_\H)_y
\]
Let us fix $y\in Y$. We apply naturality to the morphism 
$$\chi_{B_y} \maps H \to H$$ mapping  any vector field $\psi_x$ 
to the vector field $\chi_{B_y}(x) \psi_x$. Its image by the functor 
$T'$ defines a projection operator
\[
(T'\chi_{B_y})_y \equiv T'_{B_y} = \direct_{B_y} \extd t'_y\, \unit_{T'_{y,x}} \otimes \unit_{\H_x}
\]
Since $B_y$ is a $t'_y$-null set, this operator acts trivially on 
$T'\H$. It then follows from naturality that
\beq \label{Invisible1}
(\alpha_\H)_y T_{B_y} = T'_{B_y} (\alpha_\H)_y = 0.
\eeq
Likewise, applying naturality to the morphism $$\chi_{B'_y} \maps H \to H$$ leads to
\beq \label{Invisible2}
0 =  (\alpha_\H)_y T_{B'_y} = T'_{B'_y} (\alpha_\H)_y
\eeq
We now use the following decompositions of the identities operators 
on the Hilbert spaces $(T\H)_y$ and $(T'\H)_y$ into direct sums of 
projections:
\[
\unit_{(T\H)_y} = T_{A_y}\oplus T_{B_y}, \qquad \unit_{(T'\H)_y} = T'_{A_y} \oplus T'_{B'_y}
\]
to write:
\[
 (\alpha_\H)_y = [T'_{A_y} \oplus T'_{B'_y}]  (\alpha_\H)_y [T_{A_y}\oplus T_{B_y}]
\]
Together with (\ref{Invisible1}) and  (\ref{Invisible2}), it yields:
\[
 (\alpha_\H)_y = T'_{A_y} (\alpha_\H)_y T_{A_y}
\]
To conclude, observe that
\[
 T_{A_y} = (ip_\H)_y, \qquad T'_{A_y} = (i'p'_\H)_y
\]
We finally obtain:
\[
 (\alpha_\H)_y = (i'p'_\H)_y (\alpha_\H)_y (ip_\H)_y
\]
which shows our equality (\ref{magic-equality}). This completes the proof of the theorem.
\end{proof}

This allows an easy definition for the 2-morphisms in $\me$:

\begin{defn} A {\bf measurable natural transformation} is a
bounded natural transformation between measurable functors.
\end{defn}

For our work it will be useful to have explicit formulas
for composition of matrix natural transformations.
So, let us compute the vertical composite of two matrix
natural transformations $\alpha$ and $\alpha'$:
$$
\xymatrix{
  H^X\ar@/^5ex/[rr]^{T, t}="g1"\ar[rr]^(0.35){T', t'}\ar@{}[rr]|{}="g2"
  \ar@/_5ex/[rr]_{T'', t''}="g3"&&H^Y
  \ar@{=>}^{\alpha} "g1"+<0ex,-2ex>;"g2"+<0ex,1ex>
  \ar@{=>}^{\alpha'} "g2"+<0ex,-1ex>;"g3"+<0ex,2ex>
}
$$
For any object $\H \in H^X$, we get morphisms
$\alpha_\H$ and $\alpha'_\H$ in $H^Y$.  Their composite is easy to calculate:
\[
\begin{array}{cccc}
   (\alpha_{\H'})(\alpha_\H)_y \maps &  \displaystyle
            \direct_X \extd t_y  \; T_{y,x} \tensor {\H_x} & \to &
                                      \displaystyle
       \direct_X \extd t''_y  \; T''_{y,x} \tensor {\H_x} \\ \\
      &   \direct_X \extd t_y  \; \psi_{y,x} & \mapsto &
        \direct_X \extd t''_y  \;
[(\tilde \alpha'_{y,x}\tilde \alpha_{y,x})\tensor \unit_{\H_x}](\psi_{y,x})
\end{array}
\]
So, the composite is a measurable natural transformation
$\alpha' \cdot \alpha$ with:
\beq \label{resc-vert2mor}
(\widetilde{\alpha'\cdot\alpha})_{y,x}
= \tilde\alpha'_{y,x} \tilde\alpha_{y,x}  .
\eeq
For some calculations it will be useful to have this equation written 
explicitly in terms of the original fields $\alpha$ and $\alpha'$, rather 
than their rescalings:
\beq \label{vert2mor}
\left(\alpha'\cdot \alpha\right)_{y,x} =
            \sqrnd{t''_y}{t_y}
            \sqrnd{t'_y}{t''_y}
            \sqrnd{t_y}{t'_y} \;
\alpha'_{y,x} \alpha_{y,x}
\eeq
This equality defines the composite field almost everywhere for the 
geometric mean measure $\sqrt{t_y t''_y}$.

Next, let us compute the horizontal composite of two
matrix natural transformations:
\[
  \xymatrix{
  H^X\ar@/^2ex/[rr]^{T, t}="g1"\ar@/_2ex/[rr]_{T',  t'}="g2"&& H^Y
  \ar@/^2ex/[rr]^{U,  u}="g3"\ar@/_2ex/[rr]_{U', u'}="g4"&& H^Z
  \ar@{=>}^{\alpha} "g1"+<0ex,-2.5ex>;"g2"+<0ex,2.5ex>
  \ar@{=>}^{\beta} "g3"+<0ex,-2.5ex>;"g4"+<0ex,2.5ex>
}
\]
Recall that the horizontal composite $\beta \circ \alpha$ is
defined so that
\[
\xymatrix{
  UT\H \ar[rr]^{U\alpha_\H}\ar[dd]_{\beta_{T\H}}\ar[ddrr]|-{(\beta\circ \alpha)_\H}&&UT'\H
\ar[dd]^{\beta_{T'\H}}\\
  \\
  U'T\H \ar[rr]_{U'\alpha_\H}&&U'T'\H
}
\]
commutes. Let us pick an element $\psi \in UT\H$, which can be written
in the form
\[
\psi_z = \direct_X \extd (ut)_z \,\psi_{z,x}, \quad \mbox{with} \quad
\psi_{z,x} = \direct_Y \extd k_{z, x} \, \psi_{z, y, x}
\]
by definition of the composite field $UT$. Note that, thanks to
Eq.\ (\ref{integration_meas}) which defines the family of measures $k_{z,
x}$, the section $\psi_z$ can also be written as
\[
\psi_z = \direct_Y \extd u_z \,\psi_{z,y}, \quad
\mbox{with} \quad \psi_{z, y} = \direct_X \extd t_y\,\psi_{z,y,x}
\]
Having introduced all these notations, we now evaluate the image of
$\psi$ under the morphism $(\beta\circ\alpha)_\H$:
 \begin{align*}
((\beta\circ\alpha)_\H)_z(\psi_z)
&= (U'\alpha_\H)_z\circ(\beta_{T\H})_z (\psi_z)   \\
&= \left (\direct_Y {\extd u'_z} \,\unit_{U'_{z,y}} \tensor
(\alpha_{\H})_y \right)
    \left( \direct_Y \extd u'_z \,[\tilde\beta_{z,y}\tensor
\unit_{(T\H)_y}](\psi_{z, y}) \right)         \\
&=  \direct_Y \extd u'_z \,[\tilde\beta_{z,y} \tensor
(\alpha_{\H})_y](\psi_{z, y})           \\
&=\direct_Y \extd u'_z \direct_X \extd t'_y \,
       [\tilde\beta_{z,y}\tensor \tilde\alpha_{y,x}\tensor
\unit_{\H_x}](\psi_{z,y,x})
\end{align*}
Applying the disintegration theorem, we can rewrite this last direct
integral as an integral over $X$ with respect to the measure
\[
(u't')_z = \int_Y \extd u'_z(y) \, t'_y
\]
We obtain
\begin{align*}
((\beta\circ\alpha)_\H)_z(\psi_z) &=\direct_X \extd (u't')_z
           \direct_Y \extd k'_{z,x} \,
[\tilde\beta_{z,y}\tensor \tilde\alpha_{y,x}\tensor \unit_{\H_x}](\psi_{z,y,x})\\
&= \direct_X \extd (u't')_z
 [(\widetilde{\beta\circ\alpha})_{z, x} \otimes \unit_{\H_x}] (\psi_{z, x})
\end{align*}
where
\beq \label{resc-hor2mor}
(\widetilde{\beta \circ \alpha})_{z,x}(\psi_{z, x}) =
\direct_Y \extd k'_{z,x}\,
[\tilde\beta_{z,y}\tensor \tilde\alpha_{y,x}](\psi_{z, y, x}).
\eeq
Equivalently, in terms of the original fields $\alpha$ and $\beta$:
\beq \label{hor2mor}
\left(\beta \circ \alpha\right)_{z,x}(\psi_{z, x})  =  \sqrnd{(u't')_z}{(ut)_z}   \direct_Y \extd k'_{z,x}\,
[\sqrnd{u_z}{u'_z} \sqrnd{t_y}{t'_y} \beta_{z,y}\tensor \alpha_{y,x}](\psi_{z, y, x})
\eeq

A special case is worth mentioning.  When the source and target
morphisms of $\alpha$ and $\beta$ coincide, we have $k=k'$, and the
horizontal composition formula above simply says
$(\beta\circ\alpha)_{z,x}$ is a direct integral of the fields of
operators $\beta_{z,y}\tensor\alpha_{y,x}$.

Besides composition of 2-morphisms in $\me$ we also need identity
2-morphsms.  Given a matrix functor $T\maps H^X \to H^Y$, its identity
2-morphism $\unit_T\maps T \To T$ is, up $t$-\alme--equivalence,
given by the field of identity operators:
\[
(\unit_T)_{y,x}=\unit_{T_{y,x}}\maps T_{y,x} \arr T_{y,x}.
\]
This acts as an identity for the vertical composition; the identity
2-morphism of an identity morphism, $\unit_{\unit_X}$, acts as an
identity for horizontal composition as well.

In calculations, it is often convenient to be able to describe a
2-morphism either by $\alpha$ or its rescaling $\tilde\alpha$.  The
relationship between these two descriptions is given by the following:
\begin{lemma}
The fields $\alpha_{y,x}$ and $\alpha'_{y,x}$ are $\sqrt{t_y
t'_y}$-equivalent if and only if their rescalings $\tilde\alpha_{y,x}$
and $\tilde\alpha'_{y,x}$ are $t'_y$-equivalent.
\end{lemma}
\begin{proof}
For each $y$, let $A_y$ and $\tilde A_y$ be the subsets of $X$ on
which $\alpha \neq \alpha'$, and $\tilde\alpha\neq \tilde\alpha'$,
respectively. Observe that $\tilde A_y$ is the intersection of $A_y$
with the set of $x$ for which the rescaling factor is non-zero:
\[
\tilde A_y = A_y \cap \left\{x : {\textstyle \sqrnd{t_y}{t'_y}(x)}
\not= 0\right\}.
\]
Supposing first that $\alpha_{y,x}$ and $\alpha'_{y,x}$ are $\sqrt{t_y
t'_y}$-equivalent, we have $\sqrt{ t_y t'_y}\,(A_y) = 0$, so by the
definition of the geometric mean measure
\[
\sqrt{ t_y  t'_y}\,(A_y) =  \int_{A_y} \extd t'_y\, \sqrnd{t_y}{t'_y} = 0.
\]
Thus the rescaling factor vanishes for $t'_y$-almost every $x \in
A_y$; that is, $\tilde A_y$ has $t'_y$-measure zero. Conversely, if
$t'_y(\tilde A_y) =0$, we have:
\[
    \sqrt{ t_y t'_y}\,(A_y) = \sqrt{ t_y t'_y}\,(\tilde A_y) + \sqrt{
    t_y t'_y}\,(A_y - \tilde A_y).
\]
The first term on the right vanishes because $\sqrt{ t_y t'_y} \ll
t'_y$, while the second vanishes since $\sqrnd{t_y}{t'_y} = 0$ on $A_y
- \tilde A_y$.  So, the rescaling $\alpha \mapsto \tilde\alpha$
induces a one-to-one correspondence between $\sqrt{tt'}$-classes of
fields $\alpha$ and $t'$-classes of rescaled fields $\tilde\alpha$.
\end{proof}

%
\subsubsection{Construction of $\me$ as a 2-category}
\label{2-category}
%

\begin{theo} \label{me} There is a sub-2-category $\me$
of $\Cat$ where the objects are measurable categories, the morphisms
are measurable functors, and the 2-morphisms are measurable
natural transformations.
\end{theo}

In Section \ref{morphisms} we showed that for any measurable space $X$,
the identity $\unit_X \maps H^X \to H^X$ is a matrix functor.
It follows that the identity on any measurable category is
a measurable functor.  Similarly, in Section \ref{2-morphisms} we showed
that for any matrix functor $T$, the identity $1_T \maps T \To T$
is a matrix natural transformation.  This implies that the identity on
any measurable functor is a measurable natural transformation.
To prove that the composite of measurable functors is measurable,
we will use the sequence of lemmas below.  Since measurable natural
transformations are just bounded natural transformations between
measurable functors, by Thm.\ \ref{magic.thm}, it will then
easily follow that measurable natural transformations are closed under
vertical and horizontal composition.

\begin{lemma}
\label{me.composition.0}
A composite of matrix functors is boundedly naturally isomorphic to a
matrix functor.
\end{lemma}

\begin{proof} This was proved by Yetter \cite[Thm.\ 45]{Yetter2},
and we have sketched his argument in Section \ref{morphisms}.
Yetter did not emphasize that the natural isomorphism is bounded, but
one can see from equation (\ref{bounded.natural.iso}) that it is.
\end{proof}

\begin{lemma}
\label{me.composition.1}
If $F \maps H^X \to H^Y$ is a $C^*$\!-equivalence,
then there is a measurable bijection between $X$ and $Y$,
and $F$ is a measurable functor.
\end{lemma}

\begin{proof}
This was proved by Yetter \cite[Thm.\ 40]{Yetter2}.   In fact,
Yetter failed to require that $F$ be linear on morphisms,
which is necessary for this result.  Careful examination of his
proof shows that it can be repaired if we include this extra condition,
which holds automatically for a $C^*$\!-equivalence.
\end{proof}

\begin{lemma}
\label{me.composition.2}
If $T\maps \HH \to \HH'$ is a measurable functor
and $F\maps \HH \to H^X$, $G\maps H^Y \to \HH'$
are {\rm arbitrary} $C^*$\!-equivalences,
then $T$ is naturally isomorphic to the composite
\[
\xymatrix{
   \HH \ar[r]^(.4){F}&
   H^X \ar[r]^{\tilde{T}}&
   H^Y \ar[r]^(.6){G}&
   \HH'
}
\]
for some matrix functor $\tilde{T}$.
\end{lemma}

\begin{proof} The proof is analogous to the proof
of Lemma \ref{composition.2}.  Since $T$ is measurable we know
there exist $C^*$\!-equivalences $F'\maps \HH \to H^{X'}$, $G'\maps
H^{Y'} \to \HH'$ such that $T$ is boundedly naturally isomorphic to
the composite
\[
\xymatrix{
   \HH \ar[r]^(.4){F'}&
   H^{X'} \ar[r]^{\tilde{T}'}&
   H^{Y'} \ar[r]^(.6){G'}&
   \HH'
}
\]
for some matrix functor $\tilde{T}'$.  By Lemma \ref{me.composition.1}
we may assume $X' = X$ and $Y' = Y$.  So, let $\tilde{T}$ be
the composite
\[
\xymatrix{
   H^X \ar[r]^(.6){\bar{F}}&
   \HH \ar[r]^(.4){F'}&
   H^X \ar[r]^{\tilde{T}'}&
   H^Y \ar[r]^(.6){G'}&
   \HH' \ar[r]^(.4){\bar{G}}&
   H^Y
}
\]
where the weak inverses $\bar{F}$ and $\bar{G}$ are chosen using the
fact that $F$ and $G$ are $C^*$\!-equivalences.  Since $F' \bar{F} \maps
H^X \to H^X$ and $\bar{G} G' \maps H^Y \to H^Y$ are
$C^*$\!-equivalences, they are matrix functors by Lemma
\ref{me.composition.1}.  It follows that $\tilde{T}$ is a composite of
three matrix functors, hence boundedly naturally isomorphic to a
matrix functor by Lemma \ref{me.composition.0}.  Moreover, the
composite
\[
\xymatrix{
   \HH \ar[r]^(.4){F}&
   H^X \ar[r]^{\tilde T}&
   H^Y \ar[r]^(.6){G}&
   \HH'
}
\]
is boundedly naturally isomorphic to $T$.
Since $F$ and $G$ are $C^*$\!-equivalences and $\tilde T$ is boundedly
naturally isomorphic to a matrix functor, it follows that $T$
is a measurable functor.
\end{proof}

\begin{lemma}
\label{me.composition.3}
A composite of measurable functors is measurable.
\end{lemma}

\begin{proof}
The proof is analogous to the proof of Lemma \ref{composition.3}.
Suppose we have a composable pair of measurable functors
$T\maps \HH \to \HH'$ and $U\maps \HH' \to \HH''$.  By definition,
$T$ is boundedly naturally isomorphic to a composite
\[
\xymatrix{
   \HH \ar[r]^(.4){F}&
   H^X \ar[r]^{\tilde{T}}&
   H^Y \ar[r]^(.6){G}&
   \HH'
}
\]
where $\tilde{T}$ is a matrix functor and $F$ and $G$ are $C^*$\!-equivalences.
By Lemma \ref{me.composition.2}, $U$ is naturally isomorphic
to a composite
\[
\xymatrix{
   \HH' \ar[r]^(.4){\bar{G}}&
   H^Y \ar[r]^{\tilde{U}}&
   H^X \ar[r]^(.55){H}&
   \HH''
}
\]
where $\tilde{U}$ is a matrix functor,
$\bar{G}$ is the chosen weak inverse for $G$, and $H$ is a
$C^*$\!-equivalence.   The composite $UT$ is thus
boundedly naturally isomorphic to
\[
\xymatrix{
   \HH \ar[r]^(.4){F}&
   H^X \ar[r]^{\tilde{U}\tilde{T}}&
   H^Z \ar[r]^(.55){H}&
   \HH''
}
\]
Since $\tilde{U}\tilde{T}$ is a matrix functor by Lemma
\ref{me.composition.0}, it follows that $UT$ is a measurable functor.
\end{proof}

%
\section{Representations on measurable categories}
\label{MeasRep}
%

With the material presented in the previous sections, we now have a
general framework to study representations of 2-groups on measurable
categories---that is, representations in the 2-category $\me$.
Unpacking this representation theory and seeing what it amounts to
concretely is now an essentially computational matter, which we
turn to in this section.

We begin by summarizing the main results.

\subsection{Main results}
\label{main_results}

Let us summarize our main results.  We now assume that $\G$ is a
skeletal 2-group.  In the crossed module description, since the
homomorphism $\d\maps H\to G$ is trivial, $\G$ simply amounts to an
{\em abelian} group $H$ and an action $\rhd$ of a group $G$ as
automorphisms of $H$.  We also assume that all the spaces and maps
involved are measurable.  Under these assumptions we can describe
representations of $\G$, as well as intertwiners and 2-intertwiners,
in terms of familiar geometric constructions---but living in the
category of measurable spaces, rather than smooth manifolds.

To understand these constructions, we first define
$H^*$ to be the set of measurable homomorphisms
\[  \chi \maps H \to \C^\times \]
where $\C^\times$ is the multiplicative group of nonzero complex
numbers.  The set $H^*$ becomes a group under pointwise multiplication:
\[   (\chi \xi)(h) = \chi(h) \xi(h) . \]
Under some mild conditions on $H$, $H^*$ is again a
measurable space, and its group operations are measurable.
The left action $\rhd$ of $G$ on $H$ naturally
induces a right action of $G$ on $H^*$, say $(\chi,g) \mapsto \chi_g$,
given by
\[      \chi_g[h] = \chi[g\rhd h].\]
This promotes $H^*$ to a right $G$-space.

Essentially---ignoring technical conditions on measures, and issues of
\alme-equivalence and categorical equivalence---we then have the following
dictionary relating representation theory to geometry:
\vskip 1em
\begin{center}
{
\begin{tabular}{c|c}
\hline
\\
     \xy (0,0)*{\txt{representation theory of a
 \\ skeletal 2-group $\G=(G,H,\rhd)$   }} \endxy    &   geometry
            \\
    &        \\     \hline \\
a representation of $\G$ on $H^X$  &  a right action of $G$ on $X$, and a map $X \to H^\ast$  \\ & making  $X$ a `measurable $G$-equivariant bundle' over $H^\ast$ \\
\\
an intertwiner between   &   a `$G$-equivariant measurable family of measures' $\mu_y$ on $X$, \\
representations on $H^X$ and $H^Y$       & and a `$G$-equivariant Hilbert space bundle' over $Y \times X$  \\
\\
a 2-intertwiner    &  a map of $G$-equivariant Hilbert space bundles \\ \\
\hline
\end{tabular}} \vskip 2em
\end{center}
Let us now explain this correspondence in more detail.

\subsubsection*{Representations}

Consider a representation $\rho \maps \G \to \me$ on a measurable
category $H^X$.  An essential step in understanding such a representation
is understanding what the measurable automorphisms of the category 
$H^X$ look like.  In Section~\ref{invertible}, we show that any
automorphism of $H^X$ is 2-isomorphic to one induced by pullback 
along some measurable automorphism $f\maps X\to X$.  Such an 
automorphism, which we denote $H^f\maps H^X\to H^X$, acts on
fields of Hilbert spaces and linear maps on $X$, simply by pulling them 
back along $f$.

In Thm.~\ref{thm:repclassify}, we show that if $\rho$ is a representation
on $H^X$ such that for each $g\in G$, $\rho(g)=H^{f_g}$ for some $f_g$,
then $\rho$ is determined, up to equivalence of representations, by:
\begin{itemize}
\item
a right action $\lhd$ of $G$ as measurable transformations of
the measurable space $X$,
\item a map $\chi \maps X \to H^*$ that is $G$-equivariant, i.e.:
\beq
\label{equivariance}
 \chi(x \lhd g)= \chi(x)_g
\eeq
for all $x \in X$ and $g \in G$.
\end{itemize}
Geometrically, this states that the map $\chi \maps X\to H^\ast$ is an  {\bf equivariant fiber bundle} over the `character group' $H^\ast = \hom(H,\C^\times)$:
$$
\xymatrix{X  \ar[d]^{\chi} \\ H^\ast}
$$
We define `measurable representations' of $\G$ to be ones of this form for
which both the map $\chi$ and the actions of $G$ on $X$ and $H^*$ are
measurable, where $H^\ast$ inherits a measurable structure from that of
$H$.  In the rest of this summary of results we consider only measurable
representations.

Two representations on $H^X$ are equivalent, by definition, if they are related
by a pair of intertwiners that are weak inverses of each other.  We discuss
general intertwiners and their geometry below; for now we merely mention that
{\em invertible} intertwiners between measurable representations correspond to
invertible measurable bundle maps:
\[
{\xygraph{
  []!{0;<1.5cm,0cm>:<.75cm, 1.3cm>::}
  []{X}="TM" :@{->}^{\sim} [r] {Y}
      :@{->}^{\chi^{\phantom{Y}}_2} [d] {H^\ast}="M"
        "TM"        :@{->}_{\chi^{\phantom{X}}_1} "M"
  }}
\]
So, equivalence of representations corresponds geometrically to
isomorphism of bundles.

We say that a representation is `indecomposable' if it is not
equivalent to a `2-sum' of nontrivial representations, where a `2-sum'
is a categorified version of the direct sum of ordinary group representations.
We say a representation is `irreducible' if, roughly speaking, it does not
contain any subrepresentations other than itself and the trivial
representation.  Irreducible representations are automatically
indecomposable, but not necessarily vice versa.  An ({\it a priori})
intermediate notion is that of an `irretractable representation'---a
representation $\rho$ such that if any composite of intertwiners
of the form
\[
\xymatrix{\rho' \ar[r]^{} & \rho\rule{0em}{.8em} \ar[r]^{} & \rho'}
\]
is equivalent to the identity intertwiner on $\rho'$, then $\rho'$ is
either trivial or equivalent to $\rho$.  While for ordinary group
representations irretractable representations are the same as
indecomposable ones, this is not true for 2-group representations in
$\me$.  We thus classify both the irretractable and indecomposable
2-group representations in $\me$.  The irreducible ones remain more
challenging: in particular, we do not know if every irretractable
representation is irreducible.

In Thm.~\ref{theo_indecomposable}
we show that a measurable representation of $\G$ on $H^X$
is indecomposable if and only if $G$ acts transitively on $X$.  The study of
indecomposable representations, and hence irreducible and irretractable
representations as special cases, is thus rooted in Klein's geometry of
homogeneous spaces. Recall that for any point $x^o \in
X$, the {\bf stabilizer} of $x^o$ is the subgroup
$S \subseteq G$ consisting of group elements $g$ with $x^o\lhd g = x^o$.
By a standard argument, we have
$$ X \cong G/S .  $$ 
Then, let $\chi^o = \chi(x^o)$.  By equation
(\ref{equivariance}), the image of $\chi\maps X \to H^\ast$ is a
single $G$-orbit in $H^\ast$, and $S$ is contained in the stabilizer
$S^\ast$ of $\chi^o$.  This shows that an indecomposable
representation essentially amounts to an equivariant map of
homogeneous spaces $\chi \maps G/S \to G/S^\ast$, where $S^\ast$ is
the stabilizer of some point in $H^\ast$, and $S\subseteq S^\ast$.  In
other words, indecomposable representations are classified up to
equivalence by a choice of $G$-orbit in $H^\ast$, along with a
subgroup $S$ of the stabilizer of a point $\chi^o$ in the orbit.

In Thm.~\ref{theo_irretractable}, we show an indecomposable
representation $\rho$ is irretractable if and only if $S$ is {\it
equal} to the stabilizer of $\chi^o$; irretractable representations
are thus classified up to equivalence by $G$-orbits in $H^\ast$.

\subsubsection*{Intertwiners}

Next we turn to the main results concerning intertwiners.  To state
these, we first need some concepts from measure theory.  Let $X$ be a
measurable space.  Recall that two measures $\mu$ and $\nu$ on $X$ are
\textbf{equivalent}, or in the same \textbf{measure class}, if they
have the same null sets.  Next, suppose $G$ acts on $X$ as measurable
transformations.  Given a measure $\mu$ on $X$, for each $g$ we define
the `transformed' measure $\mu^g$ by setting \beq \mu^g(A) := \mu(A
\lhd g^{-1}).  \eeq The measure is \textbf{invariant} if $\mu^g = \mu$
for every $g$. If $\mu^g$ and $\mu$ are only equivalent, we say that
$\mu$ is \textbf{quasi-invariant}. It is well-known that if $G$ is a
separable, locally compact topological group, acting measurably and
transitively on $X$, then there exist nontrivial quasi-invariant
measures on $X$, and moreover, all such measures belong to the same
measure class (see Appendix \ref{apx:G-spaces} for further details).

Next, let $X$ and $Y$ be two $G$-spaces. We may consider $Y$\!-indexed
families $\mu_y$ of measures on $X$. Such a family is
\textbf{equivariant}\footnote{Since we do not require {\em equality}, a more 
descriptive term would be `quasi-equivariance'; we stick to `equivariance' for simplicity.}
 under the action of $G$ if for all $g$,
$\mu_{y\lhd g}$ is {\em equivalent} to $\mu_y^g$.

With these definitions we can now give a concrete description of
intertwiners.
Suppose $\rho_1$ and $\rho_2$ are measurable representations of a
skeletal 2-group $\G$ on measurable categories $H^X$ and $H^Y$,
respectively, with corresponding equivariant bundles $\chi_1$ and $\chi_2$:
\[
{\xygraph{
  []!{0;<1.5cm,0cm>:<.75cm, 1.3cm>::}
  []{X}="TM"  [r] {Y}
      :@{->}^{\chi^{\phantom{Y}}_2} [d] {H^\ast}="M"
        "TM"        :@{->}_{\chi^{\phantom{X}}_1} "M"
  }}
\]
Then an intertwiner $\phi \maps \rho_1 \to \rho_2$ is specified, up to equivalence, by:
\begin{itemize}
\item
an equivariant $Y$\!-indexed family of measures $\mu_y$
on $X$, with each $\mu_y$ supported on $\chi_1^{-1}(\chi_2(y))$.
\item
an assignment, for each $g\in G$ and all $y$, of a $\mu_y$-class of
Hilbert spaces $\phi_{y,x}$ and linear
maps
$$\Phi^g_{y,x} \maps \phi_{y,x} \to \phi_{(y,x) \lhd g^{-1}}$$
satisfying the cocycle conditions
$$\Phi^{g'g}_{y,x} = \Phi^{g'}_{(y,x)\lhd g^{-1}} \Phi^{g}_{y,x}
\quad \text{and}\quad
\Phi^1_{y,x}= 1_{\phi_{y,x}}
$$
$\mu$-\alme\ for each pair $g,g'\in G$,
where $(y, x) \lhd g$ is short for $(y \lhd g,x \lhd g)$.
\end{itemize}

There is a more geometric way to think of these intertwiners.  For simplicity, assume that, among the measure class of fields of linear operators
\[
  \Phi^g_{y,x} \maps \phi_{y,x} \to \phi_{(y,x) \lhd g^{-1}}
\]
we may choose a representative such that the cocycle conditions hold everywhere in $Y\times X$ and for all $g \in G$.   We then think of the union of all the Hilbert spaces:
\[
\phi  = \coprod_{(y,x)}  \phi_{y,x}
\]
as a bundle of Hilbert spaces over the product space $Y \times X$.  The group $G$ acts on both the total space and the base space of this bundle.  Indeed, the maps $\Phi^g_{y,x}$ give a map $\Phi^g\maps \phi \to \phi$; the cocycle conditions then become
\[
     \Phi^{g'g} = \Phi^{g'}\Phi^g
     \quad \text{and}\quad
     \Phi^1_{y,x}= 1_{\phi_{y,x}}
\]
which are simply the conditions that $\phi \mapsto \Phi^g \phi$ define a left action of $G$ on $\phi$.  If we turn this into a right action by defining
\[
    \phi^g = \Phi^{g^{-1}}(\phi)
\]
we find that the bundle map is equivariant with respect to this action of $G$ on $\phi$ and the diagonal action of $G$ on $Y\times X$.

It is thus helpful to think of an intertwiner $\phi \maps \rho_1 \to \rho_2$ 
as being given by an equivariant family of measures $\mu_y$ and
a $\mu$-class of $G$-equivariant bundles of Hilbert spaces $\phi_{y,x}$ 
over $Y\times X$.  We emphasize that it is not clear this picture is completely 
accurate for arbitrary intertwiners, particularly when $G$ is an uncountable 
group, since there are separate cocycle equations for each pair $g,g'\in G$, 
each holding only almost everywhere. However, it is a useful heuristic 
picture, and can be made precise at least in important special cases.

As with representations, we introduce and discuss the notions of
reducibility, retractability, and decomposability for intertwiners.

\subsubsection*{2-Intertwiners}

Finally, the main results concerning the 2-intertwiners are as
follows. Consider a pair of representations $\rho_1$ and
$\rho_2$ of the skeletal 2-group $\G$ on the measurable
categories $H^X$ and $H^Y$, and two intertwiners $\phi, \psi \maps
\rho_1 \To \rho_2$.  Suppose $\phi = (\mu, \phi , \Phi)$ and
$\psi = (\nu, \psi, \Psi)$.  For any $y$, we denote by $\sqrt{\mu_y
\nu_y}$ the geometric mean of the measures $\mu_y$ and $\nu_y$. A
2-intertwiner turns out to consist of:
\begin{itemize}
\item an assignment, for each $y$, of a $\sqrt{\mu_y
\nu_y}$-class of linear maps $m_{y,x}\maps \phi_{y,x} \to \psi_{y,x}$,
which satisfies the intertwining rule
\[
\Psi^g_{y,x} \,m_{y,x} = m_{(y,x)g^{-1}} \,\Phi^g_{y,x}
\]
$\sqrt{\mu\nu}$\alme
\end{itemize}

In the geometric picture of intertwiners as equivariant bundles of Hilbert 
spaces, this characterization of a 2-intertwiner simply amounts to a 
{\bf morphism of equivariant bundles}, up to almost-everywhere equality.

The intertwiners satisfy an analogue of \textit{Schur's lemma}.
Namely, in Prop.~\ref{schur-intertwiner} we show that under some mild
technical conditions, any 2-intertwiner between irreducible
intertwiners is either null or an isomorphism.

%
\subsection{Invertible morphisms and 2-morphisms in $\me$}
\label{invertible}
%

A 2-group representation $\rho$ gives {\em invertible}
morphisms $\rho(g)$ and {\em invertible} 2-morphisms $\rho(g,h)$
in the target 2-category.  To understand 2-group representations in $\me$,
it is thus a useful preliminary step to characterize invertible measurable
functors and invertible measurable natural transformations.  We address
these in this section, beginning with the 2-morphisms.

Consider two parallel measurable functors $T$ and $T'$.  A measurable
natural transformation $\alpha \maps T \To T'$
is \textbf{invertible} if it has a vertical
inverse, namely a measurable natural transformation $\alpha'\maps T' \To T$
such that $\alpha'\cdot\alpha = \unit_T$ and $\alpha\cdot \alpha' = \unit_{T'}$.
We often call the invertible 2-morphism $\alpha$ in $\me$ a {\bf 2-isomorphism}, 
for short; we also say $T$ and $T'$ are {\bf 2-isomorphic}.  The following theorem 
classifies 2-isomorphisms in the case where $T$ and $T'$ are matrix
functors.
%
%
\begin{theo} \label{invert2}
Let $(T,t), (T', t')\maps H^X\to H^Y$ be matrix functors.  Then $(T,t)$ and $(T',t')$ 
are boundedly naturally isomorphic if and only if the measures $t_y$ and $t'_y$ 
are equivalent, for every $y$, and there is a measurable field of bounded linear 
operators $\alpha_{y,x}\maps T_{y,x} \to T'_{y,x}$ such that $\alpha_{y,x}$ is an 
isomorphism for each $y$ and $t_y$-\alme\ in $x$.  In this case, there is one 
2-isomorphism $T\To T'$ for each $t$-class of fields $\alpha_{y,x}$.
\end{theo}
%
%
%
\begin{proof}
Suppose $\alpha\maps T \To T'$ is a bounded natural isomorphism, with 
inverse $\alpha'\maps T' \To T$.  By Lemma \ref{magic.lemma}, $\alpha$ and 
$\alpha'$ are both matrix natural transformations, hence defined by fields of 
bounded linear operators $\alpha_{y,x}$ and $\alpha'_{y,x}$ on $Y\times X$.  
By the composition formula (\ref{vert2mor}), the composite $\alpha'\cdot\alpha 
= \unit_T$ is given by
\[
\left(\alpha'\cdot \alpha\right)_{y,x} =
            \sqrnd{t'_y}{t_y}
            \sqrnd{t_y}{t'_y} \;
\alpha'_{y,x} \alpha_{y,x}=
 \unit_{T_{y,x}} \qquad
   \txt{$t_y$-\alme}
\]
We know by the chain rule (\ref{chainrule}) that the product of Radon-Nikodym 
derivatives in this formula equals one $\sqrt{t_yt'_y}$-\alme, but not yet that 
equals one $t_y$-\alme\   However, by definition of the morphism $(T,t)$, the 
Hilbert spaces $T_{y,x}$ are non-trivial $t_y$-\alme; hence $\unit_{T_{y,x}} \neq 0$. 
This shows that the product of Radon-Nikodym derivatives above is 
$t_y$-\alme\ nonzero; in particular,
\[
\frac{\extd {t^{'t}_y}}{\extd t_y}(x) \not= 0 \qquad  \txt{$t_y$-\alme}
\]
where ${t'}^t_y$ denotes the absolutely continuous part of $t'_y$ in its 
Lebesgue decomposition $t'_y  = t^{'t}_y + \overline{t^{'t}_y}$ with respect 
to $t_y$. But this property is equivalent to the statement that the measure $t_y$ 
is absolutely continuous with respect to $t'_y$.  To check this, pick a measurable 
set $A$ and write
\[
t'_y(A) = \int_A \extd t_y(x) \frac{\extd {t^{'t}_y}}{\extd t_y}(x) + \overline{t^{'t}_y}(A)
\]
Now if $t'_y(A) = 0$, both terms of the right-hand-side of this equality 
vanish---in particular the integral term.  But since the Radon-Nikodym derivative 
is a strictly positive function $t_y$-\alme, this requires the $t_y$-measure of $A$ 
to be zero. So we have shown that $t'_y(A) = 0$ implies $t_y(A) = 0$  for any 
measurable set $A$, i.e. $t_y\ll t'_y$.  Starting with $\alpha\cdot \alpha' = \unit_{T'}$, 
the same analysis leads to the conclusion $t_y\ll t'_y$.  Hence the two measures 
are equivalent. From this it is immediate that
\[
 \left(\alpha'\cdot \alpha\right)_{y,x} =
 \alpha'_{y,x} \alpha_{y,x} =
 \unit_{T_{y,x}} \qquad
   \txt{$t_y$-\alme}
\]
and thus $\alpha'_{y,x} = \alpha_{y,x}^{-1}$.  In particular, the 
operators $\alpha_{y,x}$ are invertible $t_y$-\alme\

Conversely, suppose the measures $t_y$ and $t'_y$ are equivalent and 
we are given a measurable field $\alpha\maps T\to T'$ such that for all $y$, 
the operators $\alpha_{y,x}$ are invertible for almost every $x$.  It is easy to 
check, using the formula for vertical composition, that the matrix natural 
transformation defined by $\alpha_{y,x}$ has inverse defined by $\alpha_{y,x}^{-1}$.
\end{proof}

A morphism $T\maps H^X\to H^Y$ is \textbf{strictly invertible} if it
has a strict inverse, namely a 2-morphism $U\maps H^Y\to H^X$ such
that $UT = \unit_X$ and $ TU = \unit_Y$.  In 2-category theory,
however, it is more natural to weaken the notion of invertibility, so
these equations hold only \textit{up to 2-isomorphism}. In this case
we say that $T$ is \textbf{weakly invertible} or an
\textbf{equivalence}.

We shall give two related characterizations of weakly invertible 
morphisms in $\me$.  For the first one, recall that if $f\maps Y \to X$ 
is a measurable function, then any measure $\mu$ on $Y$ pushes 
forward to a measure $f_\ast \mu$ on X, by
\[
       f_\ast \mu (A) = \mu(f^{-1}A)
\]
for each measurable set $A\subseteq X$.  In the case where $\mu = \delta_y$, we have
\[
   f_\ast \delta_y =  \delta_{f(y)}
\]
Denoting by $\delta$ the $Y$\!-indexed family of measures $y\mapsto \delta_y$ on $Y$, the following theorem shows that every invertible matrix functor $T\maps H^X\to H^Y$ is essentially $(\C,f_\ast \delta)$ for some invertible measurable map $f\maps Y \to X$.

As shown by the following theorem, the
condition for a morphism to be an equivalence is very restrictive
\cite{Yetter2}:
%
%
\begin{theo} \label{invert}
A matrix functor $(T, t)\maps H^X\to H^Y$ is a measurable equivalence if and only if there is an invertible measurable function $f\maps Y \to X$ between the underlying spaces such that, for all $y$, the measure $t_{y}$ is equivalent to $\delta_{f(y)}$, and a measurable field of linear operators from $T_{y,x}$ to the constant field $\C$ that is $t_y$-\alme\ invertible.
\end{theo}
%
%
\begin{proof}
If $(T, t)$ is an equivalence, it has weak inverse that is also a matrix functor, say $(U, u)$.  The composite $UT$ is 2-isomorphic to the identity morphism $\unit_X$, and $TU$ is 2-isomorphic to $\unit_Y$.  Since $\unit_X\maps H^X \to H^X$ is 2-isomorphic to the matrix functor $(\C,\delta_x)$,  and similarly for $\unit_Y$, Thm.~\ref{invert2} implies that the composite measures $ut$ and $tu$ are equivalent to Dirac measures:
\[
(ut)_x = \int_Y \extd u_x(y) \, t_y \sim \delta_x
\qquad (tu)_y = \int_X \extd t_y(x) \, u_x \sim \delta_y
\]
An immediate consequence is that the measures $u_x$ and $t_y$ must be non-trivial, for all $x$ and $y$.  Also, for all $x$, the subset $X-\{x\}$ has zero $(ut)_x$-measure
\[
\int_Y \extd u_x(y) \, t_y(X-\{x\}) = 0
\]
As a result the nonnegative function $y \mapsto t_y(X-\{x\})$ vanishes $u_x$-almost everywhere. This means that, for all $x$ and $u_x$-almost all $y$, the measure $t_y$ is equivalent to $\delta_x$. Likewise, we find that, for all $y$ and $t_y$-almost all $x$, the measure $u_x$ is equivalent to $\delta_y$.

Let us consider further the consequences of these two properties, by
fixing a point $y_0 \in Y$. For $t_{y_0}$-almost every $x$, we know,
on one hand, that $t_y \sim \delta_x$ for $u_x$-almost all $y$ (since
this actually holds for all $x$), and on the other hand, that $u_x
\sim \delta_{y_0}$ . It follows that for $t_{y_0}$-almost every $x$,
we have $t_{y_0} \sim \delta_x$. The measure $t_{y_0}$ being
non-trivial, this requires $t_{y_0} \sim \delta_{f(y_0)}$ for at least
one point $f(y_0) \in X$; moreover this point is unique, because two
Dirac measures are equivalent only if they charge the same point. This
defines a function $f: Y \to X$ such that $t_y$ is equivalent to
$\delta_{f(y)}$. Likewise, we can define a function $g: X \to Y$ such
that $u_x$ is equivalent to $\delta_{g(x)}$. Finally, by expressing
the composite measures in terms of Dirac measures, we get $fg=\unit_X$
and $gf = \unit_Y$, establishing the invertibility of the function
$f$.

The measurability of the function $f$ can be shown as
follows. Consider a measurable set $A \subseteq X$. Since the family of
measures $t_y$ is measurable, we know the function $y \mapsto
t_y(A)$ is measurable. Since $t_y(A) = \delta_{f(y)}$, so this function
is given by:
\[
y \mapsto t_y(A) =
\begin{cases}
1 \quad \mbox{if} \quad y \in f^{-1}(A) \\
0 \quad \mbox{if not}
\end{cases}
\]
This coincides with the characteristic function of the set $f^{-1}(A)\subseteq Y$, which is measurable precisely when $f^{-1}(A)$ is measurable.  Hence, $f$ is measurable.

Finally, we can use (\ref{convol_1maps}) to compose the fields $U_{x,y}$ and $T_{y,x}$.  Since $(tu)_y \sim \delta_y$, the only essential components of the composite field are the diagonal ones:
$$
(TU)_{y,y} = \direct _X \extd k_{y,y}(x) \,T_{y,x}\otimes U_{x,y}.
$$
Applying (\ref{integration_meas}) in this case, we find that the measures $k_{y,y}$ are defined by the property
$$
\int_{X} \extd k_{y,y}(x) F(x,y) = \int_X \extd \delta_{f(y)}(x) \int_Y \delta_{g(x)}(y) F(x,y)
$$
for any measurable function $F$ on $X\times Y$. From this we obtain $k_{y,y} = \delta_{f(y)}$ and $(TU)_{y,y} = T_{y, f(y)} \otimes U_{f(y), y}$ for all $y\in Y$.  Since we know $TU$ is 2-isomorphic to the matrix functor $(\C,\delta_y)$, we therefore obtain
\[
      (TU)_{y,y} = T_{y, f(y)} \otimes U_{f(y), y} \cong \C \quad \forall y\in Y.
\]
where the isomorphism of fields is measurable.  This can only happen if each factor in the tensor product is measurably isomorphic to the constant field $\C$.

Conversely, if the measures $t_y$ are equivalent to $\delta_{f(y)}$ for an invertible measurable function $f$, and if $T_{y, f(y)} \cong \C$, construct a matrix functor $U\maps H^Y\to H^X$ from the family of measures $\delta_{f^{-1}(x)}$ and the constant field $U_{x, y} = \C$.  One can immediately check that $U$ is a weak inverse for $T$.
\end{proof}

Taken together, these theorems have the following corollary:

\begin{cor}
\label{representative.equivalences}
If $T\maps H^X \to H^Y$ is a weakly invertible measurable functor, there is a unique measurable isomorphism $f\maps Y\to X$ such that $T$ is boundedly naturally isomorphic to the matrix functor $(\C,\delta_{f(y)})$.
\end{cor}
\begin{proof}
Any measurable functor is boundedly naturally isomorphic to a matrix functor, say $T\cong(T_{y,x},t_y)$.  By Thm.~\ref{invert}, we may in fact take $T_{y,x} = \C$ and $t_y=\delta_{f(y)}$ for some measurable isomorphism $f\maps Y \to X$.  By Thm.~\ref{invert2}, two such matrix functors, say $(\C,\delta_{f(y)})$ and $(\C,\delta_{f'(y)})$ are boundedly naturally isomorphic if and only if $f=f'$, so the choice of $f$ is unique.
\end{proof}

We have classified measurable equivalences by giving one representative---a specific {\em matrix} equiv\-alence---of each 2-isomorphism class.  These representatives are quite handy in calculations, but they do have one drawback: matrix functors are not strictly closed under composition.  In particular, the composite of two of our representatives $(\C,f_\ast \delta)$ is isomorphic, but not equal, to another of this form.
While in general this is the best we might expect, it is natural to wonder whether these 2-isomorphism classes have a set of representations that is closed under composition.  They do.

If $X$ and $Y$ are measurable spaces, any measurable function
\[
    f\maps Y \to X
\]
gives a functor $H^f$ called the {\bf pullback}
\[
   H^f \maps H^X \to H^Y
\]
defined by pulling back measurable fields of Hilbert spaces and linear operators along $f$.  Explicitly, given a measurable field of Hilbert spaces $\H\in H^X$, the field $H^f \H$ has components
\[
    (H^f\H)_y = \H_{f(y)}
\]
Similarly, for $\phi\maps \H \to \H'$ a measurable field of linear operators on $X$,
\[
    (H^f\phi)_y = \phi_{f(y)}.
\]
It is easy to see that this is functorial; to check that $\H^f$ is a {\em measurable} functor, we note that it is boundedly naturally isomorphic to the matrix functor $(\C,\delta_{f(y)})$, which sends an object $\H\in H^X$ to
\[
   \direct_X \extd \delta_{f(y)}(x)\, \C\tensor \H_x \cong \H_{f(y)} = (H^f\H)_y
\]
and does the analogous thing to morphisms in $H^X$.
The obvious isomorphism in this equation is natural, and has unit norm, so is bounded.

\begin{prop}
\label{isomorphic_to_pullback}
If $T\maps H^X \to H^Y$ is a weakly invertible measurable functor, there exists a unique measurable isomorphism $f\maps Y\to X$ such that $T$ is boundedly naturally isomorphic to the pullback $H^f$.
\end{prop}
\begin{proof}
Any measurable functor from $H^X$ to $H^Y$ is equivalent to some matrix functor; by Cor.~\ref{representative.equivalences}, this matrix functor may be taken to be $(\C,\delta_{f(x)})$ for a {\em unique} isomorphism of measurable spaces $f\maps Y \to X$.  This matrix functor is 2-isomorphic to $H^f$.
\end{proof}

While the pullbacks $H^f$ are closely related to the matrix functors $(\C,f_\ast\delta)$, the former have several advantages, all stemming from the basic equations:
\beq
\label{pullback composition}
   H^{1_X} = \unit_{H^X}  \quad \text{and} \quad H^f H^g = H^{gf}
\eeq
In particular, composition of pullbacks is {\em strictly} associative, and each pullback $H^f$ has {\em strict} inverse $H^{f^{-1}}$.  In fact, there is a 2-category $M$ with measurable spaces as objects, {\em invertible} measurable functions as morphisms, and only identity 2-morphisms.  The assignments $X\mapsto H^X$ and $f\mapsto H^f$ give a contravariant 2-functor $M\to\me$.  The forgoing analysis shows this 2-functor is faithful at the level of 1-morphisms.

If $f, f'$ are distinct measurable isomorphisms, the measurable functors $H^f$ and $H^{f'}$ are never 2-isomorphic.  However, each $H^f$ has many 2-automorphisms:
\begin{theo}
\label{pullback.2-autos}
Let $f\maps Y\to X$ be an isomorphism of measurable spaces, and $H^f\maps H^X\to H^Y$ be its pullback.  Then the group of 2-automorphisms of $H^f$ is isomorphic to the group of measurable maps $Y \to \C^\times$, with pointwise multiplication.
\end{theo}
\begin{proof}
Let $\alpha$ be a 2-automorphism of $H^f$, where $f$ is invertible.
\[
  \xymatrix{
  H^X\ar@/^2ex/[rr]^{H^f}="g1"\ar@/_2ex/[rr]_{H^f}="g2"&& H^Y
  \ar@{=>}^{\alpha} "g1"+<0ex,-2ex>;"g2"+<0ex,2ex>
}
\]
Using the 2-isomorphism $\beta\maps H^f\To (\C,f_\ast\delta)$, we can write $\alpha$ as a composite
\[
    \alpha = \beta^{-1}\cdot \tilde\alpha\cdot \beta
\]
By Thm.~\ref{invert2}, $\tilde \alpha \maps (\C,f_\ast\delta) \To (\C,f_\ast\delta)$ is necessarily a matrix functor given by a measurable field of linear operators $\tilde\alpha_{y,x}\maps \C \to \C$, defined and invertible $\delta_{f(y)}$-\alme\ for all $y$.  Such a measurable field is just a measurable function $\tilde\alpha \maps Y \times X \to \C$, with $\tilde\alpha_{y,f(y)}\in \C^\times$.  From the definition of matrix natural transformations, we can then compute for each object $\H\in H^X$, the morphism $\alpha_\H\maps H^f\H \to H^f \H$.   Explicitly,
\[
\xymatrix@C=4em@R=.2em@M=.8em{
\H_{f(y)} \ar[r]^{\beta_\H\phantom{uuuu}} & \direct \extd\delta_{f(y)}(x) H_x
\ar[r]^{\tilde\alpha_\H} & \direct \extd\delta_{f(y)}(x) H_x
\ar[r]^{\phantom{uuu}\beta_\H^{-1}} & \H_{f(y)}
\\
\psi_{f(y)}\ar @{|->} [r] & \direct \extd\delta_{f(y)}(x) \psi_x
\ar @{|->} [r] & \direct \extd\delta_{f(y)}(x)\alpha_{y,x} \psi_x
\ar @{|->} [r] & \tilde \alpha_{y,f(y)}\psi_{f(y)}
}
\]
So, the natural transformation $\alpha$ acts via multiplication by
\[
    \alpha(y) := \tilde \alpha_{y,f(y)} \in \C^\times.
\]
It is easy to show that $\alpha(y)\maps Y \to \C^\times$ is measurable, since $\tilde \alpha$ and $f$ are both measurable.

Conversely, given a measurable map $\alpha(y)$, we get a 2-automorphism $\alpha$ of $H^f$ by letting
\[
      \alpha_\H\maps H^f \H \to H^f \H
\]
be given by
\begin{align*}
      (\alpha_\H)_y\maps \H_{f(y)} &\to \H_{f(y)} \\
        \psi_y &\mapsto \alpha(y)\psi_y
\end{align*}
One can easily check that the procedures just described are inverses, so we get a one-to-one correspondence.  Moreover, composition of 2-automorphisms $\alpha_1$, $\alpha_2$, corresponds to multiplication of the functions $\alpha_1(y)$, $\alpha_2(y)$, so this correspondence gives a group isomorphism.
\end{proof}

It will also be useful to know how to compose pullback 2-automorphisms horizontally:
\begin{prop}
\label{horiz.2auto}
Let $f\maps Y\to X$ and $g\maps Z\to Y$ be measurable isomorphisms, and consider the following diagram in $\me$:
\[
  \xymatrix{
  H^X\ar@/^2ex/[rr]^{H^f}="g1"\ar@/_2ex/[rr]_{H^f}="g2"&& H^Y
  \ar@/^2ex/[rr]^{H^g}="g3"\ar@/_2ex/[rr]_{H^g}="g4"&& H^Z
  \ar@{=>}^{\alpha} "g1"+<0ex,-2.5ex>;"g2"+<0ex,2.5ex>
  \ar@{=>}^{\beta} "g3"+<0ex,-2.5ex>;"g4"+<0ex,2.5ex>
}
\]
where $\alpha$ and $\beta$ are 2-automorphisms corresponding to measurable maps
\[
\alpha \maps Y \to \C^\times \quad \text{and} \quad \beta \maps Z \to \C^\times
\]
as in the previous theorem.  Then the horizontal composite $\beta\circ \alpha$ corresponds to the measurable map from $Z$ to $\C^\times$ defined by
\[
     (\beta\circ \alpha)(z) = \beta(z)\alpha(g(z))
\]
\end{prop}
\begin{proof}
This is a straightforward computation from the definition of horizontal composition.
\end{proof}

\subsection{Structure theorems}

We now begin the precise description of the representation theory, as outlined
in Section \ref{main_results}.  We first give
the detailed structure of representations, followed by that of intertwiners and 2-intertwiners.

%
\subsubsection{Structure of representations}
%

Given a generic 2-group $\mathcal{G} = (G, H, \rhd, \d)$, we are
interested in the structure of a representation $\rho$ in the target
2-category $\me$.  Since any object of $\me$ is $C^\ast$-equivalent to
one of the form $H^X$, we shall assume that
\[
\rho(\star) = H^X
\]
for some measurable space $X$.  The representation $\rho$ also gives,
for each $g\in G$, a morphism $\rho(g)\maps H^X \to H^X$, and we
assume for now that all of these morphisms are pullbacks of measurable
automorphisms of $X$.

%
\begin{theo}[Representations] \label{thm:repclassify}
Let $\rho$ be a representation of $\G = (G,H,\d,\rhd)$ on $H^X$, and assume that each $\rho(g)$  is of the form $ H^{f_g}$ for some $f_g\maps X\to X$.  Then $\rho$ is determined uniquely by:
\begin{itemize}
  \item a right action $\lhd$ of $G$ as measurable transformations of $X$, and
  \item an assignment to each $x\in X$ of a group
homomorphism $\chi(x)\maps H \to \C^\times$.
\end{itemize}
satisfying the following properties:
\begin{romanlist}
\item for each  $h\in H$, the function $x\mapsto \chi(x)[h]$ is measurable
\item any element of the image of $\d$ acts trivially on $X$ via $\lhd$.
\item the field of homomorphisms is equivariant
under the actions of $G$ on $H$ and $X$:
\[
\chi(x)[g \rhd h] = \chi(x \lhd g) [h].
\]

\end{romanlist}
\end{theo}
%
%
%
\begin{proof}
Consider a representation $\rho$ on $H^X$ and suppose that for each $g$,
\[
    \rho(g) = H^{f_g},
\]
where $f_g\maps X\to X$ is a measurable isomorphism.
Thanks to the strict composition laws (\ref{pullback composition}) for such 2-morphisms, the conditions that $\rho$ respect composition of morphisms and the identity morphism, namely
\[
   \rho(g'g) = \rho(g')\, \rho(g)  \quad \text{and} \quad  \rho(1) = \unit_{H^X}
\]
can be expressed as conditions on the functions $f_g$:
\beq
\label{rightaction}
     f_{g'g} = f_{g}f_{g'}  \quad \text{and} \quad f_1 = 1_X.
\eeq
Introducing the notation
$x \lhd g = f_{g}(x) $, these equations can be rewritten
\[
x\lhd g'g = (x \lhd g') \lhd g  \quad \text{and} \quad  x \lhd 1 = x
\]
Thus, the mapping $(x,g)\mapsto x\lhd g$ is a right action of $G$ on
$X$.

Next, consider a 2-morphism $\rho(u)$, where $u=(g, h)$ is a 2-morphism in $\G$.   Since $u$ is invertible, so is $\rho(u)$.  In particular, applying $\rho$ to the 2-morphism $(1,h)$, we get a 2-isomorphism $\rho(1,h)\maps \unit_{H^X} \To H^{f_{\d h}}$ for each $h\in H$.  Such 2-isomorphisms exists only if $f_{\d h} = 1_X$ for all $h$; that is,
$$ x \lhd \d(h) = x $$ for all $x\in X$ and $h\in H$.  Thus, the image $\d(H)$ of the homomorphism $\d$ fixes every element $x\in X$ under the action $\lhd$.

For arbitrary, $u\in G\times H$, Thm.~\ref{pullback.2-autos} implies $\rho(u)$ is given by a measurable function on $X$, which we also denote by $\rho(g,h)$:
\[
   \rho(g,h)\maps X \to \C.
\]
We can derive
conditions on the these functions from the requirement that $\rho$ respect both kinds of composition of 2-morphisms.

First, by Thm.~\ref{pullback.2-autos}, vertical composition corresponds to pointwise multiplication of functions, so the condition (\ref{2comp_vert}) that $\rho$ respect vertical composition becomes:
\beq
\label{char.vert.comp}
\rho(g,h'h)(x) = \rho(\d h g,h')(x)\, \rho(g,h)(x).
\eeq

Similarly, using the formula for horizontal composition provided by Prop.~\ref{horiz.2auto}, we obtain
\beq
\label{char.hor.comp}
\rho(g'g,h'(g'\rhd h))(x) = \rho(g',h')(x)\, \rho(g,h)(x\lhd g').
\eeq
Applying this formula in the case $g' = 1$ and $h = 1$, we find that the functions $\rho(g,h)$ are independent of $g$:
\[
\rho(g,h)(x) = \rho(1,h)(x)
\]
This allows a drastic simplification of the formula for vertical composition (\ref{char.vert.comp}).  Indeed, if we define
\beq
 \chi(x)[h] = \rho(1,h)(x),
 \eeq
then (\ref{char.vert.comp}) is simply the statement that $h\mapsto \chi(x)[h]$ is a homomorphism for each $x$:
\[
     \chi(x)[h'h] = \chi(x)[h'] \, \chi(x)[h].
\]
To check that the field of homomorphisms $\chi(x)$ satisfies the equivariance
property
\beq \label{covcharacters}
\chi(x\lhd g) [h] = \chi(x)[g \rhd h] ,
\eeq
one simply uses (\ref{char.hor.comp}) again, this time with $g=h'=1$.

To complete the proof, we show how to reconstruct the representation $\rho\maps \G \to \me$, given the measurable space
$X$, right action of $G$ on $X$, and field $\chi$ of homomorphisms
from $H$ to $\C^\times$.   This is a straightforward task.  To the unique
object of our 2-group, we assign $H^X\in \me$.   If $g\in G$ is a morphism in $\G$, we let $\rho(g)=H^{f_g}$, where $f_g(x)=x\lhd g$; if $u=(g,h)\in G\times H$ is a 2-morphism in $\G$, we let $\rho(u)$ be the automorphism of $H^{f_g}$ defined by the measurable function $x\mapsto \chi(x)[h]$.
\end{proof}

This theorem suggests an interesting question: is every representation
of $\G$ on $H^X$ equivalent to one of the above type?  As a weak piece
of evidence that the answer might be `yes', recall from Prop.\
\ref{isomorphic_to_pullback} that any invertible morphism from $H^X$
to itself is isomorphic to one of the form $H^f$.  However, this fact
alone is not enough.

The above theorem also suggests that we view representations of 2-groups
in a more geometric way, as equivariant bundles.  In a representation
of a 2-group $\G$ on $H^X$, the assignment $x \mapsto \chi(x)$ can be
viewed as promoting $X$ to the total space of a kind of bundle over
the set $\hom(H,\C^\times)$ of homomorphisms from $H$ to $\C^\times$:
\[
\xymatrix{X  \ar[d]^{\chi} \\ \hom(H,\C^\times)}
\]
Here we are using `bundle' in a very loose sense: no topology is
involved. The group $G$ acts on both the total space and the
base of this bundle: the right action $\lhd$ of $G$ on $X$ comes from
the representation, while its left action $\rhd$ on $H$
induces a right action $(\chi,g) \mapsto \chi_g$ on $\hom(H,\C^\times)$,
where
\[\chi_g[h] = \chi[g\rhd h].\]
The equivariance property in Thm.~\ref{thm:repclassify} means that the
map $\chi$ satisfies
\[    \chi(x\lhd g)=\chi(x)_g.  \]
So, we say $\chi \maps X \to \hom(H,\C^\times)$
is a `$G$-equivariant bundle'.

So far we have ignored any measurable structure on the groups $G$ and
$H$, treating them as discrete groups.  In practice these groups will
come with measurable structures of their own, and the maps involved
in the 2-group will all be measurable.  For such 2-groups the interesting
representations will be the `measurable' ones, meaning roughly that
all the maps defining the above $G$-equivariant bundle are measurable.

To make this line of thought precise, we need a concept of
`measurable group':

\begin{defn}
\label{defn:measurable_group}
We define a {\bf measurable group} to be a topological group
whose topology is locally compact, Hausdorff, and second countable.
\end{defn}

Varadarajan calls these {\bf lcsc groups}, and his book is an
excellent source of information about them \cite{Varadarajan}.  By
Lemma \ref{lem:standard_Borel}, they are a special case of {\it Polish
groups}: that is, topological groups $G$ that are homeomorphic to
complete separable metric spaces.  For more information on Polish
groups, see the book by Becker and Kechris \cite{BeckerKechris}.

It may seem odd to define a `measurable group' to be a special sort of
{\it topological} group.  The first reason is that every measurable
group has an underlying measurable space, by Lemma
\ref{lem:standard_Borel}.  The second is that by Lemma
\ref{lem:automatic_continuity}, any measurable homomorphism between
measurable groups is automatically continuous.  This implies that the
topology on a measurable group can be uniquely reconstructed from its
group structure together with its $\sigma$-algebra of measurable
subsets.

Next, instead of working with the set $\hom(H,\C^\times)$ of {\it all}
homomorphisms from $H$ to $\C^\times$, we restrict attention to the
{\it measurable} ones:

\begin{defn} If $H$ is a measurable group, let $H^*$ denote the set
of measurable (hence continuous) homomorphisms $\chi \maps H \to \C^\times$.
\end{defn}

We make $H^*$ into a group with pointwise multiplication as the
group operation:
\[   (\chi \chi')[h] = \chi[h] \, \chi'[h] . \]
$H^*$ then becomes a topological group with the compact-open topology.
This is the same as the topology where $\chi_\alpha \to \chi$ when
$\chi_\alpha(h) \to \chi(h)$ uniformly for $h$ in any fixed compact
subset of $H$.

Unfortunately, $H^*$ may not be a measurable group!  An example is the
free abelian group on countably many generators, for which $H^*$ fails
to be locally compact.  However, $H^*$ is measurable when $H$ is a
measurable group with finitely many connected components.  For more
details, including a necessary and sufficient condition for $H^*$
to be measurable, see Appendix \ref{apx:measurable_groups}.

In our definition of a `measurable 2-group', we will demand that $H$
and $H^*$ be measurable groups.  The left action of $G$ on $H$ gives a
right action of $G$ on $H^*$:
\[
\begin{array}{rrcl}
\lhd \maps H^* \times G &\to& H^*  \\
           (\chi, g)    &\mapsto& \chi_g
\end{array}
\]
where
\[\chi_g[h] = \chi[g\rhd h].  \]
We will demand that both these actions be measurable.  We do not know
if these are independent conditions.  However, in Lemma
\ref{lem:continuous_action} we show that if the action of $G$ on $H$
is continuous, its action on $H^*$ is continuous and thus measurable.
This handles most of the examples we care about.

With these preliminaries out of the way, here are the main definitions:

\begin{defn}
\label{measurable_2-group}
A {\bf measurable 2-group} $\G = (G,H,\rhd,\d)$ is a
2-group for which $G$, $H$ and $H^*$ are measurable groups
and the maps
\[
      \rhd\maps G\times H \to H,  \qquad
      \lhd\maps H^* \times G \to H^*, \qquad
      \d\maps H\to G
\]
are measurable.
\end{defn}

\begin{defn}
Let $\G = (G,H,\rhd,\d)$ be a measurable 2-group and suppose the
representation $\rho$ of $\G$ on $H^X$ is specified by the maps
\[       \lhd \maps X \times G \to X , \qquad
         \chi\maps X\to H^\ast \]
as in Thm.~\ref{thm:repclassify}.  Then $\rho$ is a {\bf measurable
representation} if both these maps are measurable.
\end{defn}

From now on, we will always be interested in {\em measurable}
representations of {\em measurable} 2-groups.
For such a representation, Lemma \ref{lem:good_topology} guarantees that we
can choose a topology for $X$, compatible with its structure as a
measurable space, such that the action of $G$ on $X$ is continuous.
This may not make $\chi \maps X \to H^\ast$ continuous.  However,
Lemma \ref{lem:automatic_continuity} implies that each $\chi(x)\maps H
\to \C^\times$ is continuous.

Before concluding this section, we point out a corollary of
Thm.~\ref{thm:repclassify} that reveals an interesting feature of the
representation theory in the 2-category $\me$. This corollary involves
a certain skeletal 2-group constructed from $\G$ (recall that a
2-group is `skeletal' when its corresponding crossed module has $\d =
0$).  Let $\G$ be a 2-group, {\it not necessarily measurable},
with corresponding crossed module
$(G, H, \d, \rhd)$.  Then, let
\[ \bar{G} = G/ \d(H) , \qquad
   \bar{H} = {H/[H,H]}
\]
Note that the image $\d(H)$ is a normal subgroup of $G$ by
(\ref{comp1}), and the commutator subgroup $[H,H]$ is
a normal subgroup of $H$.  One can check that the action $\rhd$
naturally induces an action $\bar{\rhd}$ of $\bar G$ on $\bar H$.
If we also define $\bar \d\maps \bar H \to
\bar G$ to be the trivial homomorphism, it is straightforward to check
that these data define a new crossed module, from which we get a new
2-group:

\begin{defn}
Let $\G$ be a 2-group with corresponding crossed module
$(G, H, \d, \rhd)$.  Then the 2-group $\bar{\G}$ constructed from the
crossed module $(\bar G,\bar H, \bar \d, \bar \rhd)$ is called the
{\bf skeletization} of $\G$.
\end{defn}

Now consider a representation $\rho$ of the 2-group $\G$.  First, by
Thm.~\ref{thm:repclassify}, $\d(H)$ acts trivially on $X$, so $\bar G$
acts on $X$.  Second, the group $\C^\times$ being abelian, $[H,H]$ is
contained in the kernel of the homomorphisms $\chi(x)\maps H \to
\C^\times$ for all $x$. In light of Thm.~\ref{thm:repclassify}, these
remarks lead to the following corollary:

\begin{cor}
For any 2-group, its representations of the form described in
Thm.~\ref{thm:repclassify} are in natural one-to-one correspondence
with representations of the same form of its skeletization.
\end{cor}

This corollary means measurable representations in $\me$ fail to
detect the `non-skeletal part' of a 2-group.  However, the
representation theory of $\G$ as a whole is generally richer than the
representation theory of its skeletization $\bar\G$.  One can indeed
show that, while $\G$ and $\bar \G$ can not be distinguished by
looking at their {\em representations}, they generally do not have the
same {\em intertwiners}.  In what follows, we will nevertheless
restrict our study to the case of skeletal 2-groups.

Thus, from now on, we suppose the group homomorphism $\d\maps H \to G$
to be trivial, and hence the group $H$ to be abelian.  Considering
Thm.~\ref{thm:repclassify} in light of the preceding discussion, we
easily obtain the following geometric characterization of measurable
representations of skeletal 2-groups.
\begin{theo}
\label{thm:skelrepclassify}
A measurable representation $\rho$ of a measurable skeletal 2-group
$\G = (G,H,\rhd)$ on $H^X$ is determined uniquely by a measurable
right $G$-action on $X$, together with a $G$-equivariant measurable
map $\chi \maps X \to H^\ast$.
\end{theo}
Since we consider only skeletal 2-groups and measurable
representations in the rest of the paper, this is the description of
2-group representations to keep in mind.  It is helpful to think of
this description as giving a `measurable $G$-equivariant bundle'
$$
\xymatrix{X  \ar[d]^{\chi} \\ H^*}
$$

\medskip

%
\subsubsection{Structure of intertwiners} \label{1inter}
%

In this section we study intertwiners between two fixed measurable representations $\rho_1$ and $\rho_2$ of a skeletal 2-group $\G$.
Suppose $\rho_1$ and $\rho_2$ are specified, respectively, by the
measurable $G$-equivariant bundles $\chi_1$ and $\chi_2$ , as in Thm.~\ref{thm:skelrepclassify}:
\[
{\xygraph{
  []!{0;<1.5cm,0cm>:<.75cm, 1.3cm>::}
  []{X}="TM"  [r] {Y}
      :@{->}^{\chi^{\phantom{Y}}_2} [d] {H^\ast}="M"
        "TM"        :@{->}_{\chi^{\phantom{X}}_1} "M"
  }}
\]

To state our main structure theorem for intertwiners, it is convenient to
first define two properties that a $Y$\!-indexed measurable of measures $\mu_y$
on $X$ might satisfy.  First, we say the family $\mu_y$ is \textbf{fiberwise} if
each $\mu_y$ is supported on the fiber, in $X$, over the point $\chi_2(y)$.  That
is, $\mu_y$ is fiberwise if
\[
      \mu_y(X) = \mu_y(\chi_1^{-1}(\chi_2(y)))
\]
for all $y$.  We also recall from Section \ref{main_results} that we say a measurable family of
measures is {\bf\boldmath equivariant} if for every $g\in G$ and $y\in Y$,
$\mu_{y\lhd g}$ is {\em equivalent} to the transformed measure $\mu_y^g$
defined by:
\beq
\mu_y^g(A) := \mu_y(A \lhd g^{-1}).
\eeq
Note that, to check that a given {\em equivariant} family of measures is fiberwise, it is enough to check that, for a set of representatives $y_o$ of the $G$-orbits in $Y$, the measure $\mu_{y_o}$ concentrates on the fiber over $\chi_2(y_o)$.

We are now ready to give a concrete characterization of intertwiners $\phi \maps \rho_1 \to \rho_2$ between measurable representations.  For notational simplicity we now omit the symbol `$\lhd$' for the right $G$-actions on $X$ and $Y$ defined by the representations, using simple concatenation instead.

%
%
\begin{theo}[Intertwiners] \label{thm:1intclassify}
Let $\rho_1, \rho_2$ be measurable representations of $\G=(G,H,\rhd)$, specified
respectively by the $G$-equivariant bundles $\chi_1\maps X \to H^\ast$
and $\chi_2\maps Y \to H^\ast$, as in Thm.~\ref{thm:skelrepclassify}. Given an intertwiner
$\phi\maps \rho_1\to \rho_2$, we can extract the following data:
\begin{romanlist}
\item \label{eq_and_fiber} an equivariant and fiberwise $Y$\!-indexed measurable family of measures $\mu_y$
on $X$;
\item a $\mu$-class of fields of Hilbert spaces $\phi_{y,x}$ on $Y\times X$;
\item for each $g\in G$, a $\mu$-class of fields of invertible linear maps
$\Phi^g_{y,x}\maps \phi_{y,x} \to \phi_{(y,x)g^{-1}}$
such that, for all $g,g'\in G$, the cocycle condition
$$
\Phi^{g'g}_{y,x} = \Phi^{g'}_{(y,x)g^{-1}} \Phi^g_{y,x}
$$
holds for all $y$ and $\mu_{y}$-almost every $x$.
\end{romanlist}
Conversely, such data can be used to construct an intertwiner.
\end{theo}
%
%
%
\medskip

Before commencing with the proof, note what this theorem does {\em not} state.  It does not state that the data extracted from an intertwiner are unique, nor that starting with these data and constructing an intertwiner gives `the same' intertwiner.  This does turn out to be essentially true, at least for an certain broad class of intertwiners.  The sense in which this result classifies intertwiners will be clarified in Propositions \ref{equ1int} and \ref{equ-trans1int}. 

\medskip

\begin{proof}
Recall that an intertwiner provides a morphism $\phi\maps H^X \to H^Y$ in $\me$, together with a family
$$\phi(g)\maps \rho_2(g)\phi \To \phi\rho_1(g) \quad g \in G$$
 of invertible 2-morphisms, subject to the compatibility conditions (\ref{norm1int}), (\ref{compatib}) and (\ref{pillow}), namely
\beq \label{norm1int'}
\phi(1) = \unit_\phi
\eeq
and
\beq \label{compatib'}
\left[\phi(g') \circ \unit_{\rho_1(g)} \right]
\cdot
\left[\unit_{\rho_2(g')} \circ \phi(g) \right]
=
\phi(g'g)
\eeq
and
\beq \label{pillow'}
\left[ \unit_\phi \circ \rho_1(u) \right] \cdot \phi(g) =
\phi(g) \cdot \left[\rho_2(u) \circ \unit_\phi \right]
\eeq
where $u = (g, h)$.

Let us show first that we may assume $\phi$ is a matrix functor.  Since $\phi$ is a measurable functor, we can pick a bounded natural isomorphism
\[
    m\maps \phi \To \tilde\phi
\]
where $\tilde\phi$ is a matrix functor.  We then define, for each $g\in G$, a measurable natural transformation
\[
    \tilde\phi(g) = \left[ m\circ \unit_{\rho_1(g)} \right] \cdot \phi(g)\cdot \left[\unit_{\rho_2(g)}\circ m^{-1}\right]
\]
chosen to make the following diagram commute:
\[
\xymatrix{
  H^X\ar[rrr]^{\rho_1(g)}
    \ar@{}[rrr]^(0.85){}="fx"\ar@/_3ex/[ddd]_{\phi}="a"\ar@/^3ex/[ddd]^{\tilde\phi}="b"
  &&&H^X\ar@{-->}@/_3ex/[ddd]_{\phi}="c"\ar@/^3ex/[ddd]^{\tilde\phi}="d"\\
  \\
  \\
  H^Y\ar[rrr]_{\rho_2(g)}\ar@{}[rrr]_(0.15){}="fy"&&&H^Y\\
  \ar@{=>}^{m} "a"+<2.5ex,0pt>;"b"+<-2.5ex,0pt>
  \ar@{:>}^{m} "c"+<2.5ex,0pt>;"d"+<-2.5ex,0pt>
  \ar@{} "fy";"c"|(0.3){}="f1"
  \ar@{} "fy";"c"|(0.7){}="f2"
  \ar@{} "b";"fx"|(0.3){}="b1"
  \ar@{} "b";"fx"|(0.7){}="b2"
  \ar@{=>} "f1";"f2"^{\tilde\phi(g)}
  \ar@{:>} "b1";"b2"_{\phi(g)}
}
\]
The matrix functor $\tilde\phi$, together with the family of measurable natural transformations $\tilde\phi(g)$, gives an intertwiner, which we also denote $\tilde \phi$.  The natural isomorphism $m$ gives an invertible 2-intertwiner $m\maps \phi\to \tilde\phi$.   So, every intertwiner is equivalent to one for which $\phi\maps H^X \to H^Y$ is a matrix functor.

Hence, we now assume $\phi = (\phi,\mu)$ is a matrix functor, and work out what equations (\ref{compatib'}) and (\ref{pillow'}) amount to in this case.  We use the following result, which simply collects in one place several useful composition formulas:
\begin{lemma}
\label{composition.lemma}
Let $\rho_1$ and $\rho_2$ be representations corresponding to $G$-equivariant bundles $X$ and $Y$ over $H^\ast$, as in the theorem.
\begin{enumerate}
\item Given any matrix functor $(T,t)\maps H^X \to H^Y$:
\begin{itemize}
\item  The composite $T\rho_1(g)$ is a matrix functor; in particular, it is defined by the field of Hilbert spaces $T_{y,xg^{-1}}$ and the family of measures $t^g_y$.
\item  The composite $\rho_2(g) T$ is a matrix functor; in particular, it is defined by the field of Hilbert spaces $T_{yg,x}$ and the family of measures $t_{yg}$.
\end{itemize}
\item Given a pair of such matrix functors $(T,t)$, $(T',t')$, and any matrix natural transformation $\alpha\maps T \To T'$:
\begin{itemize}
\item  Whiskering by $\rho_1(g)$ produces a matrix natural transformation whose field of linear operators is $\alpha_{yg,x}$.
\item  Whiskering by $\rho_2(g)$ produces a matrix natural transformation whose field of linear operators is $\alpha_{yg,x}$.
\end{itemize}
\end{enumerate}
That is:
\[
\xymatrix@C=5em{H^X \ar[r]^{\rho_1(g)} & H^X \ar@/^3ex/[r]^{T_{y,x},\,t_y}="f1" \ar@/_3ex/[r]_{T'_{y,x},\, t'_y}="f2" & H^Y
 \ar@{=>}^{\alpha_{y,x}} "f1"+<-2ex,-2.5ex>;"f2"+<-2ex,2.5ex>
}
\quad = \quad
\xymatrix@C=5em{H^X \ar@/^3ex/[r]^{T_{g,xg^{-1}},\, t_{y}^g}="f1" \ar@/_3ex/[r]_{T'_{y,xg^{-1}},\, {t'_y}^g}="f2" & H^Y
 \ar@{=>}^{\alpha_{y,xg^{-1}}} "f1"+<-2ex,-2.5ex>;"f2"+<-2ex,2.5ex>
}
\]
and
\[
\xymatrix@C=5em{H^X \ar@/^3ex/[r]^{T_{y,x},\,t_y}="f1" \ar@/_3ex/[r]_{T'_{y,x},\, t'_y}="f2" & H^Y \ar[r]^{\rho_2(g)} & H^Y
 \ar@{=>}^{\alpha_{y,x}} "f1"+<-2ex,-2.5ex>;"f2"+<-2ex,2.5ex>
}
\quad = \quad
\xymatrix@C=5em{H^X \ar@/^3ex/[r]^{T_{yg,x},\, t_{yg}}="f1" \ar@/_3ex/[r]_{T'_{yg,x},\, t'_{yg}}="f2" & H^Y
 \ar@{=>}^{\alpha_{yg,x}} "f1"+<-2ex,-2.5ex>;"f2"+<-2ex,2.5ex>
}
\]
\end{lemma}
\begin{proof.within.proof}
This is a direct computation from the definitions of composition for functors and natural transformations.
\end{proof.within.proof}

We return to the proof of the theorem.  Using this lemma, we immediately obtain explicit descriptions of the source and target of each $\phi(g)$: we find that composites $\rho_2(g) \phi$ and $\phi\rho_1(g)$ are the matrix functors whose families of measures are given by
\[
\mu_{yg}  \quad \mbox{and} \quad \mu^{g}_{y}
\]
respectively, and whose fields of Hilbert spaces read
\[
[\rho_2(g)\phi]_{y,x} = \phi_{yg, x} \quad \mbox{and} \quad [\phi\rho_1(g)]_{y,x} = \phi_{y, xg^{-1}}
\]
An immediate consequence is that the family $\mu_y$ is equivariant.
Indeed, since each $\phi(g)$ is a matrix natural isomorphism, Thm.~\ref{invert2}
implies the source and target measures $\mu_{yg}$ and $\mu^{g}_{y}$ are
equivalent for all $g$.  Thus, for all $g$, the 2-morphism $\phi(g)$ defines a field of invertible operators
\beq \label{source_target}
\phi(g)_{y,x} : \phi_{yg, x} \arr \phi_{y,xg^{-1}},
\eeq
determined for each $y$ and $\sqrt{\mu_{y}^g\mu_{yg}}$-\alme\ in $x$, or equivalently $\mu_{yg}$-\alme\ in $x$, by equivariance.

The lemma also helps make the compatibility condition (\ref{compatib'}) explicit.
The composites $\phi(g') \circ \unit_{\rho_1(g)}$ and $\unit_{\rho_2(g')} \circ \phi(g)$ are matrix natural transformations whose fields of operators read
\[
\left[\phi(g')\circ \unit_{\rho_1(g)}\right]_{y,x} = \phi(g')_{y, xg^{-1}}
\quad \mbox{and} \quad
\left[\unit_{\rho_2(g')} \circ \phi(g)\right]_{y,x} =\phi(g)_{yg', x}
\]
Hence, (\ref{compatib'}) can be rewritten as
\[
   \phi(g')_{y,xg^{-1}}\, \phi(g)_{yg',x} = \phi(g'g)_{y,x}.
\]
Defining a field of linear operators
\beq
\Phi^g_{y,x} \equiv \phi(g)_{yg^{-1}, x}\,,
\eeq
the condition (\ref{compatib'}) finally becomes:
\beq \label{rule1int}
\Phi^{g'g}_{y,x} = \Phi^{g'}_{(y,x)g^{-1}} \Phi^g_{y,x}.
\eeq
We note that since $\phi(g)_{y,x}$ is defined and invertible $\mu_{yg}$-\alme, $\Phi^g_{y,x}$ is defined and invertible $\mu_y$-\alme\

Finally, we must work out the consequences of the  ``pillow condition'' (\ref{pillow'}). We start by
evaluating the ``whiskered'' compositions $\rho_2(u) \circ \unit_\phi$ and $\unit_\phi \circ \rho_1(u)$, using the formula (\ref{hor2mor}) for horizontal composisiton.  By the lemma above, the composites $\rho_2(g) \phi$ and $\phi \rho_1(g)$ are matrix functors.  Hence  the 2-isomorphisms $\left[\rho_2(u) \circ \unit_\phi\right]$ and $\left[\unit_\phi\circ \rho_1(u)\right]$ are necessarily matrix natural transformations.   We can work out their matrix components using the definition of horizontal composition,
\[
\left[\rho_2(u) \circ \unit_\phi\right]_{y,x} = \chi_2(y)[h]
\]
and
\[
\left[\unit_\phi\circ \rho_1(u)\right]_{y,x} = \chi_1(xg^{-1})[h].
\]
The vertical compositions with $\phi(g)$ can then be performed with (\ref{vert2mor}); since all
the measures involved are equivalent to each other, these compositions reduce to pointwise compositions of operators---here multiplication of complex numbers. Thus the condition (\ref{pillow'}) yields the equation
\[
(\chi_2(y)[h] - \chi_1(xg^{-1})[h]) \,\phi(g)_{y,x} = 0
\]
which holds for all $h$, all $y$ and $\mu_{yg}$-almost every $x$. Thanks to the covariance of the fields of characters,
this equation can equivalently be written
as
\beq  \label{unpackpillow}
(\chi_2(y) - \chi_1(x)) \,\Phi^g_{y,x} = 0
\eeq
for all $y$ and $\mu_y$-almost every $x$.

This last equation actually expresses a condition for the
family of measures $\mu_y$. Indeed, it requires that, for every $y$, the subset of the $x \in X$ such that $\chi_1(x)\not=\chi_2(y)$ as well as
$\Phi^g_{y,x} \not= 0$ is a null set for the measure $\mu_y$. But we know that, for $\mu_y$-almost every $x$, $\Phi^g_{y,x}$ is an invertible operator
with a non-trivial source space $\phi_{y,x}$, so that it does not vanish. Therefore the condition expressed by (\ref{unpackpillow}) is that for each $y$, the measure $\mu_y$ is supported within the set $\{x \in X | \, \chi_1(x) = \chi_2(y)\} = \chi_1^{-1}(\chi_2(y))$.  So, the family $\mu_y$ is fiberwise.

Conversely, given an equivariant and fiberwise $Y$\!-indexed measurable family $\mu_y$ of measures on $X$, a measurable field of Hilbert spaces $\phi_{y,x}$ on $Y\times X$, and a measurable field  of invertible linear maps $\Phi^g_{y,x}$ satisfying the cocycle condition (\ref{rule1int}) $\mu$-\alme\ for each $g,g'\in G$, we can easily construct an intertwiner.  The  pair $(\mu_y, \phi_{y,x})$ gives a morphism $\phi\maps H^X\to H^Y$.  For each $g\in G$, $y\in Y$, and $x\in X$, we let
\[
\phi(g)_{y,x} = \Phi^g_{yg,x}\maps [\rho_2(g)\phi]_{y,x} \to [\phi\rho_1(g)]_{y,x}
\]
This gives a 2-morphism in $\me$ for each morphism in $\G$.  The cocycle condition, and the fiberwise property of $\mu_y$, ensure that the equations (\ref{compatib'}) and (\ref{pillow'}) hold.
\end{proof}

Given that the maps $\Phi^{g}_{y,x}$ are invertible, the cocycle condition (\ref{rule1int}) implies that $g=1$ gives the identity:
\beq
\label{preserves.1}
        \Phi^{1}_{y,x} = \unit_{\phi_{y,x}} \quad \text{$\mu$-\alme}
\eeq
In fact, given given the cocycle condition, this equation is clearly equivalent to the statement that the maps $\Phi^{g}_{y,x}$ are \alme-invertible.  We also easily get a useful formula for inverses:
\beq
\label{inverse.cocycle}
       \left(\Phi^{g}_{y,x}\right)^{-1} =  \Phi^{g^{-1}}_{(y,x)g^{-1}} \quad
       \text{$\mu$-\alme\ for each $g$.}
\eeq

Following our classification of representations, we noted that only some
of them deserve to be called `measurable representations' of a measurable
2-group.  Similarly, here we introduce a notion of `measurable intertwiner'.

First, in the theorem, there is no statement to the effect that the linear maps $\Phi^g_{y,x}\maps \phi_{y,x} \to \phi_{(y,x)g^{-1}}$ are `measurably indexed' by $g\in G$.
To correct this, for an intertwiner to be `measurable' we will demand that
$\Phi^g_{y,x}$ give a measurable field of linear operators on $G\times Y \times X$, where the field $\phi_{y,x}$ can be thought of as a measurable field of Hilbert spaces on $G\times Y \times X$ that is independent of its $g$-coordinate.

Second, in the theorem, for each pair of group elements $g,g'\in G$, we have a {\em separate} cocycle condition
\[
        \Phi^{g'g}_{y,x} = \Phi^{g'}_{(y,x)g^{-1}} \Phi^g_{y,x} \quad \text{$\mu$-\alme}
\]
In other words, for each choice of $g,g'$, there is a set $U_{g,g'}$ with $\mu_y(X-U_{g,g'})=0$ for all $y$, such that the cocycle condition holds on $U_{g,g'}$.
Unless the group $G$ is countable, the union of the sets $X-U_{g,g'}$ may have positive measure.  This seems to cause serious problems for characterization of such intertwiners, unless we impose further conditions. For an intertwiner to be `measurable', we will thus demand that the cocycle condition hold outside some null set, independently of $g,g'$.   Similarly, the theorem implies $\Phi^g_{y,x}$ is invertible $\mu$-\alme,  but {\em separately} for each $g$; for measurable intertwiners we demand invertibility outside a fixed null set, independently of $g$.

Let us now formalize these concepts:

\begin{defn}
Let $\Phi^g_{y,x}\maps\phi_{y,x} \to \phi_{(y,x)g^{-1}}$ be a measurable field of linear operators on $G\times Y \times X$, with $Y,X$ measurable $G$-spaces.  We say $\Phi$ is {\bf invertible} at $(y,x)$ if $\Phi^g_{y,x}$ is invertible for all $g\in G$; we say
$\Phi$ is {\bf cocyclic} at $(y,x)$ if $\Phi^{g'g}_{y,x} = \Phi^{g'}_{(y,x)g^{-1}} \Phi^g_{y,x}$ for all $g,g'\in G$.
\end{defn}

\begin{defn}
\label{mble.interwiner}
An intertwiner $(\phi, \Phi, \mu)$, of the form described in Thm.~\ref{thm:1intclassify}, is {\bf measurable} if:
\begin{itemize}
\item The fields $\Phi^g$ are obtained by restriction of a measurable field of linear operators $\Phi$ on $G \times Y\times X$;
\item $\Phi$ has a representative (from within its $\mu$-class) that is invertible and cocyclic at all points in some fixed subset $U\subseteq Y \times X$ with $\bar\mu_y(Y\times X)=\bar\mu_y(U)$ for all $y$.
\end{itemize}
More generally a {\bf measurable intertwiner} is an intertwiner that is isomorphic to one like this.
\end{defn}
The generalization in the last sentence of this definition is needed for two composable measurable intertwiners to have measurable composite.
From now on, we will always be interested in measurable intertwiners between measurable representations; we sometimes omit the word ``measurable'' for brevity, but it is always implicit.

The measurable field $\Phi^g_{y,x}$ in an intertwiner is very similar to a kind of cocycle used in the theory of induced representations on locally compact groups (see, for example, the discussion in Varadarajan's book \cite[Sec.~V.5]{Varadarajan}).  However, one major difference is that our cocycles here are much better behaved with respect to null sets.  In particular, we easily find that $\Phi$ is cocyclic and invertible at a point, it is satisfies the same properties everywhere on an $G$-orbit:
\begin{lemma}
\label{strict.on.orbits}
Let $\Phi^g_{y,x}\maps\phi_{y,x} \to \phi_{(y,x)g^{-1}}$ be a measurable field of linear operators on $G\times Y \times X$.  If $\Phi$ is invertible and cocyclic at $(y_o,x_o)$, then it is invertible and cocyclic at every point on the $G$-orbit of $(y_o,x_o)$.
\end{lemma}
\begin{proof}
If $\Phi$ is invertible and cocyclic at $(y_o,x_o)$, then for any $g,g'$ we have
\[
    \Phi^{g'}_{(y_o,x_o)g^{-1}} = \Phi^{g'g}_{y_o,x_o}\left(\Phi^{g}_{y_o,x_o}\right)^{-1}
\]
so $\Phi^{g'}$ at $(y_o,x_o)g^{-1}$ is the composite of two invertible maps, hence is invertible.  Since $g,g'$ were arbitrary, this shows $\Phi$ is invertible everywhere on the orbit.   Replacing $g'$ in the previous equation with a product $g''g'$, and using only the cocycle condition at $(y_o,x_o)$, we easily find that $\Phi$ is cocyclic at $(y_o,x_o)g^{-1}$ for arbitrary $g$.
\end{proof}

This lemma immediately implies, for any measurable intertwiner $(\phi,\Phi,\mu)$, that a representative of $\Phi$ may be chosen to be invertible and cocyclic not only on some set with null compliment, but actually everywhere on any orbit that meets this set.  This fact simplifies many calculations.

\begin{defn}
The measurable field of linear operators $\Phi^g_{y,x}\maps \phi_{y,x} \to \phi_{(y,x)g^{-1}}$ on $G\times Y\times X$ is called a {\bf\boldmath strict $G$-cocycle} if the equations
\[
\Phi^1_{y,x} = 1_{\phi_{y,x}} \quad \text{and} \quad
\Phi^{g'g}_{y,x} = \Phi^{g'}_{(y,x)g^{-1}} \Phi^g_{y,x}
\]
hold for {\em all} $(g,y,x)\in G\times Y\times X$.  An intertwiner $(\Phi,\phi,\mu)$ for which the measure-class of $\Phi^g_{y,x}$ has such a strict representative is a {\bf measurably strict} intertwiner.
\end{defn}

An interesting question is which measurable intertwiners are measurably strict.    This may be a difficult problem in general.  However, there is one case in which it is completely obvious from Lemma \ref{strict.on.orbits}: when the action of $G$ on $Y\times X$ is {\em transitive}.  In fact, it is enough for the $G$-action on
$Y\times X$ to be `essentially transitive', with respect to the family of
measures $\mu$.  We introduce a special case of intertwiners for which this is true:
\begin{defn}
\label{Transitive-Int}
A $Y$\!-indexed measurable family of measures $\mu_y$ on X is {\bf transitive} if there is a single $G$-orbit $o$ in $Y\times X$ such that, for every $y\in Y$, $\bar\mu_y = \delta_y \tensor \mu_y$ is supported on $o$.   A {\bf transitive intertwiner} is a measurable intertwiner $(\Phi, \phi, \mu)$ such that the family $\mu$ is transitive.
\end{defn}

It is often convenient to have a description of transitive families of measures using the measurable field $\mu_y$ of measures on $X$ directly, rather than the associated fibered measure distribution $\bar\mu_y$.  It is easy to check that $\mu_y$ is transitive if and only if there is a $G$-orbit $o\subseteq Y\times X$ such that whenever $(\{y\} \times A) \cap o =\emptyset$, we have $\mu_y(A)=0$.
Transitive intertwiners will play an important role in our study of intertwiners.
\begin{theo}
A transitive intertwiner is measurably strict.
\end{theo}
\begin{proof}
The orbit $o\subseteq Y\times X$, on which the measures $\bar\mu_y$ are supported, is a measurable set (see Lemma \ref{measurable_orbits}).  We may therefore take a representative of $\Phi^g_{y,x}\maps\phi_{y,x} \to \phi_{(y,x)g^{-1}}$ for which $\phi_{y,x}$ is trivial on the null set $(Y\times X)-o$.  The cocycle condition then automatically holds not only on $o$, by Lemma \ref{strict.on.orbits}, but also on its compliment.
\end{proof}

In fact, it is clear that a transitive intertwiner has an essentially {\em unique} field representative.  Indeed, any two representatives of $\Phi$ must be equal at almost every point on the supporting orbit, but then Lemma \ref{strict.on.orbits} implies they must be equal {\em everywhere} on the orbit.

Let us turn to the geometric description of intertwiners.  For simplicity, we restrict our attention to measurably strict intertwiners, for which the geometric correspondence is clearest.  Following Mackey, we can view the measurable field $\phi$ as a measurable bundle of Hilbert spaces over $Y\times X$:
$$
\xymatrix{\phi  \ar[d]^{} \\ Y\times X}
$$
whose fiber over $(y,x)$ is the Hilbert space $\phi_{y,x}$. As pointed out in Section \ref{main_results}, the strict cocycle $\Phi_{y,x}$ can be viewed as a left action of $G$ on the `total space' $\phi$ of this bundle since, by (\ref{rule1int}) and (\ref{preserves.1}), $\Phi^g\maps \phi \to \phi$ satisfies
\[
     \Phi^{g'g} = \Phi^{g'}\Phi^g
     \quad \text{and}\quad
     \Phi^1= 1_{\phi}.
\]
The corresponding right action $g\mapsto \Phi^{g^{-1}}$ of $G$ on $\phi$ is then an action of $G$ over the diagonal action on $Y\times X$:
\[
\xymatrix@R=3em@C=4em{
  \phi \ar[r]^{\Phi^{g^{-1}}}\ar[d]_{}& \phi \ar[d]^{}\\
  Y\times X \ar[r]^{\cdot \lhd g}& Y\times X
}
\]
So, loosely speaking, an intertwiner can be viewed as providing a `measurable $G$-equivariant bundle of Hilbert spaces' over $Y\times X$.   The associated equivariant family of measures $\mu$ serves to indicate, via $\mu$-\alme\ equivalence, when two such Hilbert space bundles actually describe the same intertwiner.

While these `Hilbert space bundles' are determined only up to measure-equivalence, in general, they do share many of the essential features of their counterparts in the topological category.  In particular, the `fiber' $\phi_{y,x}$ is a linear representation of the stabilizer group $S_{y,x}\subseteq G$, since the cocyle condition reduces to:
\[
\Phi^{s's}_{y,x} = \Phi^{s'}_{y,x} \Phi^s_{y,x} \maps \phi_{y,x}\to \phi_{y,x}
\]
for $s,s'\in S_{y,x}$.
\begin{defn}
Given any measurable intertwiner $\phi= (\phi,\Phi,\mu)$, we define the {\bf stabilizer representation} at $(y,x)\in Y\times X$ to be the linear representation of $S_{y,x} = \{s\in G : (y,x)s=(y,x)\}$ on $\phi_{y,x}$ defined by
\[
             \RR^\phi_{y,x}(s) = \Phi^s_{y,x}.
\]
These representations are defined $\mu_y$-\alme\ for each $y$.
\end{defn}

Along a given $G$-orbit $o$ in $Y\times X$, the stabilizer groups are all conjugate in $G$, so if we choose $(y_o,x_o)\in o$ with stabilizer $S_o=S_{y_o,x_o}$, then the stabilizer representations elsewhere on $o$ can be viewed as representations of $S_o$.
Explicitly,
\[
s \mapsto \RR^\phi_{y,x}({g^{-1} s g})
\]
defines a linear representation of $S_o$ on $\phi_{y,x}$, where $(y,x)=(y_o,x_o)g$.  Moreover, the cocycle condition implies
\[
\xymatrix@R=3em@C=8em{
  \phi_{y,x} \ar[r]^{\displaystyle\RR^\phi_{y,x}({g^{-1} s g})}\ar[d]_{\displaystyle\Phi^g_{y,x}}& \phi_{y,x} \ar[d]^{\displaystyle\Phi^g_{y,x}}\\
  \phi_{y_o,x_o} \ar[r]_{\displaystyle\RR^\phi_{y_o,x_o}(s)}& \phi_{y_o,x_o}
  }
\]
commutes for all $s\in S_o$, and all $g$ such that $(y,x)=(y_o,x_o)g$.  In other words, the maps $\Phi^g$ are intertwiners between stabilizer representations.
We thus see that the assignment $\phi_{y,x}, \Phi^g_{y,x}$ defines, for each orbit in $Y\times X$, a representation of the stabilizer group as well as a consistent way to `transport' it along the orbit with invertible intertwiners.

In the case of a transitive intertwiner, the only relevant Hilbert spaces are the ones over the special orbit $o$, so we may think of a transitive intertwiner as a Hilbert space bundle over a single orbit in $Y \times X$:
$$
\xymatrix{\phi  \ar[d]^{} \\ o}
$$
  We have also observed that the Hilbert spaces on the orbit $o$ are {\em uniquely} determined, so there is no need to mod out by $\mu$-equivalence.  We therefore obtain:
\begin{theo}
A transitive intertwiner is uniquely determined by:
\begin{itemize}
\item A transitive family of measures $\mu_y$ on $X$, with $\bar\mu_y$ supported on the $G$-orbit $o$,
\item A measurable field of linear operators $\Phi^g_{y,x}\maps\phi_{y,x} \to \phi_{(y,x)g^{-1}}$ on $G\times o$ that is cocyclic and invertible at some (and hence every) point.
\end{itemize}
\end{theo}

%
\subsubsection{Structure of 2-intertwiners}
%

We now turn to the problem of classifying the all 2-intertwiners between a fixed pair of parallel intertwiners $\phi, \psi\maps \rho_1\to\rho_2$.  If $\rho_1$ and $\rho_2$ are representations of the type described in Thm.~\ref{thm:repclassify}, then $\phi$ and $\psi$ are, up to equivalence, of the type described in Thm.~\ref{thm:1intclassify}.  Thus, we let $\phi$ and $\psi$ be given respectively by the equivariant and fiberwise families of measures $\mu_y$ and $\nu_y$, and the (classes of) fields of Hilbert spaces and invertible maps $\phi_{y,x}, \Phi^g_{y,x}$ and $\psi_{y,x}, \Psi^g_{y,x}$.    A characterization of 2-intertwiners between such intertwiners is given by the following theorem:
%
%
\begin{theo}[2-Intertwiners] \label{thm:2intclassify}
Let $\rho_1,\rho_2$ be representations on $H^X$ and $H^Y$, and let intertwiners $\phi, \psi\maps \rho_1 \to \rho_2$ be specified by the data $(\phi,\Phi,\mu)$ and $(\psi,\Phi,\nu)$ as in Thm.~\ref{thm:1intclassify}.  A 2-intertwiner $m\maps \phi \to \psi$ is specified uniquely by a $\sqrt{\mu\nu}$-class of fields of linear
maps $m_{y,x}\maps \phi_{y,x}\to \psi_{y,x}$ satisfying
\[
\Psi^g_{y,x} \,m_{y,x} = m_{(y,x)g^{-1}} \,\Phi^g_{y,x}
\]
$\sqrt{\mu\nu}$-\alme
\end{theo}
%
%
As usual, by $\sqrt{\mu\nu}$-class of fields we mean equivalence class of fields modulo identification of the fields which coincide for all $y$ and $\sqrt{\mu_y\nu_y}$-almost every $x$.

\medskip

\begin{proof}  By definition, a 2-intertwiner $m$ between the given intertwiners defines a
2-morphism in $\me$ between the morphisms $(\phi, \mu)$ and $(\psi, \nu)$, which satisfies the pillow condition (\ref{eq_pillow}), namely
\beq \label{eq_pillow'}
\psi(g) \cdot \left[\unit_{\rho_2(g)} \circ m \right] =
\left[m \circ \unit_{\rho_1(g)}\right] \cdot \phi(g)
\eeq
By Thm.~\ref{magic.thm}, since $m$ is a measurable natural transformation between matrix functors, it is automatically a {\em matrix} natural transformation.  We thus have merely to show that the conditions (\ref{eq_pillow'}) imposes on its matrix components $m_{y,x}$ are precisely those stated in the theorem.

First, using Lemma~\ref{composition.lemma}, the two whiskered composites $\unit_{\rho_2(g)} \circ m$ and $m \circ \unit_{\rho_1(g)}$ in (\ref{eq_pillow'}) are matrix natural transformations whose fields of
operators
read
\[
[\unit_{\rho_2(g)} \circ m]_{y,x} = m_{yg, x}
\quad \mbox{and}
\quad [m \circ
\unit_{\rho_1(g)}]_{y,x} = m_{y, xg^{-1}}
\]
respectively.
Next, we need to perform the vertical compositions on both sides
of the equality (\ref{eq_pillow'}). For this, we use the general formula (\ref{vert2mor}) for vertical composition of matrix natural transformations, which involves the square root of a product of three Radon-Nykodym derivatives. These derivatives are, in the present context:
\beq \label{products}
\rnd{\nu^g_y}{\mu_{yg}} \,\rnd{\nu_{yg}}{\nu_y^g} \,\rnd{\mu_{yg}}{\nu_{yg}}
\qquad \mbox{and} \qquad
\rnd{\nu^g_y}{\mu_{yg}}\, \rnd{\mu_y^g}{\nu_y^g} \,\rnd{\mu_{yg}}{\mu_y^g}
\eeq
for the left and right sides of (\ref{eq_pillow'}), respectively. Now the equivariance of the families $\mu_y,
\nu_y$ yields
\[
\rnd{\mu_{yg}}{\nu_{yg}} = \rnd{\mu_{yg}}{\nu^g_{y}}\rnd{\nu^g_{y}}{\nu_{yg}},
\qquad
\rnd{\mu_y^g}{\nu_y^g} = \rnd{\mu_{yg}}{\nu_y^g}\rnd{\mu_y^g}{\nu^g_y}
\]
so that both products in (\ref{products}) reduce to
\[
\rnd{\nu^g_y}{\mu_{yg}}\rnd{\mu_{yg}}{\nu_y^g}
\]
Thanks to the chain rule (\ref{chainrule}), namely
\[
\rnd{\mu}{\nu}\rnd{\nu}{\mu} = 1 \qquad \sqrt{\mu\nu}-\alme,
\]
this last term equals 1 almost everywhere for the geometric mean of the source and target measures for the 2-morphism described by either side of
(\ref{eq_pillow'}). This shows that the vertical composition reduces to the pointwise composition
of the fields of operators. Performing this composition and reindexing, (\ref{eq_pillow'}) takes the form
\beq \label{rule2int}
\Psi^g_{y,x} \,m_{y,x} = m_{(y,x)g^{-1}} \,\Phi^g_{y,x}
\eeq
as we wish to show.  This equation holds for all $y$ and $\sqrt{\mu_y\nu_y}$-almost every $x$.
\end{proof}

Thus, a 2-intertwiner $m \maps \phi \To \psi$ essentially assigns linear maps $m_{y,x}: \phi_{y,x} \to \psi_{y,x}$ to elements $(y,x)\in Y\times X$, in such a way that (\ref{rule2int}) is satisfied.  Diagrammatically, this equation can be written:
\[
\xymatrix@R=3em@C=4em{
  \phi_{y,x} \ar[r]^{\Phi^g_{y,x}}\ar[d]_{m_{y,x}}& \phi_{(y,x)g^{-1}} \ar[d]^{m_{(y,x)g^{-1}}}\\
  \psi_{y,x} \ar[r]_{\Psi^g_{y,x}}& \psi_{(y,x)g^{-1}}
}
\]
which commutes $\sqrt{\mu\nu}$-\alme\ for each $g$.  It is helpful to think of this as a generalization of the equation for an intertwiner between ordinary group representations.  Indeed, when restricted to elements of the stabilizer $S_{y,x}\subseteq G$ of $(y,x)$ under the diagonal action on $Y\times X$, it becomes:
\[
\RR^\psi_{{y,x}}(s) \,m_{{y,x}} = m_{{y,x}} \, \RR^\phi_{{y,x}}(s)  \qquad s\in S_{y,x}
\]
This states that $m_{y,x}$ is an intertwining operator, in the ordinary group-theoretic sense, between the stabilizer representations of $\phi$ and $\psi$.

If equation (\ref{rule2int}) is satisfied everywhere along some $G$-orbit $o$ in $Y\times X$, the maps $m_{y,x}$ of such an assignment are determined by the one $m_o \maps \phi_o \to \psi_o$ assigned to a fixed point $(y_o,x_o)$, since for $(y,x)=(y_o,x_o)g^{-1}$,
we have
\[
m_{y,x} = \Psi^g_o m_o \left(\Phi^g_o\right)^{-1}
\]
If the measure class of $m_{y,x}$ has a representative for which equation (\ref{rule2int}) is satisfied everywhere, $m_{y,x}$ is determined by its values at one representative of each $G$-orbit.

In the previous two sections, we introduced `measurable' versions of representations and intertwiners.  For 2-intertwiners, there are no new data indexed by morphisms or 2-morphisms in our 2-group.  Since a 2-group has a unique object, there are no new measurability conditions to impose.  We thus make the following simple definition.

\begin{defn}
A {\bf measurable 2-intertwiner} is a 2-intertwiner between measurable intertwiners, as classified in Thm.~\ref{thm:2intclassify}.
\end{defn}
\subsection{Equivalence of representations and of intertwiners}

In the previous sections we have characterized representations of a 2-group $\G$ on measurable categories, as well as intertwiners and 2-intertwiners.  In this section we would like to describe the \textit{equivalence
classes} of representations and intertwiners. The general notions of equivalence for representations and intertwiners was introduced, for a general target 2-category, in Section \ref{sec:2Rep}.  Recall from that section that two representations are equivalent when there is a (weakly) invertible intertwiner between them.  In the case of representations in $\me$, it is natural to specialize to `measurable equivalence' of representations:
\begin{defn}
Two measurable representations of a 2-group are {\bf measurably equivalent} if they are related by a pair of measurable intertwiners that are weak inverses of each other.
\end{defn}
In what follows, by `equivalence' of representations we always mean measurable equivalence.

Similarly, recall that two parallel intertwiners are equivalent when there is an invertible 2-intertwiner between them.  Since measurable 2-intertwiners are simply 2-intertwiners with measurable source and target, there are no extra conditions necessary for equivalent intertwiners to be `measurably' equivalent.

\medskip

Let $\rho_1$ and $\rho_2$ be measurable representations of $\G=(G,H,\rhd)$ on the measurable categories $H^X$ and $H^Y$ defined by $G$-equivariant bundles $\chi_1\maps X \to H^\ast$ and $\chi_2\maps Y \to H^\ast$.  We use the same symbol ``$\lhd$'' for the action of $G$ on both $X$ and $Y$. The following theorem explains the geometric meaning of equivalence of representations.
%
%
\begin{theo}[Equivalent representations]
\label{equirep}
Two measurable representations $\rho_1$ and $\rho_2$ are equivalent if and only if the corresponding $G$-equivariant bundles $\chi_1\maps X\to H^\ast$ and $\chi_2\maps Y \to H^\ast$ are isomorphic.  That is, $\rho_1\sim\rho_2$ if and only if there is an invertible measurable function $f\maps Y \to
X$ that is $G$-equivariant:
\[f(y\lhd g) = f(y) \lhd g\]
and fiber-preserving:
\[ \chi_1(f(y)) = \chi_2(y).\]
\end{theo}
%
%
\begin{proof}
Suppose first the representations are equivalent, and let $\phi$ be an invertible intertwiner between them. Recall that each intertwiner defines a morphism in $\me$; moreover, as shown by the law (\ref{compo_1int}), the morphism defined by the composition of two intertwiners in the 2-category of representations $\Rep(\G)$ coincides with the composition of the two morphisms in $\me$. As a consequence, the invertibility of $\phi$ yields the invertibility of its associated morphism $(\phi, \mu)$. By Theorem
\ref{invert}, this means the measures $\mu_y$ are equivalent to Dirac measures $\delta_{f(y)}$ for some invertible (measurable) function $f\maps Y \to X$.

On the other hand, by definition of an intertwiner, the family $\mu_y$ is equivariant. This means here that the measure $\delta_{f(y\lhd g)}$ is equivalent to the measure $\delta^g_{f(y)} =
\delta_{f(y)\lhd g}$.  Thus, the two Dirac measures charge the same point, so $f(y\lhd g) = f(y) \lhd g$.
We also know that the support of $\mu_y$, that is, the singlet $\{f(y)\}$, is included in the set $\{x\in X |\, \chi_1(x) = \chi_2(y)\}$. This yields $\chi_1(f(y)) = \chi_2(y)$.

Conversely, suppose there is a function $f$ which satisfies the conditions of the theorem.
One can immediately construct from it an invertible intertwiner between the two representations, by considering the family of measures  $\delta_{f(y)}$, the constant field of one-dimensional spaces $\C$ and the constant field of identity maps $\unit$.
\end{proof}

\medskip

We now consider two intertwiners $\phi$ and $\psi$ between the same pair of representations $\rho_1$ and $\rho_2$, specified by equivariant and fiberwise families of measures $\mu_y$ and $\nu_y$, and classes of fields $\phi_{y,x}, \Phi^g_{y,x}$ and $\psi_{y,x}, \Psi^g_{y,x}$. As we know, these carry standard linear representations $\mathcal{R}^\phi_{y,x}$ and $\mathcal{R}^\psi_{y,x}$ of the stabilizer $S_{y,x}$ of $(y,x)$ under the diagonal action of $G$, respectively in the Hilbert spaces $\phi_{y,x}$ and $\psi_{y,x}$.

The following proposition gives necessary conditions for intertwiners to be equivalent:
%
%
\begin{prop} \label{equ1int}
If the intertwiners $\phi$ and $\psi$ are equivalent, then for all $y\in Y$, $\mu_y$ and $\nu_y$  are in the
same measure class and  the stabilizer representations $\mathcal{R}^\phi_{y,x}$ and $\mathcal{R}^\psi_{y,x}$ are equivalent for $\mu_y$-almost every $x \in X$.
\end{prop}
%
%
%
\begin{proof}
Assume $\phi \sim \psi$, and let $m \maps \phi \To \psi$ be an invertible 2-intertwiner. Recall that any 2-intertwiner defines a 2-morphism in $\me$; moreover, the morphism defined by the composition of two 2-intertwiners in the 2-category of representations $\Rep(\G)$ coincides with the composition of the two 2-morphisms in $\me$. As a consequence, the invertibility of $m$ yields the invertibility of its associated 2-morphism.
By Thm.~\ref{invert2}, this means  that the measures of the source and the target of $m$ are equivalent. Thus, for all $y$,  $\mu_y$ and $\nu_y$  are in the same measure class.

We know that $m$ defines a $\mu$-class of fields of linear maps $m_{y,x}\maps \phi_{y,x} \to \psi_{y,x}$, such that for all $y$ and $\mu_y$-almost every $x$, $m_{y,x}$ intertwines the stabilizer representations $\mathcal{R}^\phi_{y,x}$ and $\mathcal{R}^\psi_{y,x}$. Moreover, since $m$ is invertible as a 2-morphism in $\me$, we know by Thm.~\ref{invert2} that the maps  $m_{y,x}$ are invertible.  Thus, for all $y$ and almost every $x$, the two group representations $\mathcal{R}^\phi_{y,x}$ and $\mathcal{R}^\psi_{y,x}$ are  equivalent.
\end{proof}

This proposition admits a partial converse, if one restricts
to transitive intertwiners:

\begin{prop}[Equivalent transitive intertwiners]  \label{equ-trans1int}
Suppose the intertwiners $\phi$ and $\psi$ are transitive. If for all $y$, $\mu_y$ and $\nu_y$ are in the same measure class and  the stabilizer representations $\mathcal{R}^\phi_{y,x}$ and $\mathcal{R}^\psi_{y,x}$ are equivalent for $\mu_y$-almost every $x \in X$, then $\phi$ and $\psi$ are equivalent.
\end{prop}
\begin{proof}
Let $o$ be an orbit of $Y\times X$ such that $\mu_y(A)=0$ for each $\{y\} \times A $ in  $(Y \times X) - o$. First of all, if the family $\mu_y$ is trivial, so is $\nu_y$; and in that case the intertwiners are obviously equivalent. Otherwise, there is a point  $u_o=(y_o, x_o)$ in $o$ at which the representations $\mathcal{R}^\phi_o$ and $\mathcal{R}^\psi_o$ of the stabilizer $S_o$ are equivalent.
Now, assume the two intertwiners are specified by the assignments of  Hilbert spaces $\phi_u, \psi_u$  and  invertible maps $\Phi^g_u \maps \phi_u \to \phi_{ug^{-1}}$ and $\Psi^g_u \maps \psi_u \to \psi_{ug^{-1}}$ to the points of the orbit, satisfying  cocycle conditions. These yield, for  $u=u_ok^{-1}$,
\beq \label{partcase:cocycle}
\Phi^g_{u} = \Phi^{gk}_{o} \left(\Phi^k_o\right)^{-1}, \qquad
\Psi^g_{u} = \Psi^{gk}_{o} \left(\Psi^k_o\right)^{-1}
\eeq
where $\phi_o, \Phi^g_o$ denote the value of the fields at the point $u_o$.  Now,  let $m_o \maps \phi_o \to \psi_o$ be an invertible intertwiner between the representations $\mathcal{R}^\phi_o$ and $\mathcal{R}^\psi_o$. Then for $u=u_ok^{-1}$, the formula
\[
m_u = \Psi^k_o m_o \left(\Phi^k_o\right)^{-1}
\]
defines invertible maps $m_u \maps \phi_u \to  \psi_u$. It is then straightforward to show that (\ref{partcase:cocycle}) yields the intertwining equation  $\Psi^g_{u} \,m_{u} = m_{ug^{-1}} \,\Phi^g_{u}$.
Thus, the maps $m_u$ define a 2-intertwiner $m \maps \phi \To \psi$. We furthermore deduce from the  Thm.~\ref{invert2} that $m$ is invertible. Thus, the intertwiners $\phi$ and $\psi$ are equivalent.
\end{proof}

In fact, any transitive intertwiner is equivalent to one for which the field of Hilbert spaces $\phi_{y,x}$ is {\em constant}, $\phi_{y,x}\equiv \phi_o$. More generally, this is true, for any intertwiner, on any single $G$-orbit  $o\subseteq Y\times X$ on which the cocycle is strict.  To see this, pick $u_o= (y_o,x_o)$ in $o$ and let $S_o = S_{y_o,x_o}$ be its stabilizer.  Since $o\cong G/S_o$ is a homogeneous space of $G$, there is a measurable section (see Lemma \ref{meas.sections})
\[
     \sigma\maps o \to G
\]
defined by the properties
\[
      \sigma(y_o,x_o)=1\in G \quad \text{and}\quad
          (y_o,x_o)\sigma(y,x) = (y,x)
\]
If we define $\phi_o = \phi_{y_o,x_o}$, then for each $(y,x)\in o$, we get a specific isomorphism of $\phi_{y,x}$ with $\phi_o$:
\[
      \alpha_{y,x} = \Phi^{\sigma(y,x)}_{y,x}\maps \phi_{y,x} \to  \phi_o
\]
If we then define
\[
     P^{g}_{y,x} = \alpha_{(y,x)g^{-1}} \Phi^g_{y,x} \left(\alpha_{y,x} \right)^{-1}
\]
a straightforward calculation shows that $P$ is cocyclic:
\begin{align*}
    P^{g'g}_{y,x} &= \alpha_{(y,x)(g'g)^{-1}} \Phi^{g'g}_{y,x} \left(\alpha_{y,x} \right)^{-1} \\
              &=\alpha_{(y,x)g^{-1}g'^{-1}} \Phi^{g'}_{(y,x)g^{-1}} \left(\alpha_{(y,x)g^{-1}} \right)^{-1}
                  \alpha_{(y,x)g^-1} \Phi^g_{y,x} \left(\alpha_{y,x} \right)^{-1} \\
        &= P^{g'}_{(y,x)g^{-1}} P^g_{y,x}
\end{align*}
We thus get a new measurable intertwiner $(\phi_o, P, \mu)$, which is equivalent to the original intertwiner $(\phi,\Phi,\mu)$ via an invertible 2-intertwiner defined by $\alpha_{y,x}$.

In geometric language, this shows that any `measurable $G$-equivariant bundle' can be trivialized by via a `measurable bundle isomorphism', while maintaining $G$-equivariance.  So there are no global `twists' in such `bundles', as there are in topological or smooth categories.

%
\subsection{Operations on representations}
%
\label{operations}

Some of the most interesting features of ordinary group representation theory arise because there are natural notions of `direct sum' and `tensor product', which we can use to build new representations from old.  The same is true of 2-group representation theory.  In the group case, these sums and products of representations are built from the corresponding operations in $\Vect$.  Likewise, for sums and products in our representation theory, we first need to develop such notions in the 2-category $\me$.

Thus, in this section, we first consider direct sums and tensor products of measurable categories and measurable functors.  We then use these to describe direct sums and tensor products of measurable representations, and measurable intertwiners.

\subsubsection{Direct sums and tensor products in $\me$}
\label{sec:operations}

We now introduce important operations on `higher vector spaces', analogous to
taking `tensor products' and `direct sums' of ordinary vector spaces.
These operations are well understood in the case of $\twoVe$ \cite{BarrettMackaay,KV}; here we discuss their generalization to $\me$.

We begin with `direct sums'.  As emphasized by Barrett and Mackaay
\cite{BarrettMackaay} in the
case of $\twoVe$, there are several levels of `linear structure' in a
2-category of higher vector spaces.  In ordinary linear algebra, the
{\em set} $\Vect(V,V')$ of all linear maps between fixed vector spaces
$V,V'$ is itself a {\em vector space}.  But the {\em category} $\Vect$ has a
similar structure: we can take
direct sums of both vector spaces and linear maps, making $\Vect$
into a (symmetric) {\em monoidal category}.

In {\em categorified} linear algebra, this `microcosm' of linearity goes one
layer deeper.   Here we can add {\em 2-maps} between fixed maps,
so the top-dimensional hom sets form vector spaces.
But there are now two distinct ways
of taking `direct sums' of {\em maps}.  Namely, since
we can think of a map between 2-vector spaces as a `matrix of
vector spaces', we can either take the `matrix of direct
sums', when the matrices have the same size, or, more generally,
we can take the `direct sum of matrices'.   These ideas lead
to two distinct operations which we call the `direct sum'
and the `2-sum'.  The direct sum leads to the idea that the hom {\em categories}, consisting of all maps between fixed 2-vector spaces, as well as 2-maps between those, should be {\em monoidal} categories; the second leads to the idea that a 2-category of 2-vector spaces should itself be a `monoidal 2-category'.

Let us make these ideas more precise, in the case of $\me$.
The most obvious level of linear structure in $\me$ applies only at
2-morphism level.  Since sums and constant multiples of bounded natural transformations are bounded, the set of measurable natural transformations between fixed measurable functors is a complex vector space.

Next, fixing two measurable spaces $X$ and $Y$, let $\Mat(X,Y)$ be the
category with:
\begin{itemize}
  \item matrix functors $(T,t)\maps H^X \to H^Y$ as objects
  \item matrix natural transformations as morphisms
\end{itemize}
$\Mat(X,Y)$ is clearly a linear category, since composition is bilinear with respect to the vector space structure on each hom set.

Next, there is a notion of \textit{direct sum} in $\Mat(X, Y)$, which
corresponds to the intuitive idea of a `matrix of direct
sums'.   Intuitively, given two matrix functors $(T,t),(T',t')\in \Mat(X,Y)$, we
would like to form a new matrix functor with matrix components
$T_{y, x} \oplus T'_{y,x}$.  This makes sense as long as the families of
measures $t_y$ and $t'_y$ are equivalent, but in general we must
be a bit more careful.  We first define a $y$-indexed measurable
family of measures $t\oplus t'$ on $X$ by
\beq
\label{directsum-meas}
       (t\oplus t')_y = t_y + t'_y
\eeq
This will be the family of measures for a matrix functor we will call the
direct sum of $T$ and $T'$.
To obtain the corresponding field of Hilbert spaces, we use the
Lebesgue decompositions of the measures with respect to each other:
\[
t = t^{t'}+ \overline{t^{t'}}, \qquad t' = t'^t + \overline{t'^t}
\]
with $t^{t'} \ll t'$ and $\overline{t^{t'}} \perp t'$, and similarly
$t'^{t} \ll t$ and $\overline{t'^{t}} \perp t$. The subscript
$y$ indexing the measures has been dropped for simplicity.
The measures $t^{t'}$ and ${t'}^t$ are equivalent,
and these are singular with respect to both $\overline{t^{t'}}$
and $\overline{{t'}^t}$; moreover, these latter two measures are
mutually singular.
For each $y\in Y$, we can thus write $X$ as a disjoint union
\[
 X = A_y \amalg B_y \amalg C_y
\]
with $\overline{t^{t'}}$ supported on $A_y$, $\overline{t'^{t}}$ supported on $B_y$, and $t^{t'}, t'^t$ supported on $C_y$.  (In particular, $t_y$ is supported on $A_y \amalg C_y$ and $t'_y$ is supported on $B_y \amalg C_y$.) We then define a new $(t\oplus t')$-class of fields of Hilbert spaces $T\oplus T'$ by setting
\beq
\label{directsum-mor}
[T\oplus T']_{y,x} =
\left\{\begin{array}{cl}
T_{y,x} & x \in A_y \\
T'_{y,x} &  x \in B_y \\
T_{y, x} \oplus T'_{y,x} &  x \in C_y
\end{array}\right.
\eeq
The $(t\oplus t')$-class does not depend on the choice of sets $A_y$,
$B_y$, $C_y$, so the data $(T\oplus T',t\oplus t')$ give a well defined
matrix functor $H^X \to H^Y$,
an object of $\Mat(X,Y)$.   We call this the {\bf direct sum} of $(T,t)$
and $(T',t')$, and denote it by $(T,t)\oplus (T',t')$, or simply $T\oplus T'$ for short.
Note that this direct sum is boundedly naturally
isomorphic to the functor mapping $\H\in H^X$ to
the $H^Y$-object with components $(T\H)_y \oplus (T'\H)_y$.

There is an obvious unit object $0\in \Mat(X,Y)$ for the tensor product, defined by the trivial $Y$-indexed family of measures on $X$, $\mu_y \equiv 0$.  In fact, this is a {\em strict} unit object, meaning that we have the equations:
$$(T,t)\oplus 0 = (T,t) = 0\oplus  (T,t)$$
for any object $(T,t) \in \Mat(X,Y)$.
We might expect these to hold only up to isomorphism, but since $t + 0 = t$, and $T$ is defined up to measure-class, the equations hold strictly.

Also, given any pair of 2-morphisms in $\Mat(X,Y)$, say matrix natural transformations $\alpha$ and $\alpha'$:
\[
  \xymatrix{
  H^X\ar@/^2ex/[rr]^{T, t}="g1"\ar@/_2ex/[rr]_{U, u}="g2"&& H^Y
  \ar@{=>}^{\alpha} "g1"+<0ex,-2.5ex>;"g2"+<0ex,2.5ex>
}
\qquad
\text{ and }
\qquad
  \xymatrix{
  H^X\ar@/^2ex/[rr]^{T', t'}="g1"\ar@/_2ex/[rr]_{U', u'}="g2"&& H^Y
  \ar@{=>}^{\alpha'} "g1"+<0ex,-2.5ex>;"g2"+<0ex,2.5ex>
}
\]
we can construct their direct sum, a matrix natural transformation
\[
  \xymatrix@C=4em{
  H^X\ar@/^2ex/[rr]^{T\oplus T', t\oplus t'}="g1"\ar@/_2ex/[rr]_{U\oplus U', u\oplus u'}="g2"&& H^Y
  \ar@{=>}^{\alpha\oplus \alpha'} "g1"+<-3ex,-2.5ex>;"g2"+<-3ex,2.5ex>
}
\]
 as follows.  Again, dealing with measure-classes is the tricky part.
This time, let us decompose $X$ in two ways, for each $y$:
\[
 X = A_y \amalg B_y \amalg C_y = A'_y \amalg B'_y \amalg C'_y
\]
with $t_y$ supported on $A_y \amalg C_y$, $u_y$ on $B_y \amalg C_y$,
and $t_y$ and $u_y$ equivalent on $C_y$, and similarly,
$t'_y$ supported on $A'_y \amalg C'_y$, $u'_y$ on $B'_y \amalg C'_y$,
and $t'_y$ and $u'_y$ equivalent on $C'_y$.  We then
define
\beq
\label{directsum-2mor}
[\alpha \oplus \alpha']_{y,x} =
\left\{\begin{array}{cl}
\alpha_{y,x}\oplus\alpha'_{y,x} & x \in C_y\cap C'_y \\
\alpha_{y,x} &  x \in C_y - C'_y \\
\alpha'_{y,x} &  x \in C'_y - C_y \\
0 & \text{otherwise}
\end{array}\right.
\eeq
For this to determine a matrix natural transformation between the indicated matrix functors, we must show that our formula determines the field of linear operators for each $y$ and $\mu_y$-almost every $x$, where
\[
\mu_y = \sqrt{(t_y + t'_y)(u_y + u'_y)}
\]
On the set $C_y\cap C'_y$, the measures $t_y, u_y, t'_y, u'_y$ are all equivalent, hence are also equivalent to $\mu$, so $\alpha$ is clearly determined on this set.  On the set $C_y - C'_y$, we have $t_y\sim u_y$, while $t'_y \bot u'_y$.  Using these facts, we show that
\[
     \mu_y \sim  \sqrt{(t_y + t'_y)(t_y + u'_y)}
     \sim t_y + \sqrt{t'_y u'_y}  = t_y \sim \sqrt{t_y u_y}
     \qquad
     \text{on $C_y - C'_y$}.
\]
But the matrix components of $\alpha\oplus \alpha$ given in (\ref{directsum-2mor}) are determined precisely $\sqrt{t_y u_y}$-\alme, hence $\mu_y$-\alme\ on $C_y - C'_y$.
By an identical argument with primed and un-primed symbols reversing roles, we find
\[
     \mu_y \sim \sqrt{t'_y u'_y}
          \qquad
     \text{on $C'_y - C_y$}.
\]
So the components of $\alpha\oplus \alpha'$ are determined $\mu_y$-\alme\ for each $y$, hence give a matrix natural transformation.

We have defined the `direct sum' in $\Mat(X, Y)$ as a binary operation on objects (matrix functors) and a binary operation on morphisms (matrix natural transformations).  One can check that the direct sum is functorial, i.e. it respects composition and identities:
\[
    (\beta \cdot \alpha) \oplus (\beta' \cdot \alpha') =
    (\beta \oplus \beta') \cdot (\alpha \oplus \alpha')
\]
and
\[
    \unit_{T} \oplus \unit_{T'} = \unit_{T\oplus T'}.
\]
%
\begin{defn} The {\bf direct sum} in $\Mat(X, Y)$ is the functor:
$$\oplus\maps \Mat(X, Y) \times \Mat(X,Y)\to\Mat(X,Y).$$
defined by
\begin{itemize}
\item
The direct sum of objects $T, T' \in \Mat(X, Y)$ is the object $T \oplus
T'$
specified by the family of measures $t\oplus t'$ given in
(\ref{directsum-meas}), and by the $t\oplus t'$-class of fields $[T \oplus
T']_{y,x}$ given in (\ref{directsum-mor});
\item
The direct sum morphisms $\alpha \maps T \to U$ and $\alpha' \maps T'
\to U'$ is the morphism $\alpha \oplus \alpha' \maps T\oplus T' \to
U\oplus U'$ specified by the $\sqrt{(t+u)(t'+u')}$-class of fields of linear maps given in
(\ref{directsum-2mor}).
\end{itemize}
\end{defn}
%

The direct sum can be used to promote $\Mat(X,Y)$ to a monoidal category.
There is an obvious `associator' natural transformation; namely, given objects $T,T',T''\in \Mat(X,Y)$,
we get a morphism
\[
 A_{T,T',T''} \maps (T\oplus T')\oplus T''\to T\oplus (T'\oplus T'')
\]
obtained by using the usual associator for direct sums of Hilbert spaces, on the common support of the respective measures $t$, $t'$, and $t''$.  The left and right `unit laws', as mentioned already, are identity morphisms.  A straightforward exercise shows that that $\Mat(X,Y)$ becomes a monoidal category under direct sum.

There is also an obvious `symmetry' natural transformation in $\Mat(X,Y)$,
\[
         S_{T,T'} \maps T\oplus T' \to T'\oplus T
\]
making $\Mat(X,Y)$ into a {\em symmetric} monoidal category.

We can go one step further.  Given {\em any} measurable categories $\HH$ and $\HH'$, the `hom-category' $\me(\HH,\HH')$ has
\begin{itemize}
  \item measurable functors $T\maps \HH \to \HH$ as objects
  \item measurable natural transformations as morphisms
\end{itemize}
An important corollary of Thm.~\ref{magic.thm} is that this category is equivalent to some $\Mat(X,Y)$.  Picking an adjoint pair of equivalences:
\[
\xymatrix{ \me(\HH,\HH') \ar @< 2pt> [r]^{F} &
                  \Mat(X,Y) \ar @< 2pt> [l]^{\overline F}  }
\]
we can transport the (symmetric) monoidal structure on $\Mat(X,Y)$ to one on $\me(\HH,\HH')$ by a standard procedure.  For example, we define a tensor product of $T,T'\in \me(\HH,\HH')$ by
\[
  T \oplus T' = \overline F ( F(T)\oplus F(T') ).
\]
This provides a way to take direct sums of arbitrary parallel measurable functors, and arbitrary measurable natural transformations between them.

\medskip

We now explain the notion of `{2-sum}', which is a kind of sum that applies not only to measurable functors and natural transformations, like the direct sum defined above, but also to measurable categories themselves.

First, we define to 2-sum of measurable categories of the form $H^X$ by the formula
\[
  H^X \boxplus H^{X'} = H^{X \amalg X'}
\]
where $\amalg$ denotes disjoint union.  Thus, an object of $H^X \boxplus H^{X'}$ consists of a measurable field of Hilbert spaces on $X$, and one on $X'$.

Next, for arbitrary matrix functors $(T, t) \maps H^X \to H^Y$ and $(T', t') \maps H^{X'} \to H^{Y'}$, we will define a matrix functor $(T\boxplus T',t \boxplus t')$ called the 2-sum of $T$ and $T'$.  Intuitively, whereas the `direct sum' was like a `matrix of direct sums', the `2-sum' should be like a `direct sum of matrices'.  Thus, we use the
fields of Hilbert spaces $T$ on $Y\times X$ and $T'$ on $Y'\times X'$
to define a field $T\boxplus T'$ on $Y\amalg Y' \times X \amalg
X'$, given by
\beq \label{2sum-mor}
[T\boxplus T']_{y,x} =
\left\{\begin{array}{cl}
T_{y,x} & (y,x) \in Y \times X \\
T'_{y,x} & (y,x) \in Y' \times X' \\
0 & \text{otherwise}
\end{array}\right.
\eeq
This is well defined on measure-equivalence classes, almost everywhere with respect to the $Y\amalg Y'$-indexed family $t \boxplus t'$ of measures on $X \amalg X'$, defined by:
\beq \label{2sum-meas}
[t \boxplus t']_{y} =
\left\{\begin{array}{cc} t_{y} & y \in Y \\ t'_{y} & y \in Y' \end{array}\right.
\eeq
In this definition we have identified $t_y$ with its obvious extension
to a measure on $X \amalg X'$.

Finally, suppose we have two arbitrary matrix natural transformations,
defined by the
fields of linear maps $\alpha_{y,x} \maps T_{y,x} \to
U_{y,x}$ and $\alpha'_{y', x'} \maps T'_{y',x'} \to U'_{y',x'}$.
From these, we construct a new field of maps from $[T\boxplus T']_{y,x}$ to
$[U\boxplus U']_{y,x}$, given by
\beq \label{2sum-2mor}
[\alpha\boxplus \alpha']_{y,x} =
\left\{\begin{array}{cc} \alpha_{y,x} & (y,x) \in Y \times X \\ \alpha'_{y,x} & (y,x) \in Y' \times X'  \\
0 & \text{otherwise}
\end{array}\right.
\eeq
This is determined $\sqrt{(t \boxplus t')(u\boxplus u')}$-\alme, and hence
defines a matrix natural transformation $\alpha\boxplus \alpha' \maps T\boxplus T' \To U \boxplus U'$.

\begin{defn} The term {\bf 2-sum} refers to any of the following binary operations, defined on certain objects, morphisms, and 2-morphisms in $\me$:

\begin{itemize}
\item
The {\bf 2-sum} of measurable categories $H^X$ and $H^{X'}$ is the measurable category $H^X \boxplus H^{X'} = H^{X \amalg X'}$;

\item
The {\bf 2-sum} of matrix functors $(T, t) \maps H^X \to H^Y$ and $(T', t') \maps H^{X'} \to H^{Y'}$ is the matrix functor $(T \boxplus T', t \boxplus t') \maps H^{X \amalg X'} \to H^{Y \amalg Y'}$ specified by the family of measures $t \boxplus t'$ given in (\ref{2sum-meas}) and the class of fields $T \boxplus T'$ given in (\ref{2sum-mor});

\item
The {\bf 2-sum} of matrix natural transformations $\alpha \maps (T, t) \To (U , u)$ and $\alpha' \maps (T', t') \To (U , u')$ is the matrix natural transformation $\alpha \boxplus \beta \maps (T \boxplus  T', t \boxplus t') \To (U \boxplus U', u\boxplus u')$ specified by the class of fields of linear operators given in (\ref{2sum-2mor}).
\end{itemize}
\end{defn}
%

It should be possible to extend the notion of 2-sum to apply to arbitrary objects, morphisms, or 2-morphisms in $\me$, and define additional structure so that $\me$ becomes a `monoidal 2-category'.
While we believe our limited definition of `2-sum' is a good starting point for a more thorough treatment, we make no such attempts here.   For our immediate purposes, it suffices to know how to take 2-sums of objects, morphisms, and 2-morphisms of the special types described.

There is an important relationship
between the direct sum $\oplus$ and the 2-sum $\boxplus$.
Given arbitrary---not necessarily parallel---matrix functors $(T, t) \maps H^X \to H^Y$ and $(T', t')
\maps H^{X'} \to H^{Y'}$, their 2-sum can be written as a direct sum:
\beq \label{2-sum-from-direct-sum}
     T\boxplus T' \cong [T\boxplus 0'] \oplus [0\boxplus T']
\eeq
Here $0$ and $0'$  denote the unit objects in the monoidal categories $\Mat(X,Y)$ and $\Mat(X',Y')$.  A similar relation holds for matrix natural transformations.

\medskip

We now briefly discuss `tensor products'.  As with the additive structures discussed above, there may be multiple layers of related multiplicative structures.  In particular, we can presumably use the ordinary tensor product of Hilbert spaces and linear maps to turn each $\Mat(X,Y)$, and ultimately each $\me(\HH,\HH')$, into a (symmetric) monoidal category.  But, we should also be able to turn $\me$ itself into a monoidal 2-category, using a `tensor 2-product' analogous to the `direct 2-sum'.

We shall not develop these ideas in detail here, but it is perhaps worthwhile outlining the general structure we expect.  First, the tensor product in $\Mat(X,Y)$ should be given as follows:
\begin{itemize}
  \item Given objects $(T,t)$, $(T',t')$, define their tensor product $(T\tensor T', t\tensor t')$ by the family of measures \[
  (t\tensor t')_y = \sqrt{t_yt'_y}
\]
and the field of Hilbert spaces
\[
[T\otimes T']_{y,x} =
T_{y,x}\tensor T_{y,x}
\]
  \item Given morphisms $\alpha \maps T \to U$ and $\alpha' \maps T'
\to U'$, define their tensor product $\alpha \tensor \alpha' \maps T\tensor T' \to
U\tensor U'$ by the class of fields defined by
\[
   (\alpha\tensor \alpha')_{y,x} = \alpha_{y,x} \tensor \alpha'_{y,x}
\]
\end{itemize}
These are simpler than the corresponding formulae for the direct sum, as null sets turn out to be easier to handle.  As with the direct sum, we expect the tensor product to give $\Mat(X,Y)$ the structure of a symmetric monoidal category, allowing us to transport this structure to any hom-category $\me(\HH,\HH')$ in $\me$.

Next, let us describe the `tensor 2-product'.
\begin{itemize}
\item
Given two measurable categories of the form $H^X$ and $H^{X'}$, we define their tensor 2-product to be
\[
      H^X\boxtimes H^{X'} := H^{X\times X'}
\]

\item Given matrix functors $(T, t) \maps H^X \to H^Y$ and $(T', t') \maps H^{X'} \to H^{Y'}$, define their tensor 2-product to be the matrix functor $(T\boxtimes T', t\boxtimes t')$
defined by the $Y \times Y'$-indexed family of measures on $X \times X'$
\[
[t\boxtimes t']_{y,y'} = t_y \otimes t'_{y'},
\]
where $\tensor$ on the right denotes the ordinary tensor product of measures,
and the field of Hilbert spaces
\[
[T \boxtimes T']_{(y, y'), (x, x')} = T_{y, x} \otimes T_{y', x'}.
\]

\item Given matrix natural transformations $\alpha \maps (T, t) \To (U , u)$ and $\alpha' \maps (T', t') \To (U , u')$, define their tensor 2-product to be the matrix natural transformation $\alpha \boxtimes \beta \maps (T \boxtimes  T', t \boxtimes t') \To (U \boxtimes U', u\boxtimes u')$ specified by:
\[
[\alpha \boxtimes \alpha']_{(y, y'), (x, x')}  = \alpha_{y, x}  \otimes \alpha_{y', x'}
\]
determined almost everywhere with respect to the family of geometric mean measures:
$$
\sqrt{(t\boxtimes t')( u\boxtimes u')} = \sqrt{tu} \boxtimes \sqrt{t'u'}
$$
\end{itemize}
As with the 2-sum, it should be possible to use this tensor 2-product to make $\me$ into a monoidal 2-category.   We leave this to further work.

\subsubsection{Direct sums and tensor products in $\Rep(\G)$}
\label{sec:operations-in-2Rep}

Now let $\G$ be a skeletal measurable 2-group, and consider the representation 2-category $\Rep(\G)$ of (measurable) representations of $\G$ in $\me$.  Monoidal structures in $\me$ give rise to monoidal structures in this representation category in a natural way.

Let us consider the various notions of  `sum' that $\Rep(\G)$ inherits from $\me$.  First, and most obvious, since the 2-morphisms in $\me$ between a fixed pair of morphisms form a vector space, so do the 2-intertwiners between fixed intertwiners.

Next, fix two representations $\rho_1$ and $\rho_2$, on the measurable categories $H^X$ and $H^Y$, respectively.  An intertwiner $\phi\maps \rho_1 \to \rho_2$ gives an object of $\phi\in \me(H^X,H^Y)$ and for each $g\in G$ a morphism in $\me(H^X,H^Y)$.   Since $\me(H^X,H^Y)$ is equivalent to $\Mat(X,Y)$, the former becomes a symmetric monoidal category with direct sum, and this in turn induces a direct sum of intertwiners between $\rho_1$ and $\rho_2$.    We get a direct sum of 2-intertwiners in an analogous way.
%
\begin{defn}
Let $\rho_1$, $\rho_2$ be representations on $H^X$ and $H^Y$.  The {\bf direct sum of intertwiners} $\phi,\phi'\maps \rho_1 \to \rho_2$ is the intertwiner $\phi\oplus \phi' \maps \rho_1 \to \rho_2$  given by the morphism $\phi \oplus  \phi'$ in $\me$, together with the 2-morphisms $\phi(g) \oplus \phi'(g)$ in $\me$.  The {\bf direct sum of 2-intertwiners} $m\maps \phi \to \psi$ and $m'\maps \phi' \to \psi'$ is the 2-intertwiner given by the measurable natural transformation $m\oplus m' \maps \phi \oplus\phi' \to \psi \oplus\psi'$.
\end{defn}
%
The intertwiners define families of measures $\mu_y$ and $\mu'_y$, and classes of fields of Hilbert spaces $\phi_{y,x}$ and $\phi'_{y, x}$ and invertible maps $\Phi^g_{y,x}$ and $\Phi'^g_{y, x}$ that are invertible and cocyclic.  It is straightforward to deduce the structure of the direct sum of intertwiners in terms of these data:
%
\begin{prop} \label{directsumInt}
Let $\phi= (\phi,\Phi,\mu)$, $\phi'= (\phi',\Phi',\mu')$ be measurable intertwiners with the same source and target representations.  Then the intertwiner $\phi\oplus \phi'$ specified by the family of measures $\mu + \mu'$, and the classes of fields $\phi_{y,x} \oplus \phi'_{y, x}$ and $\Phi^g_{y,x} \oplus \Phi'^g_{y, x}$, is a direct sum for $\phi$ and $\phi'$.
\end{prop}
%
%

The intertwiner specified by the family of trivial measures, $\mu_y \equiv 0$, plays the role of unit for the direct sum. This unit is the {\bf null intertwiner} between $\rho_1$ and $\rho_2$.

Finally, $\Rep(\G)$ inherits a notion of `2-sum'.  We begin with the representations.
%
%
%
\begin{defn}
The {\bf 2-sum of representations} $\rho \boxplus \rho'$ is the representation defined by
\[
(\rho \boxplus\rho')(\varsigma) = \rho(\varsigma)\boxplus\rho'(\varsigma)
\]
where $\varsigma$ denotes the object $\star$, or any morphism or 2-morphism in $\G$.
\end{defn}
%
%
%
We immediately deduce, from the definition of the 2-sum in $\me$, the structure of the 2-sum of representations:
%
\begin{prop} Let $\rho$, $\rho'$ be measurable representations of $\G = (G,H,\rhd)$, with corresponding equivariant maps $\chi\maps X \to H^\ast$, $\chi'\maps X \to H^\ast$.  The {\bf 2-sum of representations} $\rho \boxplus \rho'$ is the representation on the measurable category $H^{X\amalg X'}$, specified by  the action of $G$ induced by the actions on $X$ and $X'$, and the obvious equivariant map $\chi\amalg\chi' \maps X \amalg X' \to H^\ast$.
\end{prop}
%
The empty space $X = \emptyset$ defines a representation\footnote{Note that the measurable category $H^{\emptyset}$ is the category with just one object and one morphism.} which plays the role of unit element for the
direct sum. This unit element is the \textbf{null representation}.

There is a notion of 2-sum for intertwiners, which allows one to define the sum of intertwiners that are not necessarily parallel.  This notion can essentially be deduced from that of the direct sum, using (\ref{2-sum-from-direct-sum}).   Indeed, if $\phi=(\phi,\Phi,\mu)$ is a measurable intertwiner, a 2-sum of the form $\phi \boxplus 0$ is simply given by the trivial extensions of the fields $\phi,\Phi,\mu$ to a disjoint union, and likewise for $0\boxplus \phi'$; we then simply write $\phi \boxplus \phi'$ as a direct sum via (\ref{2-sum-from-direct-sum}) and the analogous equation for 2-morphisms.

\medskip

There should also be notions of `tensor product' and `tensor 2-product'  in the representation 2-category $\Rep(\G)$.  Since we have not constructed these products in detail in $\me$, we shall not give the details here; the constructions should be analogous to the `direct sum' and `2-sum' just described.

%
\subsection{Reduction, retraction, and decomposition}
\label{reduction}
%

In this section, we introduce notions of reducibility and
decomposability, in analogy with group representation theory, as well
as an {\it a priori} intermediate notion, `retractability'.  These
notions make sense not only for representations, but also for
intertwiners. We classify the indecomposable, irretractable and
irreducible measurable representations, and intertwiners between
these, up to equivalence.

\subsubsection{Representations}
\label{irreps}

Let us start with the basic definitions.

\begin{defn}
A representation $\rho'$ is a {\bf subrepresentation} of a given 
representation $\rho$ if there exists a weakly monic intertwiner 
$\rho'\to\rho$.
\end{defn}
We remind the reader that an intertwiner $\phi\maps \rho' \to \rho$ is
(strictly) {\bf monic} if whenever $\xi,\xi'\maps \tau \to \rho$ are
intertwiners such that $\phi\cdot \xi = \phi \cdot \xi'$, we have
$\xi=\xi'$; we say it is {\bf weakly monic} if this holds up to
invertible 2-intertwiners, i.e.\ $\phi\cdot \xi \cong \phi \cdot \xi'$
implies $\xi\cong\xi'$.
\begin{defn}
A representation $\rho'$ is a {\bf retract} of $\rho$ if there exist
intertwiners $\phi \maps \rho' \to \rho $ and $\psi \maps \rho \to
\rho'$ whose composite $\psi \phi$ is equivalent to the identity
intertwiner of $\rho'$
\[
\xymatrix{\rho' \ar[r]^{\phi} & \rho \ar[r]^{\psi} & \rho'}
\quad \simeq \quad
\xymatrix{\rho' \ar[r]^{\unit_{\rho'}} & \rho'}
\]
\end{defn}

\begin{defn}
A representation $\rho'$ is a {\bf 2-summand} of $\rho$ if $\rho
\simeq \rho' \boxplus \rho''$ for some representation $\rho''$.
\end{defn}

It is straightforward to show that any 2-summand is automatically a
retract, since the diagram
\[
        \rho' \to \rho' \boxplus \rho'' \to \rho',
\]
built from the obvious `injection' and `projection' intertwiners, is
equivalent to the identity.  On the other hand, we shall see that a
representation $\rho$ generally has retracts that are not 2-summands;
this is in stark contrast to linear representations of ordinary
groups, where summands and retracts coincide.

Similarly, any retract is automatically a subrepresentation, since
$\psi\phi\simeq 1$ easily implies $\phi$ is weakly monic.

Any representation $\rho$ has both itself and the null representation
as subrepresentations, as retracts, and as summands.  This leads us
to the following definitions:

\begin{defn}
A representation $\rho$ is {\bf irreducible} if it has exactly two
subrepresentations, up to equivalence, namely $\rho$ itself and the
null representation.
\end{defn}

\begin{defn}
A representation $\rho$ is {\bf irretractable} if it has exactly two
retracts, up to equivalence, namely $\rho$ itself and the null
representation.
\end{defn}

\begin{defn}
A representation $\rho$ is {\bf indecomposable} if it has exactly two
2-summands, up to equivalence, namely $\rho$ itself and the null
representation.
\end{defn}
Note that according to these definitions, the null representation is
neither irreducible, nor indecomposable, nor irretractable.  An
irreducible representation is automatically irretractable, and an
irretractable representation is automatically indecomposable.  A
priori, neither of these implications is reversible.

Indecomposable representations are characterized by the following theorem:
%
%
\begin{theo}[Indecomposable representations] \label{theo_indecomposable}
Let $\rho$ be a measurable representation on $H^X$, making $X$ into a
measurable $G$-space. Then $\rho$ is indecomposable if and only if $X$
is nonempty and $G$ acts transitively on $X$.
\end{theo}
%
%
\begin{proof}
Observe first that, since the null representation is not
indecomposable, the theorem is obvious for the case $X =
\emptyset$. We may thus assume $\rho$ is not the null representation.

Assume first $\rho$ indecomposable, and let $U$ and $V$ be two
disjoint $G$-invariant subsets such that $X = U \amalg V$. $\rho$
naturally induces representations $\rho_U$ in $H^U$ and $\rho_V$ in
$H^V$, and furthermore $\rho = \rho_U \boxplus \rho_V$. Since by
hypothesis $\rho$ is indecomposable, at least one of these
representations is the null representation. Consequently $U =
\emptyset$ or $V = \emptyset$. This shows that the $G$-action is
transitive.

Conversely, assume $G$ acts transitively on $X$, and suppose $\rho
\sim \rho_1 \boxplus \rho_2$ for some representations $\rho_i$ in
$H^{X_i}$. There is then a splitting $X = X'_1 \amalg X'_2$, where
$X'_i$ is measurably identified with $X_i$ and $G$-invariant. Since by
hypothesis $G$ acts transitively on $X$, we deduce that $X'_i =
\emptyset = X_i$ for at least one $i$. Thus, $\rho_i$ is the null
representation for at least one $i$; hence $\rho$ is indecomposable.
\end{proof}

Let $o$ be any  $G$-orbit in $H^\ast$; pick a point $x_o^\ast$, and let $S_o^\ast$ denote its stabilizer group.
The orbit can be identified with the homogeneous space $G /S_o^\ast$. Let also $S \subset S^\ast_o$ be any closed subgroup of $S$. Then $X:=G/S$ is a measurable $G$-space (see Lemma \ref{lem:quotient_space} in the Appendix). The canonical projection onto $G /S_o^\ast$ defines a $G$-equivariant map $\chi \maps X \to H^\ast$. This map is measurable: to see this, write $\chi = \pi s$, where $s$ is a {\bf measurable section} of $G/S$ as in Lemma \ref{meas.sections}, and $\pi \maps G \to G/S^\ast_o$ is the measurable projection. Hence, the pair $(o, S)$ defines a measurable representation; this representations is clearly indecomposable.

Next, consider the representations given by two pairs $(o, S)$ and $(o', S')$. When are they equivalent? Equivalence means that there is an isomorphism $f \maps G/S \to G/S'$ of measurable $G$-equivariant bundles over $H^\ast$. Such isomorphism exists if and only if the orbits are the same $o=o'$ and the subgroups $S, S'$ are conjugate in $S_o^\ast$. Hence, there is class of inequivalent indecomposable representations labelled by an orbit $o$ in $H^\ast$ and a conjugacy class of subgroups $S \subset S^\ast_o$.

Now, let $\rho$ be any indecomposable representation on
$H^X$. Thm.~\ref{theo_indecomposable} says $X$ is a transitive
measurable $G$-space. Transitivity forces the $G$-equivariant map
$\chi \maps X \to H^\ast$ to map onto a single orbit $o \simeq
G/S^\ast_o$ in $H^\ast$. Moreover, it implies that $X$ is isomorphic
as a $G$-equivariant bundle to $G/S$ for some closed subgroup $S
\subset S^\ast_o$. Hence, $\rho$ is equivalent to the representation
defined by the orbit $o$ and the subgroup $S$.

These remarks yield the following:
%
\begin{cor}
Indecomposable representations are classified, up to equivalence, by a
choice of $G$-orbit $o$ in the character group $H^\ast$, along with a
conjugacy class of closed subgroups $S \subset S^\ast_o$ of the
stabilizer of one of its points.
\end{cor}
%

Irretractable representations are characterized by the following theorem:
%
\begin{theo}[Irretractable representations] \label{theo_irretractable}
Let $\rho$ be a measurable representation, given by a measurable
$G$-equivariant map $\chi \maps X \to H^\ast$, as in
Thm.~\ref{thm:skelrepclassify}.  Then $\rho$ is irretractable if and
only if $\chi$ induces a $G$-space isomorphism between $X$ and a
single $G$-orbit in $H^\ast$.
\end{theo}
%
%
\begin{proof}
First observe that, since a $G$-orbit in $H^\ast$ is always nonempty, and the null representation is not irretractable, the theorem is obvious for the case $X = \emptyset$.  We may thus assume $\rho$ is not the null representation.

Now suppose $\rho$ is irretractable, and consider a single $G$-orbit
$X^\ast$ contained in the image $\chi(X) \subset H^\ast$. $X^\ast$ is
a measurable subset (see Lemma.~\ref{measurable_orbits} in the
Appendix), so it naturally becomes a measurable $G$-space, with
$G$-action induced by the action on $H^\ast$. The canonical injection
$X^\ast \to H^\ast$ makes $X^\ast$ a measurable equivariant bundle
over the character group. These data give a non-null representation
$\rho^\ast$ of the 2-group on the measurable category $H^{X^\ast}$.

We want to show that $\rho^\ast$ is a retract of $\rho$. To do so, we
first construct an $X$-indexed family of measures $\mu_x$ on $X^\ast$
as follows: if $\chi(x) \in X^\ast$, we choose $\mu_x$ to be the Dirac
measure $\delta_{\chi(x)}$ which charges the point $\chi(x)$;
otherwise we choose $\mu_x$ to be the trivial measure. This family is
fiberwise by construction; the covariance of the field of characters
ensures that it is also equivariant:
\[
    \delta^g_{\chi(x)} = \delta_{\chi(x)g} = \delta_{\chi(xg)}.
\]
To check that the family is measurable, pick a measurable subset
$A^\ast \subset X^\ast$. The function $x \mapsto \mu_x(A^\ast)$
coincides with the characteristic function of the set $A
=\chi^{-1}(A^\ast)$, whose value at $x$ is $1$ if $x \in A$ and $0$
otherwise; this function is measurable if the set $A$ is.  Now, since
we are working with measurable representations, the map $\chi$ is
measurable: therefore $A$ is measurable, as the pre-image of the
measurable $A^\ast$. Thus, the family of measures $\mu_x$ is
measurable. So, together with the $\mu$-classes of one-dimensional
fields of Hilbert spaces and identity linear maps, it defines an
intertwiner $\phi \maps \rho^\ast \to \rho$.

Next, we want to construct a $X^\ast$-indexed equivariant and
fiberwise family of measures $\nu_{x^\ast}$ on $X$.  To do so, pick an
element $x_o^\ast \in X^\ast$, denote by $S_o^\ast \subset G$ its
stabilizer group.  We require some results from topology and measure
theory (see Appendix \ref{apx:G-spaces}).  First, $S^\ast_o$ is a
closed subgroup, and the orbit $X^\ast$ can be measurably identified
with the homogenous space $G/S^\ast_o$; second, there exists a
measurable section for $G/S^\ast_o$, namely a measurable map $n \maps
G / S^\ast_o \to G$ such that $\pi n =\mbox{Id}$, where $\pi \maps G
\to G/S^\ast_o$ is the canonical projection, and $n\pi(e)=e$. Also,
the action of $G$ on $X$ induces a measurable $S^\ast_o$-action on the
fiber over $x_o^\ast$; any orbit of this fiber can thus be measurably
identified with a homogeneous space $S^\ast_o/S$, on which nonzero
quasi-invariant measures are known to exist.

So let $\nu_{x^{\ast}_o}$ be (the extension to $X$ of) a
$S^\ast_o$-quasi-invariant measure on the fiber over $x^\ast_o$.
Using a measurable section $n \maps G / S^\ast_o \to G$, each $x^\ast
\in X^\ast$ can then be written unambiguously as $x^\ast_o n(k)$ for
some coset $k \in G/S^\ast_o$. Define
\[
\nu_{x^\ast} := \nu^{n(k)}_{x_o^\ast}
\]
where by definition $\nu^g(A) = \nu(Ag^{-1})$.  We obtain by this
procedure a measurable fiberwise and equivariant family of measures on
$X$.  Together with the ($\nu$-classes of) constant one-dimensional
field(s) of Hilbert spaces $\C$ and constant field of identity linear
maps, this defines an intertwiner $\psi \maps \rho \to \rho^\ast$.

We can immediately check that the composition $\psi\phi$ of these two
intertwiners defined above is equivalent to the identity intertwiner
$\unit_{\rho^{\ast}}$, since the composite measure at $x^\ast$,
\[
\int_X \extd \nu_{x^\ast}(x) \mu_{x} = \nu_{x^\ast}(\chi^{-1}(x^\ast))
\, \delta_{x^\ast}
\]
is equivalent to the delta function $\delta_{x^\ast}$. This shows that
$\rho^\ast$ is a retract of $\rho$.

Now, by hypothesis $\rho$ is irretractable; since the retract
$\rho^\ast$ is not null, it must therefore be equivalent to $\rho$. We
know by Thm.~\ref{equirep} that this equivalence gives a measurable
isomorphism $f: X^\ast \to X$, as $G$-equivariant bundles over
$H^\ast$.  In our case, $f$ being a bundle map means
\[
\chi(f(x^\ast)) = x^\ast.
\]
Together with the invertibility of $f$, this relation shows that the
image of the map $\chi$ is $X^\ast$, and furthermore that $\chi =
f^{-1}$. We have thus proved that $\chi\maps X \to X^\ast$ is an
invertible map of $X$ onto the orbit $X^\ast\subseteq H^\ast$.

Conversely, suppose $\chi$ is invertible and maps $X$ to a single
orbit $X^\ast$ in $H^\ast$ and consider a non-null retract $\rho'$ of
$\rho$.  We denote by $X'$ the underlying space and by $\chi'$ the
field of characters associated to $\rho'$. Pick two intertwiners $\phi
\maps \rho' \to \rho $ and $\psi \maps \rho \to \rho'$ such that $\psi
\phi \simeq \unit_{\rho'}$. These two intertwiners provide an
$X$-indexed family of measures $\mu_{x}$ on $X'$ and a $X'$-indexed
family of measures $\nu_{x'}$ on $X$ which satisfy the property that,
for each $x'$, the composite measure at $x'$ is equivalent to a Dirac
measure: 
\beq \label{pty} \int_X \extd \nu_{x'}(x) \mu_{x} \, \sim \, \delta_{x'} 
\eeq

An obvious consequence of this property is that the measures
$\nu_{x'}$ are all non-trivial.  Since $\nu_{x'}$ concentrates on the
fiber over $\chi'(x')$ in $X$, this fiber is therefore not empty. This
shows that $\chi'(X')$ is included in the $G$-orbit $\chi(X) =
X^\ast$. The $G$-invariance of the subset $\im\chi'$ shows furthermore
that this inclusion is an equality, so $\chi(X) =
\chi'(X')$. Consequently the map $f = \chi^{-1} \chi'$ is a well
defined measurable function from $X'$ to $X$; it is surjective,
commutes with the action of $g$ and obviously satisfies $\chi f =
\chi'$. Now, by hypothesis, the fiber over $\chi'(x')$ in $X$, on
which $\nu_{x'}$ concentrates, consists of the singlet $\{f(x')\}$: we
deduce that $\nu_{x'} \sim \delta_{f(x')}$.  The property (\ref{pty})
thus reduces to $\mu_{f(x')} \sim \delta_{x'}$ for all $x'$, which
requires $f$ to be injective. Thus, we have found an invertible
measurable map $f\maps X' \to X$ that is $G$-equivariant and preserves
fibers of $\chi\maps X \to H^\ast$.  By Thm.~\ref{equirep}, the
representations $\rho$ and $\rho'$ are equivalent; hence $\rho$ is
irretractable.
\end{proof}

Any irretractable representation is indecomposable; up to equivalence,
it thus takes the form $(o, S)$, where $o$ is a $G$-orbit in $H^\ast$
and $S$ is a subgroup of $S^\ast_o$. However, the converse is not
true: there are in general many indecomposable representations $(o,
S)$ that are retractable. Indeed, $(o, S)$ defines an invertible map
$\chi \maps G/S \to G/S^\ast_o$ only when $S=S^\ast_o$. The existence
of retractable but indecomposable representations has been already noted
by Barrett and Mackaaay \cite{BarrettMackaay} in the context of the
representation theory of 2-groups on finite dimensional 2-vector
spaces.  We see here that this is also true for representations on more
general measurable categories.

%
%
\begin{cor}
Irretractable measurable representations are classified, up to equivalence, by $G$-orbits in the character group $H^\ast$.
\end{cor}
%
%
%

%
%
%

\subsubsection{Intertwiners}
\label{irrint}

Because 2-group representation theory involves not only intertwiners between representations, but also 2-intertwiners between intertwiners, there are obvious analogs for intertwiners of the concepts
discussed in the previous section for representations.  We define sub-intertwiners, retracts and 2-summands of intertwiners in a precisely analogous way, obtaining notions of irreducibility, irretractability, and indecomposability for intertwiners, as for representations.
\begin{defn}
An intertwiner $\phi'\maps \rho_1 \to \rho_2$ is a {\bf sub-intertwiner} of $\phi\maps \rho_1 \to \rho_2$ if there exists a monic 2-intertwiner $m\maps \phi' \To \phi$.
\end{defn}
We remind the reader that a 2-intertwiner $m\maps \phi' \To \phi$ is {\bf monic} if whenever $n,n'\maps \psi \To \phi'$ are 2-intertwiners such that $m\cdot n = m \cdot n'$, we have $n=n'$.
%
%
\begin{defn}
An intertwiner $\phi'\maps \rho_1 \to \rho_2$ is a {\bf retract} of $\phi\maps \rho_1 \to \rho_2$ if there  exist 2-intertwiners $m\maps \phi' \To \phi$ and $n\maps \phi \To \phi'$ such that the vertical product $n \cdot m$ equals the identity 2-intertwiner of $\phi'$
\[
\xymatrix{
   \rho_1\ar@/^4ex/[rr]^{\phi'}="g1"\ar[rr]^(0.35){\phi}\ar@{}[rr]|{}="g2"
  \ar@/_4ex/[rr]_{\phi'}="g3"&&\rho_2
  \ar@{=>}^{m} "g1"+<0ex,-2ex>;"g2"+<0ex,1ex>
  \ar@{=>}^{n} "g2"+<0ex,-1ex>;"g3"+<0ex,2ex>
}
\quad = \quad
\xymatrix{
  \rho_1 \ar@/^2ex/[rr]^{\phi'}="g1"\ar@/_2ex/[rr]_{\phi'}="g2"&& \rho_2
  \ar@{=>}^{\unit_{\phi'}} "g1"+<0ex,-2.5ex>;"g2"+<0ex,2.5ex>
}
\]
\end{defn}
%
%
\begin{defn}
An intertwiner $\phi'\maps \rho_1 \to \rho_2$ is a {\bf summand} of $\phi\maps \rho_1 \to \rho_2$ if $\phi \cong \phi' \boxplus \phi''$ for some intertwiner $\phi''$.
\end{defn}

Any summand is a retract, and any retract is a sub-intertwiner.  Recall from Section \ref{sec:operations} that the {\bf null intertwiner} between measurable representations on $H^X$ and $H^Y$ is defined by the trivial family of measures, $\mu_y = 0$ for all $y$.  It is easy to see that the null intertwiner is a summand (hence also a retract, and a sub-intertwiner) of any intertwiner.

\begin{defn}
An intertwiner $\phi$ is {\bf irreducible} if it has exactly two sub-intertwiners, up to 2-isomorphism, namely $\phi$ itself and the null intertwiner.
\end{defn}
%
%
\begin{defn}
An intertwiner $\phi$ is {\bf irretractable} if it has exactly two retracts, up to 2-isomorphism, namely $\phi$ itself and the null intertwiner.  
\end{defn}
%
%
\begin{defn}
An intertwiner $\phi$ is {\bf indecomposable} if it has exactly two summands, up to 2-isomorphism, namely $\phi$ itself and the null intertwiner.
\end{defn}

According to these definitions, the null representation is neither irreducible, nor indecomposable, nor irretractable.
An irreducible intertwiner is automatically irretractable, and an irretractable intertwiner is automatically indecomposable.  A priori, neither of these implications is reversible.

\medskip

To dig deeper into these notions, we need some concepts from ergodic theory: ergodic measures, and their generalization to measurable families of measures.  In what follows, we denote by $\triangle$ the symmetric difference operation on sets:
\[
       U\,\triangle\, V = (U \cup V) - (U \cap V)
\]
When $U$ is a subset of a $G$-set $X$, we use the notation $Ug = \{ug\, | \, u\in U\}$.
%
\begin{defn}
\label{ergodic.measure}
 A measure $\mu$ on $X$ is {\bf ergodic} under a $G$-action if for any measurable subset $U \subset X$ such
that $\mu(U \,\triangle\, Ug) = 0$ for all $g$, we have either $\mu(U) = 0$ or $\mu(X - U) = 0$.
\end{defn}
%
In the case of quasi-invariant measures, there is a useful alternative criterion for ergodicity.  Roughly speaking, an ergodic quasi-invariant measure has as many null sets as possible without vanishing entirely.  More precisely, we have the following lemma:
\begin{lemma} \label{ergodic}
Let $\mu$ be a quasi invariant measure with respect to a $G$-action. Then $\mu$ is ergodic if and only if any quasi-invariant measure $\nu$ that is absolutely continuous with respect to $\mu$ is  either zero or equivalent to $\mu$.
\end{lemma}
\begin{proof}
Assume first $\mu$ is ergodic. Let $\nu$ be a quasi-invariant measure with $\nu \ll \mu$.
Consider the Lebesgue decomposition $\mu = \mu^{\nu} + \overline{\mu^\nu}$.  As shown in Prop.~\ref{fact1}, the
two measures are mutually singular, so there is a measurable set $U$ such that $\mu^{\nu}(A) = \mu^{\nu}(A \cap U)$ for every measurable set $A$, and $\overline{\mu^\nu}(U) = 0$.
Hence, for all $g\in G$, $\mu^\nu(Ug - U) = 0$.  Now, $\nu \ll \mu$ implies $\mu^{\nu} \sim \nu$,
so we know $\mu^{\nu}$ is also quasi-invariant.  This implies
$\mu^\nu(U - Ug) = \mu((Ug^{-1} - U)g) = 0$ for all $g$.
We then have
\[
 \mu(U\,\triangle\, Ug) = \mu(Ug - U) + \mu(U-Ug) = 0
\]
for all $g\in G$. Since $\mu$ is ergodic, we conclude that either $\mu(U) = 0$,
in which case $\mu^\nu = 0$ and therefore $\nu=0$, or $\mu(X - U) = 0$, in which case $\mu\sim\mu^\nu$, and hence $\mu\sim \nu$.

Conversely, suppose every quasi-invariant measure subordinate to $\mu$ is either zero or equivalent to $\mu$.
Let $U$ be a measurable set such that $\mu(U\,\triangle\, Ug) = 0$ for all $g\in G$.
Define a measure $\nu$ by setting $\nu(A) = \mu(A \cap U)$ for each measurable set $A$.  Obviously $\nu \ll \mu$.   Since $U\,\triangle\,Ug$ is $\mu$-null,
\[
    \nu(A) = \mu(A\cap U) = \mu(A\cap Ug)
\]
for all $g$ and every measurable set $A$.  In particular, applying this to $Ag$,
\[
   \nu(Ag) = \mu(Ag\cap U) = \mu((A \cap U)g),
\]
so quasi-invariance of $\nu$ follows from that of $\mu$.
Thus, $\nu$ is a quasi-invariant measure such that $\nu \ll \mu$; this, by hypothesis, yields either $\nu=0$,
hence $\mu(U) = 0$, or $\nu \sim \mu$, hence $\mu(X - U) = 0$. We conclude that $\mu$ is ergodic.
\end{proof}

The notion of ergodic measure has an important generalization to the case of measurable families of measures:
\begin{defn}
\label{minimal.family} Let $X$ and $Y$ be measurable $G$-spaces.  A $Y$\!-indexed equivariant family of measures $\mu_y$ on $X$ is {\bf minimal}
if:
\begin{romanlist}
\item \label{minimal.family.1} there exists a $G$-orbit $Y_o$ in $Y$
 such that $\mu_y = 0$ for all $y\in Y- Y_o$, and
\item \label{minimal.family.2} for all $y$, $\mu_y$ is ergodic under the action of the
stabilizer $S_y \subset G$ of $y$.
\end{romanlist}
\end{defn}

Notice that an ergodic measure is simply a minimal family whose index space is the one-point $G$-space.  The criterion given in the previous lemma extends to the case of minimal equivariant families of measures:
\begin{lemma} \label{minimal family}
Let $\mu_y$ be an equivariant family of measures. The family is minimal if and only if, for  any equivariant family $\nu_y$
such that $\nu_y \ll \mu_y$ for all $y$, $\nu_y$ is either trivial or satisfies $\nu_y \sim \mu_y$ for all $y$.
\end{lemma}
\begin{proof}
The `only if' part of the statement is a direct application of Lemma \ref{ergodic}; let us prove the `if' part.

Suppose every equivariant family subordinate to $\mu_y$ is either zero or equivalent to $\mu_y$.  We first show that $\mu_y$ satisfies property $(i)$ in Def.~\ref{minimal.family}.   Assuming the family $\mu_y$ is non-trivial, let $Y_o$ be a G-orbit in $Y$ on
which $\mu_y \not=0$.  Define an equivariant family $\nu_y$ by setting $\nu_y = \mu_y$ if $y \in Y_o$ and $0$ otherwise.
This family is non-trivial and obviously satisfies $\nu_y \ll \mu_y$; this by hypothesis yields $\mu_y \sim \nu_y$.
Therefore $\mu_y = 0$ for all $y\in Y - Y_o$.

We now turn to property $(i)$ in Def.~\ref{minimal.family}.  Fix $y_o\in Y_o$, and let  $S_o \subseteq G$ be its stabilizer. To show that $\mu_{y_o}$ is ergodic under the ation of $S_o$ pick a measurable subset $U$ such that $\mu_{y_o}(U \,\triangle\, Us) = 0$ for all $s \in S_o$.  By equivariance of the family $\mu_y$, this implies
\beq
\label{equivariant.symmetric.differences}
    \mu_{y_o g} (Ug \,\triangle\, Usg)
\eeq
for all $s \in S_o$ and all $g\in G$.  Then, for every $y = y_o g$ in $Y_o$, define a measure $\nu_y$ by setting $\nu_y(A) = \mu_y(A \cap Ug)$. This is well defined, since any $g'$ such that $y = y_o g'$ is given by $g'=sg$ for some $s\in S_o$, and by (\ref{equivariant.symmetric.differences}) we have $\mu_y(A \cap Ug) = \mu_y(A \cap U sg)$.

The family $\nu_y$ is equivariant; indeed, for any $g\in G$, and $y = y_o g'\in Y_o$:
\[
\nu_{yg}(Ag) = \mu_{yg}((A \cap Ug')g),
\]
so equivariance of $\nu_y$ follows from that of $\mu_y$. Since we also obviously have $\nu_y \ll \mu_y$ for all $y$, by hypothesis the family $\nu_y$ is either trivial, or satisfies $\nu_y \sim \mu_y$ for all $y$. In the former case, $\mu_{y_o}(U) = 0$; in the latter, $\mu_{y_o}(X - U) = 0$. Thus, $\mu_{y_o}$ is ergodic under the action of $S_o$.  Since $y_o$ was arbitrary, $(ii)$ is proved, and the family $\mu_y$ is minimal.
\end{proof}

{\em Transitive} families of measures, for which there exists a $G$-orbit $o$ in $Y\times X$ such that $\mu_y(A) = 0$ for every measurable $\{y\} \times A$ in the complement $Y\times X- \, o$,  are particular examples of minimal families. Indeed, the obvious projection $Y\times X \to Y$ maps the orbit $o$ into an orbit $Y_o$ such that $\mu_y=0$ unless $y\in Y_o$; furthermore for all $y \in Y_o$,   $\mu_y$ is quasi-invariant under the action of the stabilizer $S_y$ of $y$ and concentrates on a single orbit,  so it is clearly ergodic.

It is useful to investigate the converse: Is a minimal family of measures necessarily transitive?

This is not the case, in general. To understand this, we need to dwell further on the notion of quasi-invariant ergodic measure. First note that each orbit in $X$ naturally defines a measure class of  such measures: we indeed know that an orbit defines a measure class of quasi-invariant measures; now the uniqueness of such a class yields the minimality property stated in Lemma.~\ref{ergodic}, hence the ergodicity of the measures.

However, not every quasi-invariant and ergodic measure need belong to one of the classes defined by the orbits.
In fact, given a measure $\mu$ on $X$, quasi-invariant and ergodic under a measurable $G$-action, there should be {\em at most} one orbit with positive measure, and its complement in $X$ should be a null set. If there is an orbit with positive measure, $\mu$ belongs to the class that the orbit defines.
But it may also very well be that {\em all} $G$-orbits are null sets. Consider for example the group $G = \Z$,  acting on the unit circle  $X = \{z \in \C \, | \,  |z| = 1\}$  in the complex plane as $e^{i\theta} \mapsto e^{i \theta + \alpha \pi}$, where $\alpha  \in \R - \Q$ is some fixed irrational number.  It can be shown that the linear measure $\extd \theta$ on $X$ is ergodic, whereas the orbits, which are all countable, are null sets.

This makes the classification of the equivalence classes of ergodic quasi-invariant measures quite difficult in general. Luckily, there is a simple criterion, stated in the following lemma, that precludes the kind of behaviour illustrated in the above example.
For $X$ a measurable $G$-space, we call a measurable subset $N \subset X$ a {\bf measurable cross-section} if it intersects each $G$-orbit in exactly one point.
\begin{lemma} \label{lemma.minmeas}
{\bf \cite[Lemma 6.14]{Varadarajan}}
Let $X$ be a measurable $G$-space. If $X$ has a measurable cross-section, then any ergodic measure on $X$ is supported on a single $G$-orbit.
\end{lemma}

Roughly speaking, the existence of a measurable cross-section ensures
that the orbit space is ``nice enough''. Thus, for example, making
such assumption is equivalent to requiring that the orbit space is
countably separated as a Borel space; or, in the case of a continuous
group action, that it is a $T_0$ space \cite{Glimm}.

\medskip

Having introduced these concepts, we now begin our study of indecomposable, irretractable and irreducible intertwiners. Consider a pair of representations $\rho_1$ and $\rho_2$ on the measurable categories $H^X$ and $H^Y$; denote by $\chi_1$ and $\chi_2$ the corresponding fields of characters.  Let $\phi \maps \rho_1 \to \rho_2$ be  an intertwiner; denote by $\mu_y$ the corresponding equivariant and fiberwise family $\mu_y$ of measures on $X$.

The following proposition gives a necessary condition for the intertwiner to be indecomposable (hence to be irretractable or irreducible):
%
\begin{prop} \label{cond.IrretInt}
If the intertwiner $\phi=(\phi,\Phi,\mu)$ is indecomposable, its family of measures $\mu_y$ is minimal.
\end{prop}
%
%
\begin{proof}
Assume $\phi$ is indecomposable, and  consider an equivariant family of measures $\nu_y$ such that $\nu_y \ll \mu_y$ for all $y$. The Lebesgue decompositions:
\[
\mu_y = \mu_y^{\nu_y} + \overline{\mu_y^{\nu_y}}
\]
define two new fiberwise and equivariant families measures. Together with the $\mu^\nu$-classes of fields and the $\overline{\mu^{\nu}}$-classes of fields induced by the $\mu$-classes of $\phi$, these specify two intertwiners $\psi, \overline{\psi}$.

The measures $\mu_y^{\nu_y}$ and $\overline{\mu_y^{\nu_y}}$ are mutually singular for all $y$. Using the definition of the direct sum of intertwiners, we find
\[
\phi = \psi \oplus \overline{\psi}.
\]

Now, by hypothesis $\phi$ is indecomposable, so that either $\psi$ or $\overline{\psi}$ is the null intertwiner. In the former case, the family $\mu^\nu$ is trivial. This means that $\nu_y \perp \mu_y$ for all $y$; since, furthermore,  $\nu_y \ll \mu_y$, it implies that $\nu$ is trivial. In the latter case, the family $\overline{\mu^\nu}$ is trivial. This mean that $\nu_y \sim \mu_y$ for all $y$. We conclude with Lemma~\ref{minimal family} that the family $\mu_y$ is minimal.
\end{proof}

We can be more precise by focusing on the {\em transitive}
intertwiners, as defined in Def.~\ref{Transitive-Int}.  Suppose the
intertwiner $\phi\maps \rho_1 \to \rho_2$ is transitive, and specified
by the assignments $\phi_{y,x}, \Phi^g_{y,x}$ of Hilbert spaces and
invertible maps to the points of an $G$-orbit $o$ in $Y \times
X$. These define ordinary linear representations
$\mathcal{R}^\phi_{y,x}$ of the stabilizer $S_{y,x}$ of $(y,x)$ under
the diagonal action of $G$.

The following propositions give a criterion for $\phi$ to be indecomposable, irretractable, or irreducible:
%
%
\begin{prop}[Indecomposable and irretractable transitive interwiners]
\label{Indec-Irret-TransInt}
Let $\phi=(\phi, \Phi, \mu)$ be a transitive intertwiner.  Then the following are equivalent:
\begin{itemize}
  \item $\phi$ is indecomposable 
  \item $\phi$ is irretractable
  \item the stabilizer representations $\mathcal{R}^\phi_{y,x}$ are indecomposable.
\end{itemize}
\end{prop}
%
%
%
\begin{prop}[Irreducible transitive interwiners]
\label{Irred-TransInt}
Let $\phi=(\phi, \Phi, \mu)$ be a transitive intertwiner.  Then $\phi$ is irreducible if and only if the stabilizer representations $\mathcal{R}^\phi_{y,x}$ are irreducible.
\end{prop}
%

Let us prove these two propositions together:

\begin{proof}
Fix a point $y_o \in Y$ such that $\mu_{y_o} \not=0$, and let $S_{y_o}$ be its stabilizer. The action of $G$ on $X$ induces an action of $S_{y_o}$ on the fiber over $\chi_2(y_o)$ in $X$. Since by hypothesis $\phi$ is transitive, $\mu_{y_o}$ concentrates on a single $S_{y_o}$-orbit $\imath_o\subseteq X$.
Next, fix $x_o$ in $\imath_o$,  and let $S_o = S_{y_o,x_o}$ denote stabilizer of $(y_o, x_o)$ under the diagonal action.  Let also $\phi_o = \phi_{y_o,x_o}, \Phi^g_o= \Phi^g_{y_o,x_o}$ be the space and maps  assigned to the point $(y_o, x_o)$, and let $\mathcal{R}^\phi_o$ be the corresponding linear representation $s\mapsto\Phi^s_o$ of $S_o$. Note that the representations $\mathcal{R}^\phi_{y,x}$ are all indecomposable (or irreducible) if  $\mathcal{R}^\phi_o$ is.

We begin with Prop.~\ref{Indec-Irret-TransInt}. Suppose first that $\phi$ is indecomposable. Consider a Hilbert space decomposition $\phi_o = \phi'_o \oplus \phi''_o$ that is invariant under $\mathcal{R}^\phi_o$; assume that $\phi'_o$ is non-trivial.  Given a point $(y,x) = (y_o,x_o)g^{-1}$ in the orbit, the isomorphism $\Phi^g_o\maps \phi_o \to \phi_{y,x}$ gives a splitting
\[
 \phi_{y,x}= \phi'_{y,x}\oplus \phi''_{y,x}  \quad \text{where}\quad \phi'_{y,x} = \Phi^g_o(\phi'_o),\;
  \phi''_{y,x} = \Phi^g_o(\phi''_o)
\]
This decomposition is independent of the representative of $gS_o$ chosen, hence depends only on the point $(y,x)$; indeed, for every $s\in S_o$, we have
\[
   \Phi_o^{gs}(\phi'_o) = \Phi^g_o \Phi^s_o(\phi'_o) = \Phi^g_o(\phi'_o) = \phi'_{y,x}
\]
by invariance of $\phi'_o$,  and likewise for $\phi''$.  The decomposition of $\phi_{y,x}$ is also invariant under the representation $\mathcal{R}^\phi_{y,x}$ of $S_{y,x}$:
\[
    \Phi^s_{y,x}(\phi'_{y,x}) = \Phi^{sg}_{y,x}(\phi'_o) = \Phi^{s(g^{-1}sg)}(\phi'_o) = \Phi^g_o(\phi'_o) = \phi'_{y,x}
\]
since for any $s\in S_{y,x}$, we have $(g^{-1}sg)\in S_o$, and likewise for $\phi''_{y,x}$.
These data give us a transitive intertwiner $\phi' = (\phi'_{y,x}, \Phi'^g_{y,x},\mu_y)$, where $\Phi'^g_{y,x}$ simply denotes the restriction of $\Phi^g_{y,x}$ to $\phi'_{y,x}$.

By construction, $\phi'$ is a summand of  $\phi$, distinct from the null intertwiner.  Now, we have assumed $\phi$ is indecomposable; so we have that $\phi' \simeq \phi$. We then deduce from Prop.~\ref{equ1int} that the representation $\mathcal{R}_{o}^\phi$ is equivalent to its restriction to $\phi'_{o}$. Thus,  $\mathcal{R}_{o}^\phi$ is indecomposable.

Next, suppose that the linear representations $\mathcal{R}_{y,x}^\phi$ are indecomposable. We will show that $\phi$ is irretractable; since an irretractable is automatically indecomposable, this will complete the proof of Prop.~\ref{Indec-Irret-TransInt}.

Let  $\phi'$ be a retract of $\phi$, specified by the family of measures $\mu'_y$ and the assignments of Hilbert spaces $\phi'_{y,x}$ and invertible maps $\Phi'^g_{y,x}$. By definition one can find 2-intertwiners $m\maps \phi \To \phi'$ and
$n \maps \phi' \To \phi$ such that $n\cdot m = \unit_{\phi'}$.  This last equality requires that the geometric mean measures $\sqrt{\mu_y\mu_y'}$ be equivalent to $\mu'_y$, or equivalently that $\mu'_y \ll \mu_y$.  Hence, the $S_o$-quasi-invariant measure $\mu'_{y_o}$ concentrates on the orbit $\imath_o$.  Non-trivial $S_o$-quasi-invariant measures on $\imath_o$ are unique up to equivalence, so we conclude that $\mu'_{y_o}$ is either trivial or equivalent to $\mu_{y_o}$. In the first case, $\phi'$ is trivial, so we are done.

In the second case, where $\mu'_y \sim \mu_y$ for all $y$, the linear maps $m_{y,x} \maps \phi'_{y,x} \to \phi_{y,x}$ and $n_{y,x} \maps \phi_{y,x} \to \phi'_{y,x}$ are intertwining operators between the representations $\mathcal{R}^{\phi'}_{y,x}$ and $\mathcal{R}^\phi_{y,x}$; they satisfy $n_{y,x} m_{y,x} = \unit_{\phi'_{y,x}}$. Thus,  $\mathcal{R}^{\phi'}_{y,x}$ is a retract of $\mathcal{R}^\phi_{y,x}$, hence a direct summand.  But $\mathcal{R}^\phi_{y,x}$ is indecomposable: so the two representations must be equivalent. Hence the map $m_{y,x}$ is invertible.  The 2-intertwiner $m\maps \phi' \To \phi$ is thus invertible, which shows $\phi'$ and $\phi$ are equivalent. We conclude that $\phi$ is irretractable.

We now prove Prop.~\ref{Irred-TransInt}. Suppose first that $\phi$ is irreducible.
Consider a non-trivial subspace $\phi'_o \subset \phi_o$ that is invariant under $\mathcal{R}^\phi_o$. Given a point $(y,x) = (y_o,x_o)g^{-1}$ in the orbit, the isomorphism $\Phi^g_o\maps \phi_o \to \phi_{y,x}$ gives a subspace $\phi'_{y,x} := \Phi^g_{y,x}(\phi'_o)$ of $\phi_{y,x}$ that is invariant under $\mathcal{R}^\phi_{y,x}$.  These data give us a transitive intertwiner $\phi' = (\phi'_{y,x}, \Phi'^g_{y,x},\mu_y)$, where $\Phi'^g_{y,x}$ simply denotes the restriction of $\Phi^g_{y,x}$ to $\phi'_{y,x}$.

The canonical injections $\imath_{y,x} \maps \phi'_{y,x} \to \phi_{y,x}$ define a monic 2-intertwiner $\imath \maps \phi' \to \phi$; this shows that $\phi'$ is a sub-intertwiner of $\phi$. But $\phi$ is irreducible: we therefore have that $\phi' \simeq \phi$. This means that $\mathcal{R}_{o}^\phi$ is equivalent to its restriction to $\phi'_{o}$. Thus, $\mathcal{R}_{o}^\phi$ is irreducible.

Conversely, suppose that the representations $\mathcal{R}_{y,x}^\phi$ are irreducible. Consider a sub-intertwiner $\phi'$ of $\phi$, giving a family of measures $\mu'_y$. First of all, note that the existence of a monic 2-intertwiner $m \maps \phi' \to \phi$ forces $\mu'$ to be transitive, with $\bar{\mu}_y$ supported on the orbit $o$. In particular, we have that $\mu'_y \sim \mu_y$ for all $y$.

Next, fix a monic 2-intertwiner $m$. It gives injective linear maps $m_{y,x} \maps \phi'_{y,x} \to \phi_{y,x}$; these define subspaces $m_{y,x}(\phi'_{y,x})$ in $\phi_{y,x}$ that are invariant under $\mathcal{R}^\phi_{y,x}$. Since the representations are by hypothesis irreducible, this means that the maps $m_{y,x}$, and hence $m$, are invertible. We obtain that $\phi' \simeq \phi$, and conclude that $\phi$ is irreducible.
\end{proof}

These results allow us to classify, up to equivalence, the indecomposable and irreducible intertwiners between fixed measurable representations $\rho_1, \rho_2$. We may assume that these representations are indecomposable, and given by the pairs $(o, S_1)$ and $(o, S_2)$. They are thus specified by the $G$-equivariant bundles $X = G/S_1$ and $Y = G/S_2$ over the same $G$-orbit $o \simeq G/S^\ast_o$ in $H^\ast$; $S_1$ and  $S_2$ are some closed subgroups of $S^\ast_o$. In the following, we denote $y_o = S_2 e$, and fix a (not necessarily measurable) cross-section of the $S_2$-space $S^\ast_o/S_1$ -- namely, a subset that intersects each $S_2$-orbit $\imath_o$ in exactly one point $x_o := S_1k_o$.

Let $\phi\maps \rho_1 \to \rho_2$ an indecomposable (resp.\ irreducible) intertwiner. We will assume that $\phi$ is {\sl transitive}, keeping in mind the following consequence of Prop.~\ref{cond.IrretInt} and Lemma \ref{lemma.minmeas}:
\begin{lemma}
Suppose that the $S_2$-space $S^\ast_o/S_1$ has a measurable cross-section. Then every indecomposable intertwiner $\phi\maps (o, S_1) \to (o, S_2)$ is transitive.
\end{lemma}

The intertwiner $\phi$ gives a non-trivial $S_2$-quasi-invariant measure $\mu_{y_o}$ in the fiber $S^\ast_o/S_1 \subset X$ over $S^\ast_o e$. Moreover, the transitivity of $\phi$ implies that this measure is supported on a single $S_2$-orbit $\imath^{\phi}_o$ in $S^\ast_o/S_1$. Note that any two such measures are equivalent.
$\phi$  also gives an indecomposable (resp. irreducible) linear representation $\mathcal{R}^\phi_o$ of the group
\[
S_o = k_o^{-1} S_1 k_o \cap S_2.
\]
So, $\phi$ gives a pair $(\imath^\phi_o, \mathcal{R}^\phi_o)$, where $\imath^\phi_o$ is a $S_2$-orbit in 
$S^\ast_o /S_1$ and $\mathcal{R}^\phi_o$ is an indecomposable (resp. irreducible) representation of $S_o$. 
We easily deduce from Prop.~\ref{equ-trans1int} that two equivalent transitive  intertwiners give two pairs with the same orbit and equivalent linear representations.

Conversely, given any orbit $\imath_o$ and any linear representation $\mathcal{R}_o$ of $S_o$ on some Hilbert space $\phi_o$, there is an intertwiner $\phi = (\phi, \Phi, \mu)$ such that $\imath_o^\phi = \imath_o$ and $\mathcal{R}_o^\phi = \mathcal{R}_o$.
Indeed, a measurable equivariant and fiberwise family of measures is obtained by choosing a $S_2$-quasi-invariant measure $\mu_o$ supported on  $\imath_o$ and a measurable section $n \maps G/S_2 \to G$, and by setting, for each $y = y_o n(k)$:
\[
\mu_{y} := \mu^{n(k)}_{o}
\]
To construct the measurable fields of spaces and linear maps, fix a measurable section $\bar{n} \maps G/S_o \to G$, denote by $\pi \maps G \to G/S_o$ the canonical projection, and consider the function $\alpha \maps G \to S_o$ given by:
\[
\alpha(g) = (\bar{n}\pi)(g^{-1}) g.
\]
This function satisfies the property that $\alpha(gs) = \alpha(g) s$ for all $s \in S_o$. Using this, we define a family $\Phi^g_o$ of isomorphisms of $\phi_o$ as:
\[
\Phi^g_o = \mathcal{R}_o(\alpha(g))
\]
and construct a measurable field $\Phi^g_{y,x}$ by setting, for each $(y,x) = (y_o, x_o) k^{-1}$:
\[\Phi^g_{y,x} = \Phi^g_o (\Phi^k_o)^{-1}.\]
These data specify a transitive intertwiner $\phi$; this intertwiner is  indecomposable (resp. irreducible) if $\mathcal{R}_o$ is.

These remarks yield the following:

%
\begin{cor} \label{classification-TransIrredInt}
Indecomposable (resp. irreducible)  transitive intertwiners $\phi \maps (o, S_1) \to (o, S_2)$ are classified, up to equivalence, by a choice of a $S_2$-orbit $\imath_o$ in  $S^\ast_o / S_1$, along with an equivalence class of indecomposable (resp. irreducible) linear representations $\mathcal{R}_o$ of the group $k_o^{-1} S_1 k_o \cap S_2$.
\end{cor}
%

We close this section with a version of Schur's lemma for
irreducible intertwiners:
%
%
\begin{prop}[Schur's Lemma for Intertwiners]
\label{schur-intertwiner}
Let $\phi, \psi \maps (o, S_1), \to (o, S_2)$ be two irreducible transitive intertwiners. Then any 2-intertwiner $m \maps \phi \To \psi$ is either null or an isomorphism. In the latter case, $m$ is unique, up to a normalization factor.
\end{prop}
%
%
%
\begin{proof}
We may assume $\phi$ and $\psi$ are given by the pairs
$(\imath^\phi_o, \mathcal{R}^\phi_o)$ and $(\imath^\psi_o,
\mathcal{R}^\psi_o)$ of $S_2$-orbits in $S^\ast_o /S_1$ and
irreducible linear representations.  Let $\mu_y$ and $\nu_y$ denote
the two families of measures. If the orbits are distinct
$\imath^\phi_o \not= \imath^\psi_o$, the measures $\mu_{y}$ and
$\nu_{y}$ have disjoint support, so that their geometric mean is
trivial. In this case, any 2-intertwiner $m \maps \phi \To \psi$ is
trivial.

Suppose now $\imath^\phi_o = \imath^\psi_o$. In this case, we have
that $\mu_y \sim \nu_y$ for all $y$. Let $m \maps \phi \To \psi$ be a
2-intertwiner, given by the assignment of linear maps $m_{y,x} \maps
\phi_{y,x} \to \psi_{y,x}$. Because of the intertwining rule
(\ref{rule2int}), the assignment is entirely specified by the data
$m_o := m_{y_o, x_o}$.

Now, $m_o$ defines a standard intertwiner between the irreducible
linear representations $\mathcal{R}^\phi_o$ and
$\mathcal{R}^\psi_o$. Therefore $m_o$, and hence $m$, is either
trivial or invertible; in the latter case, it is unique, up to a
normalization factor.
\end{proof}

\section{Conclusion}
\label{conclusion}

We conclude with some possible avenues for future investigation.
First, it will be interesting to study examples of the general theory
described here.  As explained in the Introduction, representations of
the Poincar\'e 2-group have already been studied by Crane and
Sheppeard \cite{CraneSheppeard}, in view of obtaining a 4-dimensional
state sum model with possible relations to quantum gravity.
Representations of the Euclidean 2-group (with $G = {\rm SO}(4)$
acting on $H = \R^4$ in the usual way) are somewhat more tractable.
Copying the ideas of Crane and Sheppeard, this 2-group gives a state
sum model \cite{BaratinFreidel1, BaratinFreidel2} with interesting
relations to the more familiar Ooguri model.

There are also many other 2-groups whose representations are worth
studying.  For example, Bartlett has studied representations of finite
groups $G$, regarded as 2-groups with trivial $H$ \cite{Bartlett}.  He
considers {\it weak} representations of these 2-groups, where
composition of 1-morphisms is preserved only up to 2-isomorphism.
More precisely, he considers {\it unitary} weak representations on
finite-dimensional 2-Hibert spaces.  These choices lead him to a
beautiful geometrical picture of representations, intertwiners and
2-intertwiners --- strikingly similar to our work here, but with
$\U(1)$ gerbes playing a major role.  So, it will be very interesting
to generalize our work to weak representations, and specialize it to
unitary ones.

To define {\it unitary} representations of measurable 2-groups, we
need them to act on something with more structure than a measurable
category: namely, some sort of infinite-dimensional 2-Hilbert space.
This notion has not yet been defined.  However, we may hazard a guess
on how the definition should go.

In Section \ref{measurable_categories}, we argued that the measurable
category $H^X$ should be a categorified analogue of $L^2(X)$, with
direct integrals replacing ordinary integrals.  However, we never
discussed the inner product in $H^X$.  We can define this only after
choosing a measure $\mu$ on $X$.  This measure appears in the formula
for the inner product of vectors $\psi, \phi \in L^2(X)$:
\[       \langle \psi, \phi \rangle = \int \overline{\psi}(x)
\phi(x) \, \extd \mu(x)  \in \C .\]
Similarly, we can use it to define the {\bf inner product} of
fields of Hilbert spaces $\H, \K \in H^X$:
\[       \langle \H, \K \rangle = \direct \overline{\H}(x)
\otimes \K(x) \, \extd \mu(x) \in \Hilb . \]
Here $\overline{\H}(x)$ is the complex conjugate of the
Hilbert space $\H(x)$, where multiplication by $i$ has been
redefined to be multiplication by $-i$.  This is naturally
isomorphic to the Hilbert space dual $\H(x)^*$, so we can also
write
\[       \langle \H, \K \rangle \cong \direct {\H}(x)^*
\otimes \K(x) \, \extd \mu(x) . \]

Recall that throughout this paper we are assuming our measures are
$\sigma$-finite; this guarantees that the Hilbert space $\langle \H,
\K \rangle$ is {\it separable}.  So, we may give a preliminary
definition of a `separable 2-Hilbert space' as a category of the form
$H^X$ where $X$ is a measurable space equipped with a measure $\mu$.

As a sign that this definition is on the right track, note that when
$X$ is a finite set equipped with a measure, $H^X$ is a
finite-dimensional 2-Hilbert space as previously defined \cite{Baez2}.
Moreover, every finite-dimensional 2-Hilbert space is equivalent
to one of this form \cite[Sec.\ 2.1.2]{Bartlett}.

The main thing we lack in the infinite-dimensional case, which we
possess in the finite-dimensional case, is an intrinsic definition of
a 2-Hilbert space.  An intrinsic definition should not refer to the
measurable space $X$, since this space merely serves as a `choice of
basis'.  The problem is that it seems tricky to define direct
integrals of objects without mentioning this space $X$.

The same problem afflicted our treatment of measurable categories.
Instead of giving an intrinsic definition of measurable categories, we
defined a measurable category to be a $C^*$\!-category that is
$C^*$\!-equivalent to $H^X$ for some measurable space $X$.  This made
the construction of $\me$ rather roundabout.  We could try a similar
approach to defining a 2-category of separable 2-Hilbert spaces, but
it would be equally roundabout.

Luckily there is another approach, essentially equivalent to the one
just presented, that does not mention measure spaces or measurable
categories!  In this approach, we think of a 2-Hilbert space as a
category of representations of a commutative von Neumann algebra.

The key step is to notice that when $\mu$ is a measure on a measurable
space $X$, the algebra $L^\infty(X,\mu)$ acts as multiplication
operators on $L^2(X,\mu)$.  Using this one can think of
$L^\infty(X,\mu)$ as a commutative von Neumann algebra of operators on
a separable Hilbert space.  Conversely, any commutative von Neumann
algebra of operators on a separable Hilbert space is isomorphic---as
a $C^*$\!-algebra---to one of this form \cite[Part I, Chap.\ 7, Thm.\
1]{Dixmier}.  The technical conditions built into our definition of
`measurable space' and `measure' are precisely what is required to
make this work (see Defs.\ \ref{defn:measurable_space} and
\ref{defn:measure}).

This viewpoint gives a new outlook on fields of Hilbert spaces.
Suppose $A$ is commutative von Neumann algebra of operators on a
separable Hilbert space.  As a $C^*$\!-algebra, we may identify $A$ with
$L^\infty(X,\mu)$ for some measure $\mu$ on a measurable space $X$.
Define a {\bf separable representation} of $A$ to be a representation of
$A$ on a separable Hilbert space.  It can then be shown that every
separable representation of $A$ is equivalent to the representation of
$L^\infty(X,\mu)$ as multiplication operators on $\direct \H(x) \extd
\mu(x)$ for some field of Hilbert spaces $\H$ on $X$.  Moreover, this
field $\H$ is essentially unique \cite[Part I, Chap. 6, Thms.\ 2 and
3]{Dixmier}.

This suggests that we define a {\bf separable 2-Hilbert space} to be a
category of separable representations of some commutative von Neumann
algebra of operators on a separable Hilbert space.  More generally, we
could drop the separability condition and define a {\bf 2-Hilbert space}
to be a category of representations of a commutative von Neumann
algebra.

While elegant, this definition is not quite right.  Any category
`equivalent' to the category of representations of a commutative von
Neumann algebra---in a suitable sense of `equivalent', probably stronger
than $C^*$\!-equivalence---should also count as a 2-Hilbert space.  A
better approach would give an intrinsic characterization of categories
of this form.  Then it would become a {\it theorem} that every
2-Hilbert space is equivalent to the category of representations of
a commutative von Neumann algebra.

Luckily, there is yet another simplification to be made.  After all,
a commutative von Neumann algebra can be recovered, up to
isomorphism, from its category of representations.  So, we can
forget the category of representations and focus on the von Neumann
algebra itself!

The problem is then to redescribe morphisms between 2-Hilbert spaces,
and 2-morphisms between these, in the language of von Neumann algebras.
There is a natural guess as to how this should work, due to
Urs Schreiber.  Namely, we can define a bicategory $\twoHi$ for which:
\begin{itemize}
\item
objects are commutative von Neumann algebras $A,B, \dots$,
\item
a morphism $\H \maps A \to B$ is a Hilbert space $\H$ equipped with
the structure of a $(B,A)$-bimodule,
\item
a 2-morphism $f \maps \H \to \K$ is a homomorphism of $(B,A)$-bimodules.
\end{itemize}
Composition of morphisms corresponds to tensoring bimodules.  Note
also that given an $(B,A)$-bimodule and a representation of $A$,
we can tensor the two and get a representation of $B$.  This is how
an $(B,A)$-bimodule gives a functor from the category of representations
of $A$ to the category of representations of $B$.  Similarly, a
homomorphism of $(B,A)$-bimodules gives a natural transformation between
such functors.

Let us briefly sketch the relation between this version of $\twoHi$
and the 2-category $\me$ described in this paper.  First, given
separable commutative von Neumann algebras $A$ and $B$, we can write
$A \cong L^\infty(X,\mu)$ and $B \cong L^\infty(Y,\nu)$ where $X,Y$
are measurable spaces and $\mu,\nu$ are measures.  Then, given an
$(B,A)$-bimodule, we can think of it as a representation of $B
\otimes A \cong L^\infty(Y \times X, \nu \tensor \mu)$.  By the remarks
above, this representation comes from a field of Hilbert spaces on $Y
\times X$.  Then, given a 2-morphism $f \maps \H \to \K$, we can
represent it as a measurable field of bounded operators between the
corresponding fields of Hilbert spaces.

While the details still need to be worked out, all this suggests that
a theory of 2-Hilbert spaces based on commutative von Neumann algebras
should be closely linked to the theory of measurable categories described
here.

Even better, the bicategory $2\Hilb$ just described sits inside a
larger bicategory where we drop the condition that the von Neumann
algebras be commutative.  Representations of 2-groups in this larger
bicategory should also be interesting.  The reason is that Schreiber
has convincing evidence that the work of Stolz and Teichner
\cite{StolzTeichner} provides a representation of the so-called
`string 2-group' \cite{BCSS} inside this larger bicategory.  For
details, see the last section of Schreiber's recent paper on two
approaches to quantum field theory \cite{Schreiber}.  This is yet
another hint that infinite-dimensional representations of 2-groups
may someday be useful in physics.

\subsection*{Acknowledgments}

We thank Jeffrey Morton for collaboration in the early stages of this
project.  We also thank Jerome Kaminker, Benjamin Weiss, and the
denizens of the $n$-Category Caf\'e, especially Bruce Bartlett and Urs
Schreiber, for many useful discussions.  Yves de Cornulier and Todd
Trimble came up with most of the ideas in Appendix
\ref{apx:measurable_groups}.  Our work was supported in part by the
National Science Foundation under grant DMS-0636297, and by the
Perimeter Institute for Theoretical Physics.

\appendix

\section{Tools from measure theory}
\label{tools}

This Appendix summarizes some tools of measure theory used in the
paper. The first section recalls basic terminology and states
the well-known Lebesgue decomposition and Radon-Nikodym theorems. The
second section defines the geometric mean of two measures
and derives of some of its key properties.  The third section studies
measurable abelian groups and their duals.  Finally, the fourth section
presents a few standard results about measure theory on $G$-spaces.

Recall that for us, a {\bf measurable space} is shorthand for a
{\bf standard Borel space}: that is, a set $X$ with a $\sigma$-algebra
$\mathcal{B}$ of subsets generated by the open sets for some
second countable locally compact Hausdorff topology on $X$.
We gave two other equivalent definitions of this concept in
Prop.\ \ref{lem:standard_Borel}.

Also recall that for us, all measures are $\sigma$-finite.  So,
a {\bf measure} on $X$ is a function $\mu \maps \mathcal{B} \to
[0,+\infty]$ such that
$$
\mu(\bigcup_n A_n) = \sum_n \mu(A_n)
$$
for any sequence $(A_n)_{n\in \N}$ of mutually disjoint measurable
sets, such that $X$ is a countable union of
$S_i \in \mathcal{B}$ with $\mu(S_i) < \infty$.

\subsection{Lebesgue decomposition and Radon-Nikodym derivatives}
\label{Lebesgue-Radon-Nykodym}

In a fixed measurable space $X$, a measure $t$ is {\bf absolutely
continuous} with respect to a measure $u$, written $t\ll u$, if every
$u$-null set is also $t$-null.  The measures are {\bf equivalent},
written $t \sim u$, if they are absolutely continuous with respect to
each other: in other words, they have the same null sets.  The two
measures are {\bf mutually singular}, written $t \perp u$, if we can
find a measurable set $A \subseteq X$ such that
\[   t(A) = u(X - A) = 0.\]
If $A \subseteq X$ is a measurable set with $u(X - A) = 0$ we say the
measure $u$ is {\bf supported} on $A$.

%
\begin{theo}[Lebesgue decomposition]
Let $t$ and $u$ be ($\sigma$-finite) measures on $X$. Then there is a
unique pair of measures $t^u$ and $\overline{t^u}$ such that
 \[
     t = t^u + \overline{t^u}
     \qquad \mbox{with} \quad t^u \ll u \quad \mbox{and}
\quad \overline{t^u} \perp u.
\]
\end{theo}
The notation chosen here is particularly useful when we have more than
two measures around and need to distinguish between Lebesgue
decompositions with respect to different measures.

This result is completed by the following useful propositions. Fix
two measures $t$ and $u$ on $X$.
\begin{prop} \label{fact1}
In the Lebesgue decomposition $t = t^u + \overline{t^u}$, we have
$t^u \perp \overline{t^u}$.
\end{prop}
\begin{proof}
Given that $\overline{t^u} \perp u$, there is a
measurable set $A$ such that $u$ is supported on $A$
and $\overline{t^u}$ is supported on $X-A$:
\[
      u(S) = u(S\cap A) \qquad \overline{t^u}(S) = \overline{t^u}(S - A)
\]
for all measurable sets $S$.  But then absolute continuity of $t^u$
with respect to $u$ implies $t^u(X-A) = 0$, and therefore $t^u(S)=
t^u(S\cap A)$.  That is, $t^u$ is supported on $A$, so $t^u \perp
\overline{t^u}$.
\end{proof}
\begin{prop} \label{mutual-decomposition}
Consider the Lebesgue decompositions $t = t^u + \overline{t^u}$ and $u
= u^t + \overline{u^t}$.  Then $\overline{t^u} \perp \overline{u^t}$
and $t^u \sim u^t$.
\end{prop}
\begin{proof}
Given that $\overline{t^u} \perp u$, there is a measurable set $A$
such that $u$ is supported on $A$ and $\overline{t^u}$ is supported
on $X-A$.  Note first that $\overline{u^t}$ is supported on $A$, as
$u$ is. This shows that $\overline{t^u} \perp \overline{u^t}$.

Next, fix a $t^u$-null set $S$; we thus have that $t(S) =
\overline{t^u}(S)$. Since $\overline{t^u}$ is supported on $X-A$, it
follows that $t(S\cap A) = 0$. Using the fact that $u^t \ll t$, we
obtain $u^t(S \cap A) = 0$. But $u^t$ is supported on $A$, as $u$ is;
therefore $u^t(S)= u^t(S \cap A) = 0$. Thus, we have shown that $u^t
\ll t^u$. We show similarly $t^u \ll u^t$, and conclude that $t^u \sim
u^t$.
\end{proof}

The Lebesgue decomposition theorem is refined by the Radon--Nikodym
theorem, which provides a classification of absolutely continuous
measures:
\begin{theo}[Radon-Nikodym]
Let $t$ and $u$ be two $\sigma$-finite measures on $X$. Then $t\ll u$
if and only if $t$ can be written as $u$ times a function
$\rnd{t}{u}$, the {\bf Radon--Nikodym derivative}: that is,
\[
        t(A) = \int_A \extd u \, \rnd{t}{u}
\]
\end{theo}

\subsection{Geometric mean measure}
\label{apx:gmean}

Suppose $X$ is a measurable space on which are defined two
measures, $u$ and $t$.  If each measure is absolutely
continuous with respect to the other, then we have the equality
\[
      \sqrt{\frac{\extd t}{\extd u}} \extd u =
      \sqrt{\frac{\extd  u}{\extd t }} \extd t
\]
so we can define the `geometric mean' $\sqrt{\extd u \extd t}$ of the
two measures to be given by either side of this equality.  In the more
general case, where $u$ and $t$ are not necessarily mutually
absolutely continuous, we may still define $\sqrt{\extd u \extd t}$,
as we shall see.

Using the notation of the first section we have the
following key fact.  Recall once more that all our measures are assumed
$\sigma$-finite.
\begin{prop}
If $u$ and $t$ are measures on the same measurable
space $X$ then
\[
      \sqrt{\frac{\extd t^u}{\extd u}} \extd u
        = \sqrt{\frac{\extd u^t}{\extd t }} \extd t
\]
\end{prop}
\begin{proof}
Our notation for the Lebesgue decomposition means
\[
     u = u^t + \overline{u^t}  \qquad u^t \ll t \quad \overline{u^t} \perp t
\]
and likewise,
\[
         t  = t^u + \overline{t^u} \qquad t^u \ll u \quad \overline{t^u} \perp u.
\]
Prop.\ \ref{fact1} shows that $u^t$ and $\overline{u^t}$ are mutually
singular.  So there is a measurable set $A$ with $t$ and $u^t$ are
supported on $A$ and, and $\overline{u^t}$ supported on $X-A$.
Similarly, there is a measurable set $B$ with $u$ and $t^u$ supported
on $B$, and $\overline{t^u}$ supported on $X-B$.  These sets divide
$X$ into four subsets: $A\cap B$, $A-B$, $B-A$, and $X-(B\cup A)$.
The uniqueness of the Lebesgue decomposition implies the decomposition
of the restriction of a measure is given by the restriction of the
decomposition.  On $A\cap B$, $u$ and $t$ restrict to $u^t$ and $t^u$,
which are mutually absolutely continuous.  Hence, on this subset, we
have
\[
      \sqrt{\frac{\extd t^u}{\extd u}} \extd u = \sqrt{\frac{\extd u^t}{\extd t }} \extd t
\]
On the other three subsets of $X$, we have, respectively $u=0$, $t
=0$, and $u=t =0$. In each case, both sides of the previous equation
are zero.
\end{proof}

Given this proposition, we define the {\bf geometric mean} of the
measures $u$ and $t$ to be:
\[
   \sqrt{\extd t \extd u}:=
      \sqrt{\frac{\extd t ^u}{\extd u}} \extd u = \sqrt{\frac{\extd u^t}{\extd t }} \extd t
\]
Outside of this appendix, to reduce notational clutter, we generally
drop the superscripts in Radon--Nikodym derivatives and simply write,
for example:
\[
    \rnd{t}{u} := \rnd{t^u}{u}.
\]

\begin{prop}
\label{null}
Let $t$, $u$ be measures on $X$.  Then a set is
$\sqrt{tu}$-null if and only if it is the union of a $t$-null set and
a $u$-null set.  Equivalently, expressed in terms of almost-everywhere
equivalence, the relation `$\sqrt{tu}$-\alme' is the transitive
closure of the union of the relations `$t$-\alme' and `$u$-\alme'.
\end{prop}

\begin{proof}
First, $\sqrt{tu}\ll t$ and  $\sqrt{tu}\ll u$; indeed $\sqrt{tu}$ is equivalent to both $t^u$ and $u^t$.   So clearly the union of a $t$-null set and a $u$-null set is also $\sqrt{tu}$-null.

Conversely, suppose $D\subseteq X$ has $\sqrt{tu}(D) = 0$.  Then $u(D)
= \bar{u^t}(D)$, and $t(D) = \bar{t^u}(D)$.  But $ \bar{u^t}\perp
\bar{t^u}$, so we can pick a set $P\subseteq X$ on which $\bar{u^t}$
is supported and $\bar{t^u}$ vanishes.  Then $t(D\cap P) = 0$ and $u(D
- P)=0$, so $D$ is the union of a $t$-null set and a $u$-null set.

Expressing this in terms of equivalence relations, suppose $f_1(x) =
f_2(x)$ $\sqrt{tu}$-\alme\ in the variable $x$; we will construct
$g(x)$ such that $g(x) = f_1(x)$ $t$-\alme\ and $g(x) = f_2(x)$
$u$-\alme.  Let $D$ be the set on which $f_1$ and $f_2$ differ, and
let $P$ be the set defined in the previous paragraph.  Set $g(x) :=
f_1(x) = f_2(x)$ on $X-D$, $g(x) := f_1(x)$ on $D - P$, and $g(x) :=
f_2(x)$ on $D \cap P$.  This defines $g$ on all of $x$.  Now $f_1$ and
$g$ differ only on $D\cap P$, which has $t$-measure 0; $f_2$ and $g$
differ only on $D - P$, which has $u$-measure 0.
\end{proof}

Now suppose we have three measures $t$, $u$, and $v$ on the same
space.  How are the geometric means $\gmean{t}{u}$ and $\gmean{t}{v}$
related?  An answer to this question is given by the following lemma,
which is useful for rewriting an integral with respect to one of these
geometric means as an integral with respect the other.

\begin{lemma}
\label{chain}
Let $t$, $u$, and $v$ be measures on $X$.  Then we
have an equality of measures
\[
    \gmean{t}{u} \sqrnd{v^u}{u} =
     \gmean{t}{v} \sqrnd{v^t}{t} \sqrnd{u^v}{v} \sqrnd{t^u}{u}
\]
\end{lemma}

\begin{proof}
Let us first define a measure $\mu$ by the left side of the desired
equality:
\[
  \extd \mu =   \gmean{t}{u} \sqrnd{v^u}{u}
\]
We then have, using the definition of geometric mean measure,
\begin{align*}
\extd \mu &=  \extd t \sqrnd{u^t}{t} \sqrnd{t^u}{u} \\
        &= (\extd t^v + \extd \overline{t^v}) \sqrnd{u^t}{t} \sqrnd{t^u}{u}
\end{align*}
where the latter expression gives the Lebesgue decomposition of $\mu$
with respect to $v$.  However, as we show momentarily, the singular
part of this decomposition is identically zero.  Assuming this result
for the moment, we then have
\begin{align*}
\extd \mu &=  \extd {t^v} \sqrnd{u^t}{t} \sqrnd{t^u}{u} \\
              &= \extd v \rnd{t^v}{v} \sqrnd{u^t}{t} \sqrnd{t^u}{u} \\
              &=\gmean{t}{v} \sqrnd{v^t}{t} \sqrnd{u^v}{v} \sqrnd{t^u}{u}
\end{align*}
as we wished to show.  To complete the proof, we thus need only see
that the $\overline{t^v}$ part of $\mu$ vanishes:
\[
    \extd \overline{t^v} \sqrnd{u^t}{t} \sqrnd{v^u}{u} = 0
\]
That is, we must show that
\[
  \overline{\mu^v}(X)=  \int_X  \extd \overline{t^v} \sqrnd{u^t}{t} \sqrnd{v^u}{u} = 0.
\]
Let $Y\subseteq X$ be a measurable set such that $\overline{t^v}$ is
supported on $Y$, while $v$ and $t^v$ are supported on its complement:
\[
        v =
v|_{X-Y} \qquad t^v = t^v|_{X-Y} \qquad \overline{t^v}= \overline{t^v}|_{Y}
\]
Similarly, let $A\subseteq X$ be such that
\[
  t=t|_A   \qquad u^t= ut|_A \qquad \overline{u^t} = \overline{u^t}|_{X-A}
\]
Note that
\[
      \rnd{u^t}{t}
\]
vanishes $t$--\alme, and hence $\overline{t^v}$--\alme\, on $X-A$.
Thus the measure
\[
  \extd \overline{t^v} \sqrnd{u^t}{t}
\]
is zero on $X-A$.  Since we also have $\overline{t^v}$ vanishing on
$X-Y$, we have
\[
\overline{\mu^v}(X)  =
   \int_{Y\cap A} \extd \overline{t^v} \sqrnd{u^t}{t} \sqrnd{v^u}{u}.
\]
Now by construction of $Y$, we have $v(Y\cap A) = 0$, and hence $v^u(Y
\cap A)$ = 0.  So
\[
   \rnd{v^u}{u}
\]
vanishes $u$--\alme, and hence $u^t$--\alme, on $Y \cap A$.  If
$C\subseteq Y\cap A$ is the set of points where the latter
Radon--Nikodym derivative does not vanish, then $u^t(C)=0$ implies
that
\[
    \sqrnd{u^t}{t}
\]
vanishes $t$--\alme, hence $\overline{t^v}$--\alme\ on C.  Thus
\[
\overline{\mu^v}(X) =
      \int_C  \extd \overline{t^v} \sqrnd{u^t}{t} \sqrnd{v^u}{u} = 0,
\]
so $\mu$ is absolutely continuous with respect to $v$.
\end{proof}

\begin{prop} Let $t, u$ be measures on $X$, and consider
the Lebesgue decompositions
$t = t^u + \overline{t^u}$ and $u = u^t + \overline{u^t}$. Then:
$$\rnd{u^t}{t}\rnd{t^u}{u} = 1 \qquad \sqrt{tu}-\alme$$
\end{prop}
\begin{proof}
Applying Lemma \ref{chain} with $v=u$ we get
\[
    \gmean{t}{u} =
     \gmean{t}{u} \sqrnd{u^t}{t}  \sqrnd{t^u}{u}
\]
Thus the function $\rnd{u^t}{t} \rnd{t^u}{u}$ differs from 1 at most
on a set of $\sqrt{tu}$-measure zero.
\end{proof}

\subsection{Measurable groups}
\label{apx:measurable_groups}

Given a measurable group $H$, it is natural to ask whether $H^*$
is again a measurable group.  The main goal of this section is to
present necessary and sufficient conditions for this to be so.
These conditions are due to Yves de Cornulier and Todd Trimble.
We also show that when $H$ and $H^*$ are measurable, a continuous
action of a measurable group $G$ on $H$ gives a continuous action of
$G$ on $H^*$.

Recall that for us, a {\bf measurable group} is a locally compact Hausdorff
second countable topological group.  Any measurable group
becomes a measurable space with its $\sigma$-algebra of Borel subsets.
The multiplication and inverse maps for the group are then measurable.
However, not every measurable space that is a group with measurable
multiplication and inverse maps can be promoted to a measurable group
in our sense!  There may be no second countable locally compact Hausdorff
topology making these maps continuous.  Luckily, all the counterexamples
are fairly exotic \cite[Sec.\ 1.6]{BeckerKechris}.

\begin{lemma} \label{lem:automatic_continuity}
A measurable homomorphism between measurable groups is continuous.
\end{lemma}

\begin{proof}
Various proofs can be found in the literature.  For
example, Kleppner showed that a measurable homomorphism between
locally compact groups is automatically continuous \cite{Kleppner}.
\end{proof}

Given a measurable group $H$, we let $H^*$ be the set of measurable
--- or equivalently, by Lemma \ref{lem:automatic_continuity}, continuous ---
homomorphisms from $H$ to $\C^\times$.  We make $H^*$ into a
topological space with the compact-open topology.  $H^*$ then
becomes a topological group under pointwise multiplication.

The first step in analyzing $H^*$ is noting that every continuous
homomorphism $\chi \maps H \to \C^\times$ is trivial on the commutator
subgroup $[H,H]$ and thus also on its closure $\overline{[H,H]}$.
This lets us reduce the problem from $H$ to
\[      \Ab(H) = H/\overline{[H,H]},  \]
which becomes a topological group with the quotient topology.
Let $\pi \maps H \to \Ab(H)$ be the quotient map.  Then we have:

\begin{lemma}  Suppose $H$ is a measurable group.  Then $\Ab(H)$
is a measurable group.   $\Ab(H)^*$ is a measurable group
if and only if $H^*$ is, and in this case the map
\[       \begin{array}{rccl}
\pi^* \maps & \Ab(H)^* & \to & H^*  \\
          & \chi    & \mapsto & \chi \pi
\end{array}
\]
is an isomorphism of measurable groups.
\end{lemma}

\begin{proof} Suppose $H$ is a measurable group: that is,
a second countable locally compact Hausdorff group.
By Lemma \ref{lem:quotient_space}, the quotient
$\Ab(H)$ is a second countable locally compact Hausdorff
space because the subgroup $\overline{[H,H]}$ is closed.
So, $\Ab(H)$ is a measurable group.

The map $\pi^*$ is a bijection because every continuous homomorphism
$\phi \maps H \to \C^\times$ equals the identity on $\overline{[H,H]}$
and thus can be written as $\chi \pi$ for a unique continuous
homomorphism $\phi \maps \Ab(H) \to \C^\times$.  We can also see that
$\pi^*$ is continuous.  Suppose a net $\chi_\alpha \in \Ab(H)^*$
converges uniformly to $\chi \in \Ab(H)^*$ on compact subsets of
$\Ab(H)$.  Then if $K \subseteq H$ is compact, $\chi_\alpha \pi$
converges uniformly to $\chi \pi$ on $K$ because $\chi_a$ converges
uniformly to $\chi$ on the compact set $\pi (K)$.

It follows that $\pi^* \maps \Ab(H)^* \to H^*$ is a continuous bijection
between second countable locally compact Hausdorff spaces.  This
induces a measurable bijection between measurable spaces.  Such a map
always has a measurable inverse \cite[Chap.\ I, Cor.\ 3.3]{Parthasarathy}.
(This reference describes measurable spaces in terms of separable metric
spaces, but we have seen in Lemma \ref{lem:standard_Borel} that this
characterization is equivalent to the one we are using here.)
So, $\pi^*$ is an isomorphism of measurable spaces.  Since it is a
group homomorphism, it is also an isomorphism of measurable groups.
\end{proof}

Thanks to the above result, we henceforth assume $H$ is an abelian
measurable group.  Since
\[         \C^\times \cong \U(1) \times \R \]
as topological groups, we have
\[      H^* \cong \hom(H,\U(1)) \times \hom(H,\R)  \]
as topological groups, where $\hom$ denotes the space of
continuous homomorphisms equipped with its compact-open topology
and made into a topological group using pointwise multiplication.
The topological group
\[        \hat H = \hom(H,\U(1)) \]
is the subject of Pontrjagin duality so this part of $H^*$ is
well-understood \cite{Armacost,Morris,Pontrjagin}.  In particular:

\begin{lemma} If $H$ is an abelian measurable group, so is its
Pontrjagin dual $\hat{H}$.
\end{lemma}

\begin{proof} It is well-known that whenever $H$ is an abelian locally
compact Hausdorff group, so is $\hat{H}$ \cite[Thm.\ 10]{Morris}.  So, let
us assume in addition that $H$ is second countable, and show the same for
$\hat{H}$.

For this, first note by Lemma \ref{lem:standard_Borel} that $H$ is
metrizable.  A locally compact second-countable space is clearly
$\sigma$-compact, so $H$ is also $\sigma$-compact.  Second, note that
a locally compact Hausdorff abelian group $H$ is metrizable if and
only $\hat{H}$ is $\sigma$-compact \cite[Thm.\ 29]{Morris}.

It follows that $\hat{H}$ is also $\sigma$-compact and metrizable.
Since a compact metric space is second countable (for each
$n$ it admits a finite covering by balls of radius $1/n$), so is a
$\sigma$-compact metric space.  It follows that $\hat{H}$ is second
countable.
\end{proof}

The issue thus boils down to: if $H$ is an abelian measurable
group, is $\hom(H,\R)$ also measurable?  Sadly, the answer is ``no".
Suppose $H$ is the free abelian group on countably many generators.
Then $\hom(H,\R)$ is a countable product of copies of $\R$, with its
product topology.  This space is not locally compact.

Luckily, there is a sense in which this counterexample is the
only problem:

\begin{lemma}  Suppose that $H$ is an abelian measurable group.
Then $\hom(H,\R)$ is measurable if and only if the free abelian
group on countably many generators is not a discrete subgroup of
$H$.
\end{lemma}

\begin{proof}
First suppose $H$ is an abelian locally compact Hausdorff group.
Then $H$ has a compact subgroup $K$ such that $H/K$ is a Lie group,
perhaps with infinitely many connected components \cite[Cor.\ 7.54]{HM}.
Since any connected abelian Lie group is the product of $\R^n$ and a
torus, we can enlarge $K$ while keeping it compact to ensure that the
identity component of $H/K$ is $\R^n$.

Any continuous homomorphism from a compact group to $\R$ must have
compact range, and thus be trivial.  It follows that $K$
lies in the kernel of any $\chi \in \hom(H,\R)$, so
\[   \hom(H,\R) \cong \hom(H/K,\R).  \]
So, without loss of generality we can replace $H$ by $H/K$.  In
other words, we may assume that $H$ is an abelian Lie group with
$\R^n$ as its identity component.  The only subtlety is that $H$
may have infinitely many components.

Since $\R^n$ is a divisible abelian group, the inclusion
$j \maps \R^n \to H$ comes with a homomorphism $p \maps H \to \R^n$
with $p j = 1$, so we actually have $H \cong \R^n \times A$
as abstract groups, where $A$ is the range of $p$.   Since
$A \cap \R^n$ is trivial, $A$ is actually a discrete subgroup of
$H$.  So, as a topological group $H$ must be the product of $\R^n$
and a discrete abelian group $A$.  It follows that
\[      \hom(H,\R) \cong \R^n \times \hom(A,\R)  ,\]
so without loss of generality we may replace $H$ by the discrete
abelian group $A$, and ask if $\hom(A,\R)$ is measurable.

Since homomorphisms $\chi \maps A \to \R$ vanish on the torsion
of $A$, we may assume $A$ is torsion-free.  There are two alternatives
now:

\begin{enumerate}
\item $A$ has finite rank: i.e., it is a subgroup of the discrete
group $\Q^k$ for some finite $k$.  If we choose the smallest such $k$, then
$A$ contains a subgroup isomorphic to $\Z^k$ such that the natural
restriction map
\[    \hom(A,\R) \to \hom(\Z^k,\R) \]
is an isomorphism (actually of topological groups).
Since $\hom(\Z^k,\R)$ is locally compact, Hausdorff, and second countable,
so is $\hom(A,\R)$.  So, in this case our original topological
group $\hom(H,\R)$ is measurable.

\item $A$ has infinite rank.  This happens precisely when our original
group $H$ contains the free abelian group on a countable infinite set
of generators as a discrete subgroup.  In this case we can show that
$\hom(A,\R)$ and thus our original topological group $\hom(H,\R)$
is not locally compact.

To see this, let $U$ be any neighborhood of $0$ in $\hom(A,\R)$.
By the definition of the compact-open topology, there is a compact
(and thus finite) subset $K \subseteq A$ and a number $r > 0$ such that
$U$ contains the set $V$ consisting of $\chi \in \hom(A,\R)$ with
$|\chi(a)| \le r$ for all $a \in K$.   It suffices to show that $V$
is not relatively compact.

To do this, we shall find a sequence $\chi_n \in V$ with no cluster
point.  Since $A$ has infinite rank, we can find $a \in A$ such that
the subgroup generated by $a$ has trivial intersection with the finite
set $K$.  For each $n \in \N$, there is a unique homomorphism $\phi_n$
from the subgroup generated by $a$ and $K$ to $\R$ with $\phi_n(a) = n$
and $\phi_n(K) = 0$.  Since $\R$ is a divisible abelian group, we can
extend $\phi_n$ to a homomorphism $\chi_n \maps A \to \R$.  Since
$\chi_n$ vanishes on $K$, it lies in $V$.  But since $\chi_n(a) = n$,
there can be no cluster point in the sequence $\chi_n$.
\end{enumerate}
\end{proof}

Combining all these lemmas, we easily conclude:

\begin{theo}
Suppose $H$ is a measurable group.  Then $H^*$ is a measurable group
if and only if the free abelian group on countably many generators is
not a discrete subgroup of $\Ab(H)$.  This is true, for example, if
$H$ has finitely many connected components.
\end{theo}

We also have:

\begin{lemma}
\label{lem:continuous_action}
Let $G$ and $H$ be measurable groups with a left action $\rhd$ of $G$
as automorphisms of $H$ such that the map
\[       \rhd \maps G \times H \to H  \]
is continuous.  Then the right action of $G$ on $H^*$ given by
\[\chi_g[h] = \chi[g\rhd h] \]
is also continuous.
\end{lemma}

\begin{proof}
Recall that $H^*$ has the induced topology coming from
the fact that it is a subset of the space of continuous maps
$C(H,\C^\times)$ with its compact-open topology.  So, it suffices
to show that the following map is is continuous:
\[
\begin{array}{ccc}
        C(H, \C^\times) \times G &\to& C(H, \C^\times)  \\
                  (f, g)         &\mapsto& f_g
\end{array}
  \]
where
\[          f_g[h] = f[g \rhd h]  .\]
This map is the composite of two maps:
\[
\begin{array}{ccccc}
C(H, \C^\times) \times G
& \stackrel{1 \times \alpha}{\longrightarrow}
& C(H, \C^\times) \times C(H, H)
&\stackrel{\circ}{\longrightarrow}
& C(H, \C^\times) \\
(f, g)
& \mapsto
& (f, \alpha(g))
& \mapsto
& f \circ \alpha(g) = f_g .
\end{array}
\]
where
\[        \alpha(g)h = g \rhd h  .\]
The first map in this composite
is continuous because $\alpha$ is: in fact,
any continuous map
\[        \rhd \maps X \times Y \to X  \]
determines a continuous map
\[        \alpha \maps Y \to C(X,X)   \]
by the above formula, as long as $X$ and $Y$ are locally compact Hausdorff
spaces.   The second map
\[  C(H, \C^\times) \times C(H, H) \stackrel{\circ}{\to} C(H, \C^\times) \]
is also continuous, since composition
\[  C(Y,Z) \times C(X,Y) \stackrel{\circ}{\to} C(X, Z) \]
is continuous in the compact-open topology whenever $X,Y$ and
$Z$ are locally compact Hausdorff spaces.
\end{proof}

\subsection{Measurable $G$-spaces}
\label{apx:G-spaces}

Suppose $G$ is a measurable group.  A (right) action of $G$ on a
measurable space $X$ is a {\bf measurable} if the map $(g, x) \mapsto
xg$ of $G \times X$ into $X$ is measurable. A measurable space $X$ on
which $G$ acts measurably is called a {\bf measurable $G$-space}.

In fact, we can always equip a measurable $G$-space with a topology
for which the action of $G$ is continuous:

\begin{lemma}
\label{lem:good_topology}
{\bf \cite[Thm.\ 5.2.1]{BeckerKechris}}
Suppose $G$ is a measurable group and $X$ is a measurable
space with $\sigma$-algebra ${\cal B}$.  Then there is a way to equip
$X$ with a topology such that:
\begin{itemize}
\item
$X$ is a Polish space---i.e., homeomorphic to separable complete metric
space,
\item
${\cal B}$ consists precisely of the Borel sets for this topology, and
\item
the action of $G$ on $X$ is continuous.
\end{itemize}
\end{lemma}

Moreover:

\begin{lemma}
\label{measurable_orbits}
{\bf \cite[Cor.\ 5.8]{Varadarajan}}
Let $G$ be a measurable group and let $X$ be a measurable $G$-space.
Then for every $x \in X$, the orbit $x G = \{x g \,: \,g\in G\}$
is a measurable subset of $X$; moreover the stabilizer $S_x = \{g
\in G\, | \, x g = x \}$ is a closed subgroup of $G$.
\end{lemma}
This result is important for the following reason. Given a point
$x_o \in X$, the measurable map $$g \mapsto x_og$$ from $G$ into
$X$ allows us to measurably identify the orbit $x_oG$ with the
homogeneous space $G/S_{x_o}$ of right cosets $S_{x_o}g$, on which
$G$ acts in the obvious way.  Now, such spaces enjoy some nice
properties, some of which are listed below.

Fix a measurable group $G$ and a closed subgroup $S$ of $G$.

\begin{lemma} {\bf \cite[Thm.\ 7.2]{Mackey1957}}
\label{lem:quotient_space}
The homogeneous space $X = G/S$, equipped
with the quotient topology, is a Polish space.  Since the action of
$G$ on $X$ is continuous, it follows that $X$ becomes a measurable
$G$-space when endowed with its $\sigma$-algebra of Borel sets.
\end{lemma}

Let $\pi \maps G \to G/S$ denote the canonical projection. A
{\bf measurable section} for $G/S$ is a measurable map $s \maps
G/S \to G$ such that $\pi s$ is the identity on $G/S$ and
$s(\pi(1)) = 1$, where $1$ is the identity in $G$.

\begin{lemma} {\bf \cite[Lemma 1.1]{Mackey1952}}
\label{meas.sections}
There exist measurable sections for $G/S$.
\end{lemma}

Next we present a classic result concerning
quasi-invariant measures on homogeneous spaces. Let $X$ be a
measurable $G$-space, and $\mu$ a measure on $X$. For each $g\in
G$, define a new measure $\mu^g$ by setting $\mu^g(A) = \mu(A
g^{-1})$. We say the measure is {\bf invariant} if $\mu^g = \mu$
for each $g \in G$; we say it is {\bf quasi-invariant} if $\mu^g
\sim \mu$ for each $g \in G$.

\begin{lemma} {\bf \cite[Thm.\ 1.1]{Mackey1952}}
Let $G$ be a measurable group and $S$ a closed subgroup of $G$.
Then there exist non-trivial quasi-invariant measures on the
homogeneous space $G/S$. Moreover, such measures are all
equivalent.
\end{lemma}

\end{document}